\documentclass[12pt, dvipsnames]{amsart}
\usepackage[T1]{fontenc}

\usepackage[shortlabels, inline]{enumitem}
\usepackage{xcolor}
\usepackage[framemethod=TiKZ]{mdframed}
\usepackage{textcomp}
\usepackage{eurosym}
\usepackage[all]{xy}
\usepackage{soul}
\usepackage{amsmath,mathtools}
\usepackage{amscd}
\usepackage{amsmath,amsthm,latexsym,amssymb,accents,mathtools,amscd}
\usepackage{txfonts,mathrsfs}

\usepackage[colorlinks=true, citecolor=BrickRed, urlcolor=RoyalPurple]{hyperref}
\usepackage{graphicx}
\usepackage[utf8]{inputenc}
\usepackage[english]{babel}
\usepackage[a4paper, portrait, margin=2cm]{geometry}
\usepackage{wrapfig}
\usepackage{tikz-cd}
\usepackage{tikz}
\usetikzlibrary{shapes.misc}
\usepackage[normalem]{ulem}

\usepackage{multicol}
\newcommand{\A}{\mathcal{A}}
\newcommand{\C}{\mathcal{C}}
\newcommand{\D}{\mathcal{D}}
\newcommand{\T}{\mathcal{T}}
\newcommand{\F}{\mathcal{F}}
\newcommand{\G}{\mathcal{G}}

\newcommand{\K}{\mathcal{K}}

\renewcommand{\P}{\mathcal{P}}
\newcommand{\Q}{\mathcal{Q}}
\newcommand{\B}{\mathcal{B}}
\newcommand{\M}{\mathcal{M}}
\newcommand{\I}{\mathcal{I}}

\newcommand{\HH}{\mathcal{H}}

\newcommand{\PP}{\mathbb{P}}

\newcommand{\Z}{\mathbb{Z}}

\newcommand{\E}{\mathcal{E}}

\newcommand{\Mor}{\operatorname{Mor}}

\newcommand{\Ima}{\operatorname{Im}}

\newcommand{\red}{_\mathrm{red}}

\newcommand{\defeq}{\overset{\text{\textup{def}}}{=}}

\newcommand{\hadgesh}[1]{\textcolor{Brown}{\emph{#1}}}

\definecolor{ForestGreen}{rgb}{0.0, 0.27, 0.13}
\definecolor{DukeBlue}{HTML}{001A57}
\definecolor{myblue}{rgb}{0.29,0.47,1}

\renewcommand{\mod}{\textup{-}\curs{mod}}

\newcommand{\Hom}{\mathrm{Hom}}

\newcommand{\Id}{\mathrm{Id}}

\renewcommand{\lim}{\mathop{\textup{lim}}\limits}

\newcommand{\higherlim}[2]{\displaystyle\setbox1=\hbox{\rm lim}
\setbox2=\hbox to \wd1{\leftarrowfill} \ht2=0pt \dp2=-1pt
\setbox3=\hbox{$\scriptstyle{#1}$}
\def\test{#1}\ifx\test\empty
\mathop{\mathop{\vtop{\baselineskip=5pt\box1\box2}}}\nolimits^{#2}
\else
\ifdim\wd1<\wd3
\mathop{\hphantom{^{#2}}\vtop{\baselineskip=5pt\box1\box2}^{#2}}_{#1}
\else
\mathop{\mathop{\vtop{\baselineskip=5pt\box1\box2}}_{#1}}%
\nolimits^{#2}
\fi\fi}

\newcommand{\rk}{\operatorname{rk}\nolimits}

\newcommand{\Ker}{\operatorname{Ker}\nolimits}
\newcommand{\coKer}{\operatorname{coKer}\nolimits}
\newcommand{\VVect}{{\mathrm{Vect}_k}}

\newcommand{\xto}[1]{\xrightarrow{#1}}

\DeclareMathAlphabet\EuR{U}{eur}{m}{n}
\SetMathAlphabet\EuR{bold}{U}{eur}{b}{n}

\newcommand{\curs}{\EuR}

\renewcommand{\mod}{\mbox{-}\curs{mod}}
\newcommand{\Fun}{\curs{Fun}}
\newcommand{\Gr}{\curs{Gr}}
\newcommand{\GG}{\curs{G}}
\newcommand{\EE}{\curs{E}}

\newcommand{\obj}{\mathrm{Obj}}

\newcommand{\Cat}{\curs{Cat}}
\newcommand{\SCat}{\curs{SCat}}
\newcommand{\sSCat}{\curs{sSCat}}
\newcommand{\Aut}{\mathrm{Aut}}
\newcommand{\supp}{\mathrm{Supp}}

\newcommand{\widebar}[1]{\overset{{\mskip1mu\leaders\hrule height0.4pt\hfill\mskip1mu}}{#1}\vphantom{#1}}
\newcommand{\bul}{_\bullet}
\newcommand{\RH}{\mathscr{R}}

\usepackage{amsthm}
\usepackage{amssymb,amsmath,extarrows}
\newtheorem{Thm}{Theorem}[section]
\newtheorem{Prop}[Thm]{Proposition}

\newtheorem{Lem}[Thm]{Lemma}
\newtheorem{Rem}[Thm]{Remark}
\newtheorem{Cor}[Thm]{Corollary}
\newtheorem{Not}[Thm]{Notation}
\newtheorem{Th}{Theorem}

\theoremstyle{definition}
\newtheorem{Defi}[Thm]{Definition}
\newtheorem{Ex}[Thm]{Example}

\title{Representation cohomology of a small category}

\author{Markus Klemetti}
\address{Institute of Mathematics, University of Aberdeen, Aberdeen, UK}
\email{markus.o.klemetti@gmail.com}
\thanks{Klemetti is partially supported by Vilho, Yrj\"o and Kalle V\"ais\"al\"a Foundation}
\author{Ran Levi}
\address{Institute of Mathematics, University of Aberdeen, Aberdeen, UK}
\email{r.levi@abdn.ac.uk}
\thanks{Levi is partially supported by EPSRC grant EP/Y028872/1}
\author{Henri Riihim\"aki}
\address{Nordita, Stockholm University, Stockholm, Sweden}
\email{henri.riihimaki@su.se}
\thanks{Riihim\"aki is  supported by the Wallenberg Initiative on Networks and Quantum Information (WINQ)}
\author{Daniel S\"olch}
\address{Institute of Mathematics, University of Aberdeen, Aberdeen, UK}
\email{daniel.solch@abdn.ac.uk}
\date{\today}

\begin{document}

\begin{abstract}
Let $\C_\bullet$ be a  simplicial object in the category $\Cat$ of small categories. For a field $k$, taking the  Grothendieck groups  of isomorphism classes of $k\C_n$-modules gives rise to a cochain complex, whose cohomology, which we refer to as representation cohomology, is the object studied in this article. In particular, to any small category $\C$, we associate a simplicial object in $\Cat$, where for each $n\geq 0$ the objects of the level $n$ category are the simplices of the nerve of $\C$.  The basic properties of the  resulting representation cohomology of these simplicial objects and certain subobjects are then studied in detail.  We present some general  theoretical computations in favourable cases. 
\end{abstract}

\maketitle

Let $\C$ be a small category and let $k$ be a field. A representation of $\C$ over $k$ is a functor $M$ from $\C$ to the category of finite dimensional vector spaces over $k$.  If $\C$ is a finite category, then a representation of $\C$ over $k$ can equivalently be thought of as a $k\C$-module, where $k\C$ is the category algebra of $\C$ over $k$ \cite{Mi, Xu}. The collection of all representations of a category $\C$ over $k$ forms a category $k\C\mod$, where the morphisms are natural transformations of functors. Representation theory of small categories is a topic of substantial interest in pure and applied algebra \cite{DW, Sc, We}. On the applied side, notable examples arise in topological data analysis \cite{BBO, EH}.

By Drozd's trichotomy theorem \cite{Dr}, small categories are divided into three  representation types - finite, tame and wild. The division refers to the classification of indecomposable $k\C$-modules. In the finite case, one has only finitely many indecomposable modules, while in other two types there are  infinitely many. Tame representation type means that there is a $1$-parameter family of indecomposable $k\C$-modules, and for a category of wild representation type  there is an $n$-parameter family of indecomposable modules for arbitrarily large values of $n$.   It is well known that in almost all small categories, including finite categories and even finite posets,  the representation type is wild. Hence, since classification of indecomposable modules in general is completely out of reach, other methods for studying representations are required. This article offers a new tool which we refer to as \hadgesh{representation cohomology}. 

Representation cohomology is defined for any simplicial or semi-simplicial  object in the category $\Cat$  of small categories and functors between them. If $\D_\bullet$ is such an object and $k$ a field, one associates the (split) Grothendieck group $\Gr(k\D_n)$ of isomorphism classes of $k\D_n$-modules with each category $\D_n$. The simplicial operators then turn $\Gr(k\D_\bullet)$ into a cosimplicial free abelian group generated in  dimension $n$  by the indecomposable $k\D_n$-modules. Representation cohomology is defined to be the cohomology of the corresponding cochain complex. In this article we start with a small category $\C$ to which we associate a family of simplicial and semi-simplicial objects in $\Cat$.  The reader who is  not familiar with  the basic theory of simplicial and semi-simplicial objects   is referred to \cite{McL} for the necessary background that we are going to assume.

Let $\Cat_+\subseteq\Cat$ denote the subcategory whose objects are  small categories that admit no non-identity invertible morphisms and whose morphisms are faithful functors (i.e., injective on morphism sets). Let $\SCat$ and $\sSCat$ denote the categories of simplicial objects and semi-simplicial objects in $\Cat$, respectively. For any small category $\C$, we define a simplicial object $\widetilde{\E}^\C_\bullet\in\SCat$, where for $n\ge 0$, the objects of the category $\widetilde{\E}^\C_n$ are the $n$-simplices of the nerve $|\C|$ of $\C$. We shall be interested primarily in certain subobjects and their associated cohomology theory. Our first theorem summarises the basic properties of these constructions. 

\begin{Th}\label{Th:Rep-Cats}
There are functors $\widetilde{\E}^{-}_\bullet, \widetilde{\G}^{-}_\bullet\colon\Cat\to\SCat$ and $\E^{-}_\bullet, \G^{-}_\bullet\colon\Cat_+\to\sSCat$, together with natural transformations $\widetilde{\iota}\colon \widetilde{\G}^{-}_\bullet\to\widetilde{\E}^{-}_\bullet$ and $\iota\colon \G^{-}_\bullet\to\E^{-}_\bullet$. Restricted to $\Cat_+$, there are also natural transformations $\epsilon\colon\E^{-}_\bullet\to\widetilde{\E}^{-}_\bullet$ and $\gamma\colon\G^{-}_\bullet\to \widetilde{\G}^{-}_\bullet$, such that $\epsilon\circ\iota = \widetilde{\iota}\circ\gamma$. These functors have the following basic properties.
\begin{enumerate}[(1)]
\item $\widetilde{\E}^\C_0 = \widetilde{\G}^\C_0 = \C$ for all $\C\in\Cat$ and $\E^\D_0=\G^\D_0 = \D$ for all $\D\in\Cat_+$. Furthermore, for all $n>0$ and any $\C\in\Cat$, $\obj(\widetilde{\E}^\C_n)=\obj(\widetilde{\G}^\C_n)$ and the  functor $\widetilde{\iota}_\C\colon\widetilde{\G}^\C_n\to\widetilde{\E}^\C_n$ is an inclusion. Similarly, $\obj(\E^\D_n)=\obj(\G^\D_n)$ for any $\D\in\Cat_+$ and $n>0$, and the  functor $\iota_\D\colon\G^\D_n\to\E^\D_n$ is an inclusion.  \label{Th:Rep-Cats-1}

\item  For any $\D\in\Cat_+$ and $n\geq 0$,   the objects of $\E^\D_n$ are the non-degenerate simplices of $|\D|_n$, and the functor $\epsilon_\D\colon\E^\D_n\to\widetilde{\E}^\D_n$ is an inclusion of a full subcategory. Similarly,    the objects of $\G^\D_n$ are the non-degenerate simplices of $|\D|_n$, and the functor $\gamma_\D\colon\G^\D_n\to\widetilde{\G}^\D_n$ is an inclusion of a full subcategory. \label{Th:Rep-Cats-2}

\item  Let $\D_\bullet\subseteq \widetilde{\E}^\C\bul$  be any  semi-simplicial subobject. Then $\Gr(k\D_\bullet)$ is a semi-cosimplicial  abelian group generated in dimension $n$ by the indecomposable $k\D_n$-modules. Furthermore, the resulting cochain complex has the structure of a unital differential graded associative algebra over $\Z$. As a result the associated cohomology has the structure of a graded associative unital ring. \label{Th:Rep-Cats-3}
\end{enumerate}
\end{Th}

The  cohomology theories associated to the (semi-)simplicial objects  in Theorem \ref{Th:Rep-Cats}\ref{Th:Rep-Cats-1} are  denoted by $\widetilde{\EE}^*(-,k)$, $\widetilde{\GG}^*(-,k)$, $\EE(-,k)$ and $\GG^*(-,k)$, respectively. We note that the multiplicative structure of Theorem \ref{Th:Rep-Cats}\ref{Th:Rep-Cats-3} need  be neither commutative nor graded commutative.   Parts \ref{Th:Rep-Cats-1} and \ref{Th:Rep-Cats-2} of the theorem are proven in Section \ref{Sec:Prelims}. Part \ref{Th:Rep-Cats-3} is stated in more detail and proved as Theorem \ref{Thm:Products}.

Next, we observe  that the information  on $\C$ contained in the graded group $\widetilde{\EE}^*(\C,k)$ is concentrated  in dimension $0$. 

\begin{Th}\label{Th:E0G0}
The cohomology theory $\widetilde{\EE}^*(-, k)$ is acyclic, and  for any small category $\C$
\[\widetilde{\EE}^0(\C, k) \cong \Gr(k\Gamma),\]
where $\Gamma = \pi_1(|\C|)$.  Furthermore,  the  map $\widetilde{\iota}^*\colon \widetilde{\EE}^0(\C,k)\to\widetilde{\GG}^0(\C,k)$ is an isomorphism.
\end{Th}

The theorem is proved as Theorem \ref{Thm:E0G0}. If one considers the first coboundary operator as a gradient for $k\C$-modules, then Theorem \ref{Th:E0G0} is essentially the statement that virtual modules of gradient $0$ are locally constant, i.e., either a genuine $k\C$-module that takes every morphism in $\C$ to an isomorphism of vector spaces or a difference of two such modules (compare to \cite[Theorem A]{BLR}). See Section \ref{Sec:Cohomology-EG} for details.

Recall that a category $\C$ is said to be \hadgesh{direct} if it contains no infinite chains of non-identity morphisms and no cycles of positive length. While typically $\widetilde{\GG}^*(\C, k)$ does not vanish  in positive dimensions, for direct categories it coincides with $H^*(|\C|,\Z)$ above dimension $1$, as our next theorem states.  

\begin{Th}\label{Th:Higher-G}
Let $\C$ be a  direct category. Then  the map
$\widetilde{\GG}^n(\C, k)\to H^n(|\C|, \Z)$
 induced by natural inclusion of simplicial objects  $|\C|_\bullet\to \widetilde{\G}^\C_\bullet$ is an epimorphism for $n=1$ and an isomorphism for $n>1$. Furthermore, $\widetilde{\GG}^1(\C, k)$ is the kernel of the first coboundary operator restricted to a certain subgroup of $\Gr(k\widetilde{\G}^\C_1)$ that can be described explicitly.
 \end{Th}
 
 The theorem is stated in more detail and proved as Theorem \ref{Thm:Higher-G}.
 
 Next we turn to studying the theories $\EE^*$ and $\GG^*$, which turn out to be much more interesting and a lot harder to compute. In both cases we consider first the $0$-dimensional cohomology. A collection of morphisms in $\C$ is said to be a \hadgesh{line-component} if the full subcategory of $\G^\C_1$ having them as objects is a connected component. Similarly, a collection of morphisms is said to be a \hadgesh{square-component} if the full subcategory of $\E^\C_1$ that has them as objects is a connected component. The following theorem examines the implications for a module $M\in k\C\mod$ whose isomorphism class is a $0$-cocycle. 
 
 \begin{Th}\label{Th:E0G0-check}
 Let $\C\in\Cat_+$  and let $M\in k\C\mod$ be a module. Then the following statements hold.
 \begin{enumerate}[(1)]
 \item If $[M]$ represents a class in $\EE^0(\C, k)$, then $M$ is locally constant on each square-component of $\C$ that contains at least two different non-identity morphisms. \label{Th:E0G0-check-1}
 
\item If $X = [M]-[N]$, where $M, N\in k\C\mod$, represents a class in $\EE^0(\C, k)$,  then $\rk X$ is constant on each square-component of $\C$ with at least two non-identity morphisms.\label{Th:E0G0-check-2}

\item If $X = [M]-[N]$, where $M, N\in k\C\mod$, represents a class in $\GG^0(\C, k)$,  then $\rk X$ is constant on each line-component of $\C$.\label{Th:E0G0-check-3}
\end{enumerate}
\end{Th}

 For our definition of the rank invariant $\rk X$, see Definition \ref{Def:Rank-Inv}. Parts \ref{Th:E0G0-check-1} and \ref{Th:E0G0-check-2} of Theorem \ref{Th:E0G0-check} are proven as Proposition \ref{Prop:E-check-0-cocycles} and Part \ref{Th:E0G0-check-3} as Proposition \ref{Prop:G-check-0-cocycles-rk}.  Theorem \ref{Th:E0G0-check} should be compared with \cite[Theorem B]{BLR}.
 
Next we consider $\GG^0$ cohomology in the case where the category in question is a finite poset that is generated by a rooted tree. For any small category $\C$, let $\underline{k}$ denote the constant $k\C$-module with value $k$ on each object and the identity homomorphism for each morphism. Let $S_\C$ denote the $k\C$-module with the value $k$ on each object and the zero homomorphism for each morphism. By \hadgesh{composition length} of a poset we mean the length of the longest chain of composable irreducible morphisms in it.
 
\begin{Th}\label{Th:0-cocycle-trees}
Let $\P$ be a finite poset of composition length at least $3$ whose Hasse diagram $\HH_\P$ is a rooted tree. Assume that the  root  of $\HH_\P$ has a unique successor. Let $M\in k\P\mod$ be a module such that $[M] \in \GG^0(\P, k)$. Then there is an isomorphism
\[M \cong \bigoplus_m \underline{k}\oplus\bigoplus_n S_\P,\] 
for some non-negative integers $m$ and $n$.
\end{Th}

Theorem \ref{Th:0-cocycle-trees} is restated and proved as Theorem \ref{Thm:0-cocycle-trees}.

For any $\C\in\Cat_+$, there is a natural reduction map $\GG^*(\C,k)\to H^*(|\C|, \Z)$ induced by the obvious inclusion of simplicial objects $|\C|_\bullet\to\G^\C_\bullet$. This gives rise to a connection between representation cohomology and reachability cohomology of quivers.  Let $X$ be a quiver and let $\T_X$ denote the  preorder associated to $X$. The cohomology of the nerve $H^*(|\T_X|, R)$ is defined to be the \hadgesh{reachability cohomology of $X$ with coefficients in $R$} \cite{CR}. In fact,  any  category $\C$ has an associated preorder $\T_\C$, where $x\le y$ if and only if $\C(x,y)\neq\emptyset$. For $\C\in\Cat_+$, such preorder is a poset.  Thus one can define reachability cohomology of $\C$ as above and obtain a natural map
\[\G^*(\T_\C, k)\to H^*(|\T_\C|, \Z).\]
This connection will be explored in a subsequent study.

Computations of $\EE^*$ and $\GG^*$ cohomology are in general very hard due to the categories involved having infinite (and generally wild) representation type. Our final theorem concentrates some computational results. 

\begin{Th}\label{Th:Computations}
For $n\geq 1$, let $[n]$ denote the poset with objects $0,1,\ldots, n$ with the natural order relation. For $n\ge 2$, let $\D_n$ denote the poset with objects $0, 1,\ldots, n+1$, and relations $0<1<j$ for all $j\geq 2$. Let $\Q$ be a finite quiver such that the underlying undirected graph is a cycle with at least two edges.  Then
\begin{enumerate}[(1)]
\item $\EE^0([n],k)\cong \Z^2$, $\GG^0([n], k) \cong \Z^3$ and the natural map $\EE^0([n], k)\to \GG^0([n],k)$ is split injective. \label{Th:Computations-1}

\item $\GG^i(\D_n, k)=0$ for all $i>0$, and $\GG^0(\D_n,k)$ is a free abelian group on a set of generators that is in bijective correspondence with the collection of all indecomposable modules $M\in k\D_n\mod$ such that $M(0)\neq 0$. In particular, $\GG^0(\D_n,k)$ is finitely generated if and only if $n\le 2$. \label{Th:Computations-2}

\item Let $\PP(\Q)$ denote the \emph{path category} of $\Q$. Then $\widetilde{\EE}^0(\PP(\Q), k)$ is a free abelian group on generators indexed by pairs $(p(t), m)$, where $p(t)\in k[t]$ is an irreducible polynomial of positive degree with non-zero constant term and $m\geq 1$ is an integer. \label{Th:Computations-3}
\end{enumerate}
\end{Th}

Parts \ref{Th:Computations-1}  and \ref{Th:Computations-3} of Theorem \ref{Th:Computations} are Propositions \ref{Prop:hG0E0-[n]} and \ref{Prop:Cyclic-E0}, respectively. Part \ref{Th:Computations-2} is proved in Propositions \ref{Prop:dandelion-1} and \ref{Prop:dandelion-2}.

To get an idea of what higher-dimensional $\EE^*$ and $\GG^*$ groups look like, we consider a small subcomplex of the defining cochain complex of both theories. Specifically, in each case and every $n\geq 0$ we consider the subgroup of the respective Grothendieck group generated by indecomposable modules where $\dim_k(M(x))\le 1$ for all objects $x$ in the category in question. This is in fact the first stage of a filtration one can define on those cochain complexes, but when the dimension of any indecomposable module at each object is bounded by $1$, there are at most finitely many isomorphism classes of indecomposable modules, which makes the calculation more feasible. We wrote a simple SAGE code that allows such computations for finite posets and used it to obtain some sample calculations. The results are displayed in Section \ref{Sec:Computational}.

The authors are grateful to Ehud Meir for several illuminating discussions on representation theory of quivers and small categories. 

\tableofcontents

\section{Preliminaries and Examples}
\label{Sec:Prelims}
In this section we give the four variations of the cohomology theories discussed in this article. We start by defining, for any small category $\C$, a simplicial object $\widetilde{\E}^\C_\bullet$ in the category $\Cat$ of small categories. This simplicial object can be thought of as an enveloping object, as all other variations are simplicial subobjects of $\widetilde{\E}^\C_\bullet$. We then use these constructions to define the respective representation cohomology theories. 

\subsection{The simplicial envelope of a small category}
\label{SSec:Prelims-Setup}
Let $n\in \mathbb N$. We denote by $[n]$ the poset, considered as a category,  with object set  $\{0,1,\ldots,n\}$ and the natural order as morphisms.  Let $\C$ be any small category. 
Let 
\[\widetilde{\E}^\C_n \defeq \C^{[n]},\]
namely the \hadgesh{functor category from $[n]$ to $\C$}, with morphisms given by natural transformations, as usual. Let $\Delta$ denote the standard simplex category with objects $[n]$, $n\geq 0$, and order-preserving functions as morphisms. Define a simplicial object in $\Cat$ \[\widetilde{\E}^\C_\bullet\colon\Delta^{\mathrm{op}}\to \Cat\] 
by $\widetilde{\E}^\C_\bullet([n]) \defeq \widetilde{\E}^\C_n$, and for a morphism $\alpha\colon [n] \to [m]$, let $\widetilde{\E}^\C_\bullet(\alpha)$ be the induced functor 
\[\widetilde{\E}^\C_m = \C^{[m]} \xto{\alpha^*} \C^{[n]} = \widetilde{\E}^\C_n.\]

More explicitly, let $|\C|$ denote the nerve of $\C$. Then  one can regard the objects of $\widetilde{\E}^\C_n$ as the $n$-simplices of $|\C|$, namely sequences
\begin{equation}\label{Eq:Object}
\underline{x}=x_0\stackrel{u_1}{\rightarrow}x_1\stackrel{u_2}{\rightarrow}\cdots\stackrel{u_n}{\rightarrow}x_n
\end{equation}
of objects and morphisms in $\C$. If
\[
\underline{y}=y_0\stackrel{v_1}{\rightarrow}y_1\stackrel{v_2}{\rightarrow}\cdots\stackrel{v_n}{\rightarrow}y_n
\]
is another such sequence, then   a morphism  $\underline{f}\colon \underline{x}\to \underline{y}$ is  a commutative diagram in $\C$
\[
\xymatrix{
x_0\ar[d]_{f_0}\ar[r]^{u_1}&  x_1 \ar[d]^{f_{1}}\ar[r]^{u_2}& \cdots  \ar[r]^{u_n} & x_n \ar[d]^{f_{n}} \\
y_0\ar[r]_{v_1}&  y_1 \ar[r]_{v_2}& \cdots  \ar[r]_{v_n} &y_n. }
\] 
Objects in $\widetilde{\E}^\C_n$ will be denoted by $\underline{x}$ to mean a sequence as in (\ref{Eq:Object}) above, if keeping track of the objects involved is required, or by the symbol $[u_1|u_2|\cdots|u_n]$ if one needs to explicitly name the morphisms involved. A morphism in $\widetilde{\E}^\C_n$ will be denoted by $\underline{f} = (f_0, f_1,\ldots, f_n)$, with $f_i\in\Mor_\C(x_i,y_i)$.

We will refer to the simplicial object $\widetilde{\E}^\C_\bullet$ as the \hadgesh{simplicial envelope of $\C$}. The following example justifies the terminology. For a small category $\C$, let $\C_\bullet$ denote the constant simplicial object in $\Cat$ with value $\C$, namely $\C_n = \C$ for all $n\geq 0$, with all simplicial operators given by the identity functor.
\newcommand{\ev}{\mathrm{ev}}

\begin{Ex}\label{Ex:C[n]=C}
Let $\C$ be a small category. Then for each $m\geq 0$, there is an  inclusion functor $\iota_m\colon\C_m\to \widetilde{\E}^\C_m$ that takes an object $c\in\C_m$ to the constant sequence $\mathbf{1_c}\defeq (c=c=\cdots=c)$ in $\widetilde{\E}^\C_m$. There are two evaluation functors $\ev_0, \ev_m\colon \widetilde{\E}^\C_m\to \C_m$ given by sending an object $\underline{c} = c_0\xto{}c_1\xto{}\cdots\xto{}c_m$ in $\widetilde{\E}^\C_m$ to $c_0$ and $c_m$,  respectively. Clearly $\ev_0\circ\iota_m=\ev_m\circ\iota_m = 1_{\C_m}$. There are natural transformations $\eta_m\colon\iota_m\circ\ev_0\to 1_{\widetilde{\E}^\C_m}$ and $\gamma_m\colon 1_{\widetilde{\E}^\C_m}\to \iota_m\circ\ev_m$ given by taking an object $\underline{c}$, as above, to the obvious morphisms $\mathbf{1_{c_0}}\to \underline{c}$ and $\underline{c}\to \mathbf{1_{c_m}}$, respectively. It follows that $\C$ and $\widetilde{\E}^\C_m$ have homotopy equivalent nerves for each $m$. 

Notice that  the evaluation functors  $\widetilde{\E}^\C_\bullet\to \C_\bullet$ does not induce morphisms of simplicial objects, because $\ev_0$  does not respect the face operator $\partial_0$ and $\ev_m$ does not respect the face operator $\partial_m$. However, the inclusion functors do define a morphism of simplicial objects $\C_\bullet \to \widetilde{\E}^\C_\bullet$ that is a homotopy equivalence in every dimension. Hence, the inclusions induce a homotopy equivalence of the geometric realisations
\[\iota_*\colon||\C_\bullet|| \xto{\simeq} ||\widetilde{\E}^\C_\bullet||,\]
where for any simplicial object $\D_\bullet$ in $\Cat$ the symbol $||\D_\bullet||$ means  first  realising $\D_n$ for all $n$ and then taking the geometric realisation of the corresponding simplicial space. Clearly, $||\C_\bullet||$ (and hence also $||\widetilde{\E}^\C_\bullet||$) is homotopy equivalent to the geometric realisation of the nerve of $\C$. \qed
\end{Ex}

\begin{Rem}
Recall that  the singular set of a topological space $X$ is the set of all continuous maps from the standard simplex $\Delta^n$ to $X$. The simplicial object $\widetilde{\E}^\C_\bullet$ could be considered as a categorification of the singular set construction. However, as we shall see, the resulting cohomology theories behave very differently from the topological analog. 
\end{Rem}

Lying between the category $\C$ and its simplicial envelope, certain (semi-)simplicial subobjects will be of interest.

\begin{Defi}\label{Def:subobjects}
Consider the following subobjects of $\widetilde{\E}^\C_\bullet$.
\begin{enumerate}[(1)]
\item For $\C\in\Cat$, let  $\widetilde{\G}^\C_\bullet\subseteq\widetilde{\E}^\C_\bullet$ denote the simplicial subobject with the same  object sets in each dimension, but where non-identity morphisms  $\underline{x}\to\underline{y}$   are induced by morphisms $x_n\to y_0$ in $\C$.  \label{Def:subobjects-1}

\item For $\C\in\Cat_+$, define $\E^\C_\bullet\subseteq \widetilde{\E}^\C_\bullet$ and $\G^\C_\bullet\subseteq \widetilde{\G}^\C_\bullet$ to be the semi-simplicial subobjects, where in each dimension they are the respective full subcategories whose objects consist of the
non-degenerate simplices of $|\C|$.  \label{Def:subobjects-2}
\end{enumerate}
\end{Defi}

The  simplicial object $\widetilde{\G}^\C_\bullet$ will be referred to as the \hadgesh{compositional simplicial envelope of $\C$}. The same argument as in Example \ref{Ex:C[n]=C} shows that here too the inclusion induces a homotopy equivalence $||\C_\bullet|| \xto{\simeq} ||\widetilde{\G}^\C_\bullet||$.

\begin{Rem}
The restriction on $\C$ in Definition \ref{Def:subobjects}\ref{Def:subobjects-2} is necessary,  because the boundary of an object of the form $[u_1|\cdots|u_i|u_i^{-1}|\cdots|u_n]$ would have at least one degenerate face, which is not allowed in these constructions. 
\end{Rem}

The definition of the simplicial envelope $\widetilde{\E}^\C_\bullet$ and Definition \ref{Def:subobjects} easily imply Statements \ref{Th:Rep-Cats-1} and \ref{Th:Rep-Cats-2} of Theorem \ref{Th:Rep-Cats}. To complete the proof we only need to prove the naturality statement of the theorem.

If $\C\in\Cat_+$  then, by the definitions, one has an obvious  commutative square of inclusions
\begin{equation}\label{Eq:Square}
\xymatrix{
\G^\C_\bullet\ar[r]^{\gamma_\C}\ar[d]_{\iota_\C} & \widetilde{\G}^\C_\bullet\ar[d]^{\widetilde{\iota}_\C}\\
\E^\C_\bullet\ar[r]^{\epsilon_\C} & \widetilde{\E}^\C_\bullet
}\end{equation}

If $\C,\D\in\Cat$ are categories, then any functor $F\colon\C\to \D$ induces  a morphism of simplicial objects in $\Cat$ 
\[\widetilde{\E}^F\colon\widetilde{\E}^\C_\bullet\to \widetilde{\E}^\D_\bullet\quad\text{and}\quad \widetilde{\G}^F\colon\widetilde{\G}^\C_\bullet\to \widetilde{\G}^\D_\bullet.\]
This shows that the constructions $\widetilde{\E}^-_\bullet$ and $\widetilde{\G}^-_\bullet$ are functorial. One easily checks that $\widetilde{\E}^F$ and $\widetilde{\G}^F$  commute with the respective inclusions $\widetilde{\iota}_\C$ and $\widetilde{\iota}_\D$,  so $\widetilde{\iota}\colon\widetilde{\G}^-_\bullet\to\widetilde{\E}^-_\bullet$  is a natural transformation.

If $\C, \D\in \Cat_+$  and $F\colon\C\to \D$ is faithful, then $F$ sends any non-degenerate sequence of morphisms in $\C$ to a non-degenerate sequence of morphisms in $\D$. Hence one similarly gets a morphism of semi-simplicial objects in $\Cat_+$ 
\[\E^F\colon\E^\C_\bullet\to \E^\D_\bullet\quad\text{and}\quad \G^F\colon\G^\C_\bullet\to \G^\D_\bullet\]
that commute with $\iota_\C$ and $\iota_\D$ in (\ref{Eq:Square}).

Finally,  the inclusion functors $\gamma$ and $\epsilon$  clearly satisfy $\widetilde{\G}^F\circ\gamma_\C = \gamma_\D\circ\G^F$ and  $\widetilde{\E}^F\circ\epsilon_\C = \epsilon_\D\circ\E^F$, respectively. Hence $\gamma$ and $\epsilon$ are both natural transformations. This completes the proof of Parts \ref{Th:Rep-Cats-1} and \ref{Th:Rep-Cats-2} of Theorem \ref{Th:Rep-Cats}.

\begin{Rem}\label{Rem:Traitor}
The constructions $\E^\C_\bullet$ and $\G^\C_\bullet$ are well defined for arbitrary objects in $\Cat_+$, but they are not functorial with respect to arbitrary functors $F\colon \C\to \D$. This is because if $F$ is  not faithful, then the map induced on nerves may send a non-degenerate object in $|\C|_n$  to a degenerate object in $|\D|_n$.
\end{Rem}

The following corollary is immediate.
\begin{Cor}\label{Cor:Equivalence}
Let $\C, \D$ be small categories, and let $F\colon\C\to\D$ be an equivalence of categories. Then the induced functors $\widetilde{\E}^F$, $\widetilde{\G}^F$, $\E^F$ and $\G^F$ are equivalences of (semi-)simplicial objects in $\Cat$.
\end{Cor}

For any small category $\C$, the representation theory of the objects we defined is obviously very rich. Next we consider some natural subobjects that could help the investigation.
\begin{Defi}\label{Def:Restrictions}
Let $\C$ be a small category.
\begin{enumerate}[(1)]
\item Let $X\subseteq|\C|$ be a simplicial subset of $|\C|$. Define $\widetilde{\E}^{\C, X}_\bullet$ to be the simplicial object in $\Cat$, where $\widetilde{\E}^{\C, X}_n\subseteq \widetilde{\E}^{\C}_n$ is the full subcategory whose objects are the $n$-simplices of $X$. 
\label{Def:Restrictions-1}
\item Let $\M\subseteq \Mor_\C$ be any subset that is closed under compositions. Define  $\widetilde{\E}^{\C, \M}_\bullet$ to be the simplicial object, where for any morphism $\underline{f}=(f_0, f_1,\ldots, f_n)$, one has $f_i\in \M$ for all $i$. \label{Def:Restrictions-2}
\end{enumerate}
\end{Defi}

One can similarly make the analogous constructions of Definition \ref{Def:Restrictions} for $\widetilde{\G}^\C_\bullet$ and, when it makes sense, for $\E^\C_\bullet$ and $\G^\C_\bullet$. 

%%%%%%%%%%%%%%%%%%%%%%%%%%%%%%%%%%%%%%%%%%%

\subsection{Examples}
\label{SSec:Prelims-Examples}

\begin{Ex}\label{Ex:Restrictions}
A canonical example  of Definition \ref{Def:Restrictions}\ref{Def:Restrictions-1} arises by considering $X = |\C|^{(n)}$, the $n$-skeleton of $|\C|$. A useful example of \ref{Def:Restrictions-2} is given by  taking $\M=\emptyset$, so each $\widetilde{\E}^{\C,\emptyset}_n$ is the discrete category with objects the $n$-simplices of $\C$. Thus, upon identifying discrete categories with their sets of objects, $\widetilde{\E}^{\C,\emptyset}_\bullet$ is isomorphic to $|\C|$ as a simplicial set. \qed
\end{Ex}

The case where $\C$ is the category of a group is particularly simple, as the next example shows.

\begin{Ex}\label{Ex:BG}
Let $\Gamma$ be a discrete group, and let $\C=\B\Gamma$ denote the one object category with morphisms given by $\Gamma$ with the obvious composition rule. Then the objects of $\widetilde{\E}^{\B\Gamma}_n$ are in $1-1$ correspondence with $\Gamma^n$. Since all morphisms in $\B\Gamma$ are automorphisms, it is easy to see that $\iota\colon\widetilde{\G}^{\B\Gamma}_\bullet\to\widetilde{\E}^{\B\Gamma}_\bullet$ is an isomorphism of simplicial objects. It follows that each $\widetilde{\E}^{\B\Gamma}_n$ is a (connected) groupoid, where every morphism $\underline{g}\to\underline{h}$ is determined uniquely by any  of the elements forming it.  Hence  $\widetilde{\E}^{\B\Gamma}_n$ is equivalent as a category to ${\B\Gamma}$ for all $n\geq 0$, and the same holds for $\widetilde{\G}^{\B\Gamma}_n$. \qed
\end{Ex}

\begin{Ex}\label{Ex:Line}
Let $\C = [3]$. Then $\widetilde{\E}^{[3]}_1$ is a poset with Hasse diagram:
 
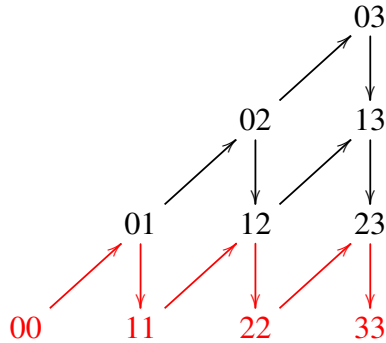
\[
 \xymatrix{
&&& 03\ar[d]\\
&&02\ar[d]\ar[ur]&13\ar[d]\\
&01\ar@[red][d]\ar[ur]&12\ar@[red][d]\ar[ur]&23\ar@[red][d]\\
\textcolor{red}{00}\ar@[red][ur]&\textcolor{red}{11}\ar@[red][ur]&\textcolor{red}{22}\ar@[red][ur]&\textcolor{red}{33}\\
}
\]
where $ij$ denotes the morphism $i<j$ in $[3]$ considered as an object in $\widetilde{\E}^{[3]}_1$. The subcategory $\E^{[3]}_1\subseteq\widetilde{\E}^{[3]}_1$ whose objects are the non-degenerate simplices in $|[3]|$ is drawn in black. The objects and morphisms in red are present in $\widetilde{\E}^{[3]}_1$ but not in $\E^{[3]}_1$.

The subcategory $\widetilde{\G}^{[3]}_1$ has the Hasse diagram 
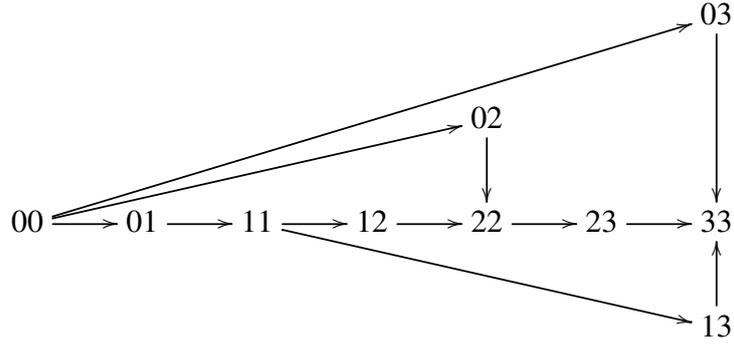
\[\xymatrix{&&&&&& 03\ar[dd]\\
&&&&02\ar[d]\\
00 \ar[r]\ar[rrrrrruu]\ar[rrrru] & 01  \ar[r] & 11   \ar[r]\ar[drrrr] & 12  \ar[r] & 22  \ar[r] & 23  \ar[r] & 33\\
&&&&&&13\ar[u]}\]
while the Hasse diagram for  $\G^{[3]}_1$ has the form
\[13\longleftarrow 01\longrightarrow 12\longrightarrow 23 \longleftarrow 02 \quad 03.\] 
Notice that $\widetilde{\G}^{[3]}_1$ is a connected category in this example, while $\G^{[3]}_1$ is not.  \qed
\end{Ex}

\subsection{Connectivity}\label{SS:Connectivity}

\begin{Lem}\label{Lem:KCn-connected}
Let $\C$ be a connected small category. Then  for each $n\geq 0$ the categories $\widetilde{\E}^\C_n$ and $\widetilde{\G}^\C_n$ are connected.
\end{Lem}
\begin{proof}
Let $\underline{a} = a_0\to\cdots\to a_n$ and $\underline{b} = b_0\to\cdots\to b_n$ be  objects in $\widetilde{\G}^\C_n$, and let $\mathbf{1_{a_n}}$ and $\mathbf{1_{b_0}}$ be the constant objects on $a_n$ and $b_0$, respectively. Then there are  morphisms  $\underline{a}\to\mathbf{1_{a_n}}$ and $\mathbf{1_{b_0}}\to\underline{b}$ in $\widetilde{\G}^\C_n$. Since $\C$ is connected, there is a zig-zag of morphisms in $\C$ from $a_n$ to $b_0$, which in turn induces a zig-zag of constant objects in $\widetilde{\G}^\C_n$ from $\mathbf{1_{a_n}}$ to $\mathbf{1_{b_0}}$. Since $\widetilde{\E}^\C_n$ has the same objects as $\widetilde{\G}^\C_n$, the same applies to $\widetilde{\E}^\C_n$.
\end{proof}

By contrast, connectivity of $\C$ is not hereditary for $\E^\C_n$ and $\G^\C_n$ if $n>0$. For instance, if $\C$ is the category with Hasse diagram given by a $1$ to $n$ corolla (a digraph with one input vertex and $n$ output vertices), where $n>1$, then $\E^\C_1 = \G^\C_1$ consists of $n$ isolated objects. If $\C$ is a small category and $a_0\to\cdots\to a_n$ is a chain of morphisms in $\C$, where $a_0$ is minimal and $a_n$ maximal ($\C(-,a_0)=\emptyset$ and $\C(a_n,-)=\emptyset$, respectively), then the corresponding object is isolated   in $\G^\C_n$. 

The following, rather specialised, concepts of connectivity will be useful later. 
\begin{Defi}
\label{Def:Line-conn}
Let $\C$ be a small category. A subset $\M\subseteq \Mor_\C$ of non-identity morphisms is said  to be  \hadgesh{line-connected} if  the full subcategory of $\G^\C_1$ whose objects are the elements of $\M$ is connected. The subset $\M$ is said to be a \hadgesh{line-component} if it is line-connected and if the subcategory of $\G^\C_1$ that it generates is a connected component. 

Similarly, $\M\subseteq \Mor_\C$ is said to be \hadgesh{square-connected} if the full subcategory of $\E^\C_1$ whose objects are the elements of $\M$ is connected. The subset $\M$ is said to be a \hadgesh{square-component} if it is square-connected and if the subcategory of $\E^\C_1$ that it generates is a connected component.  
\end{Defi}

Clearly a line-connected set of morphisms is also square-connected, but not vice versa. For instance, if $\C$ is the category $[2]$, then $\E^\C_1$ is the interval $01\to 02\to 12$, while $\G^\C_1$ has two connected components: the interval $01\to 12$ and the singleton $02$. On the other hand, the category with objects $0, 1, 2$ and unique morphisms from $0$ to $1$ and $2$ is neither line- nor square-connected.

\begin{Defi}\label{Def:Succ-Pred}
Let $\C$ be any category and let $\varphi\colon a\to b$  be a morphism in $\C$. An irreducible morphism $b\to x$ in $\C$, if one exists, is said to be a \hadgesh{successor of $\varphi$}. Similarly, an irreducible morphism $y\to a$, if one exists, is called  a \hadgesh{predecessor of $\varphi$}.
\end{Defi}

Of course neither successors nor predecessors, if they exist, are necessarily unique, except in some special circumstances.  Recall that a category $\C$ is called a \hadgesh{direct category} (or a $1$-way category) if  it contains no infinite  chains of nonidentity morphisms, including cycles of positive  length. This implies, in particular,  that all endomorphisms in a direct  category are identity morphisms.

\begin{Lem}\label{Lem:Rooted-trees-succ-pred}
Let $\C$ be a finite  direct category. Then the following statements are equivalent.
\begin{enumerate}[(1)]
\item Every morphism $\varphi\colon x\to y$ in $\C$, such that $\C(y,x)\neq\emptyset$ for some $y\neq x$, has a unique predecessor. \label{Lem:Rooted-trees-succ-pred-1}
\item The quiver whose edge  set consists of all irreducible morphisms in $\C$ forms a disjoint union of  rooted trees. \label{Lem:Rooted-trees-succ-pred-2}
\end{enumerate}
\end{Lem}
\begin{proof} Clearly, \ref{Lem:Rooted-trees-succ-pred-2} implies \ref{Lem:Rooted-trees-succ-pred-1}. For the converse, notice that for each object $x\in\C$, such that $\C(-,x)\neq\emptyset$, there is a unique object $x_0$ with $\C(-,x_0)=\emptyset$, and a unique directed path of irreducible morphisms from $x_0$ to $x$. Thus, the collection of all irreducible morphisms in $\C$ forms a disjoint union of rooted trees. 
\end{proof}

%%%%%%%%%%%%%%%%%%%%%%%%%%%%%%%%%%%%%%%
%%%%%%%%%%%%%%%%%%%%%%%%%%%%%%%%%%%%%%%

\subsection{Module categories and the Grothendieck group and ring}
\label{SSec:Prelims-modules}
Fix a field $k$ and let $\VVect$ denote the category of \hadgesh{finite dimensional} vector spaces over $k$. For $\C\in\Cat$, a $k\C$-module is a functor $M\colon\C\to\VVect$. Let $k\C\mod$ denote the category of $k\C$-modules with morphisms given by natural transformations. 

\begin{Rem}
If $\C$ is a finite category, then there is an isomorphism of categories between $k\C\mod$ and the category of ordinary finite dimensional modules over the category algebra $k\C$, but this isomorphism does not hold in case $\C$ is infinite \cite{Xu}.
\end{Rem} 

For any additive category $\A$, let $\Gr(\A)$ denote the Grothendieck group of isomorphism classes of objects in $\A$. The operation in $\Gr(\A)$ is defined by the binary operation in $\A$ and the zero object is the zero object of $\A$. We shall be interested specifically in categories of the form $k\C\mod$, which are always abelian. If $\A = (\A, \otimes, I) $ is a symmetric monoidal additive category, then the monoidal structure of $\A$ induces a product operation on $\Gr(\A)$ with $I$ as the unit of the operation. Notice that since $\VVect$ is a symmetric monoidal category, so is any module category of the form $k\C\mod$, where the monoidal structure is given by tensor product over $k$, and the unit element is the constant module on $\C$ with value $k$ and identities on all morphisms.  Hence $\Gr(k\C)$ has a commutative ring structure induced by the monoidal structure of $k\C\mod$. 

\begin{Rem}
In this article we divert from conventional notation for the Grothendieck group. The object we denote by  by $\Gr(k\C)$ would be denoted in the literature by $\mathbf{K_0}(\C, \VVect)$ or  $\mathbf{K_0}(\Fun(\C, \VVect))$ or similar variations. Also, within our own convention,  we denote what should be  $\Gr(k\C\mod)$ by $\Gr(k\C)$  for short.
\end{Rem}

\begin{Rem}
The Grothendieck group as defined above is sometimes referred to in the literature as the \hadgesh{split} Grothendieck group. One may also consider the quotient group of $\Gr(k\C)$ where the operation is defined by $[M] + [K] = [E]$ if there is a short exact sequence 
\[0\to K\to E\to M\to 0\]
of $k\C$-modules. This explains the use of the word `split'. Unless otherwise specified, we shall not consider this quotient. 
\end{Rem}

\begin{Defi}
\label{Def:Rank-Inv}
Let $\C$ be a small category. Define a homomorphism $\rk\colon\Gr(k\C)\to \Z^{\Mor_\C}$ as follows. For a module  $M\in k\C\mod$ and a morphism $\varphi$ in $\C$, let $\rk[M](\varphi) = \rk(M(\varphi))$. Extend to the full Grothendieck group by additivity. The homomorphism $\rk$ will be referred to as the \hadgesh{rank invariant}. 
\end{Defi}
 Notice that the rank invariant, as defined above, also keeps track of the Hilbert function, namely the function that associates with a module $M\in k\C\mod$ and an object $x\in \C$ the dimension $\dim_k(M(x)) = \rk[M](1_x)$.

\begin{Defi}\label{Def:Interval}
Let $\C$ be a small direct category. An \hadgesh{interval in $\C$} is a subcategory $\Q\subseteq\C$ that is convex  and connected, namely  such that  for any $x, y\in \Q$ the following two conditions are satisfied:
\begin{enumerate}[(1)]
\item $x$ and $y$ are in the same connected component of $\C$.
\item If  $z\in\C$ and $\alpha\colon x\to z$ and $\beta\colon z\to y$ are morphisms, such that $\beta\circ\alpha$ is a morphism in $\Q$, then $z\in\Q$ and $\alpha, \beta$ are morphisms in $\Q$.
\end{enumerate}

If $\Q\subseteq\C$ is an interval, then the  \hadgesh{interval module determined by $\Q$}  is the module $I_\Q\in k\C\mod$ that takes the value $k$ on each $x\in\Q$, the identity morphism for each morphism in $\Q$ and the value $0$ for each $z\notin\Q$.
\end{Defi}

Notice that convexity implies that the concept of an interval module is well defined, because if $\Q\subseteq\C$ is an interval and $x\notin\Q$, then no morphism in $\Q$ factors through $x$. The following lemma is basic.

\begin{Lem}\label{Lem:Interval-modules-indec}
For any small direct category $\C$, and interval $\Q\subseteq\C$, the interval module $I_\Q\in k\C\mod$ is indecomposable.
\end{Lem}
\begin{proof}
Suppose $I_\Q= M\oplus N$. Then, restricting $I_\Q$, $M$ and $N$ to $\Q$, we obtain a decomposition $M|_\Q\oplus N|_\Q$ for the constant module on $\Q$ with values $k$. If both $M|_\Q$ and $N|_\Q$ are nontrivial, then the objects of $\Q$ can be partitioned into two disjoint non-empty sets such that each restricted module obtains the value $k$ on one subset and vanishes on the other. But since $I_\Q$ is constant on $\Q$, this is impossible. Hence we may assume without loss of generality that $M|_\Q=I_\Q|_\Q$ and $N|_\Q=0$. But since $I_\Q$ vanishes outside of $\Q$, it follows that $N=0$ and so $M = I_\Q$.
\end{proof}

\begin{Defi}\label{Def:Inverse-Interval}
Let $\C$ and $\D$ be direct categories and let $F\colon \C\to\D$ be a functor. Let $\Q\subseteq\D$ be an interval. Let $F^{-1}(\Q)\subseteq\C$ denote the subcategory whose objects are all objects  $x\in\C$ such that $F(x)\in\Q$, and whose morphisms are all morphisms $\varphi$ between these objects such that $F(\varphi)$ is a morphism in $\Q$.
\end{Defi}

\begin{Lem}\label{Lem:Inverse-Interval}
Let $\C$ and $\D$ be direct categories and let $F\colon \C\to\D$ be a functor. Let $\Q\subseteq\D$ be an interval. Then $F^{-1}(\Q)$ is a disjoint union of intervals in $\C$. Thus if $I_\Q\in k\D\mod$ is the interval module defined by $\Q$, then $F^*(I_\Q)\in k\C\mod$ is isomorphic to a direct sum of the interval modules corresponding to the intervals in $\C$ whose disjoint union forms $F^{-1}(\Q)$.
\end{Lem}
\begin{proof}
Notice first that $F^{-1}(\Q)$, as defined in Definition \ref{Def:Inverse-Interval}, is a well-defined subcategory, i.e., it is closed under compositions. If $\varphi, \psi$ are morphisms in $F^{-1}(\Q)$ such that $\varphi\circ\psi$ is defined in $\C$, then $F(\varphi)\circ F(\psi) = F(\varphi\circ\psi)$ is defined in $\Q$, and so $\varphi\circ\psi$ is in $F^{-1}(\Q)$. 

Let $\P\subseteq F^{-1}(\Q)$ be a connected component. We claim that $\P$ is convex, namely that if $\varphi\colon x\to y$ is a morphism in $\P$ and there exists some $z\in\C$ and morphisms $\psi'\colon x\to z$ and $\psi''\colon z\to y$ such that $\psi''\circ\psi' = \varphi$, then $z\in \P$ as well as the morphisms $\psi', \psi''$. But in $\D$ one has the obvious relation $F(\varphi) =  F(\psi'')\circ F(\psi')$. Hence $z\in F^{-1}(\Q)$ and so are $\psi'$ and $\psi''$, and they are clearly all in the same connected component as $x, y$ and $\varphi$. Thus $F^{-1}(\Q)$ is a disjoint union of intervals in $\C$ as claimed. 

Now, let $I_\Q\in k\D\mod$ be the interval module corresponding to $\Q$. Then 
 $F^*(I_\Q) = I_\Q\circ F$, so 
 \[F^*(I_\Q)(x) = \begin{cases} k & F(x)\in\Q\\ 0 & F(x)\notin\Q.\end{cases}\]
Let $x, y\in \P\subseteq F^{-1}(\Q)$ be objects, where $\P$ is an interval component of $F^{-1}(\Q)$, and assume that there is a morphism $\varphi\colon x\to y$ in $\P$. Then $F^*(I_\Q)(x) = F^*(I_\Q)(y) = k$ and  $F^*(I_\Q)(\varphi)\colon F(x)\to F(y)$ is the identity. Thus, the module that receives the value $k$ at each object of $\P$ and the identity morphism for any morphisms in $\P$ and the value $0$ elsewhere is the interval module $I_\P$. 

Since $F^{-1}(\Q)$ is a disjoint union of intervals, each of which is a connected component in $F^{-1}(\Q)$, it follows that $F^*(I_\Q)$ is isomorphic to a direct sum of interval modules $I_\P$, where $\P$ is a connected component of $F^{-1}(\Q)$. This completes the proof of the lemma.
\end{proof}

\begin{Not}
In the case where an interval $\Q\subseteq\C$ consists of a single object $x\in\C$, the corresponding interval module will be denoted by \hadgesh{$S_x$}. In many cases, for instance when $\C$ is  a poset, these are the simple $k\C$-modules. However, since this is not always the case we shall refer to modules of the form $S_x$ as \hadgesh{singleton modules}. 
\end{Not}

%%%%%%%%%%%%%%%%%%%%%%%%%%%%%%%%%%%%%%%%%%%
%%%%%%%%%%%%%%%%%%%%%%%%%%%%%%%%%%%%%%%%%%%
%%%%%%%%%%%%%%%%%%%%%%%%%%%%%%%%%%%%%%%%%%%
%%%%%%%%%%%%%%%%%%%%%%%%%%%%%%%%%%%%%%%%%%%
%%%%%%%%%%%%%%%%%%%%%%%%%%%%%%%%%%%%%%%%%%%
%%%%%%%%%%%%%%%%%%%%%%%%%%%%%%%%%%%%%%%%%%%
%%%%%%%%%%%%%%%%%%%%%%%%%%%%%%%%%%%%%%%%%%%
%%%%%%%%%%%%%%%%%%%%%%%%%%%%%%%%%%%%%%%%%%%
%%%%%%%%%%%%%%%%%%%%%%%%%%%%%%%%%%%%%%%%%%%
%%%%%%%%%%%%%%%%%%%%%%%%%%%%%%%%%%%%%%%%%%%
%%%%%%%%%%%%%%%%%%%%%%%%%%%%%%%%%%%%%%%%%%%

\section{Representation cohomology of small categories}
\label{Sec:Rep-Coho}

A (semi-)simplicial object $\C_\bullet$  in $\Cat$ naturally gives rise to  a (semi-)cosimplicial abelian group by applying the Grothendieck group construction. 

\subsection{Definitions}
\label{SSec:Rep-Coho-defs}
Let $\C_\bullet$ be an arbitrary semi-simplicial object in $\Cat$. For any field $k$ the Grothendieck group $\Gr(k\C_\bullet)$ is a semi-cosimplicial abelian group, 
where the face operators $\partial_i$ induce coface homomorphisms 
\[\delta^i \colon \Gr(k\C_n)\to \Gr(k\C_{n+1})\]
defined by 
\[\delta^i[M] \defeq [\partial_i^*M] = [M\circ \partial_i],\]
for each $0\le i\le n$.  Let $d^n\colon\Gr(k\C_n)\to \Gr(k\C_{n+1})$ denote the coboundary operator obtained as the alternating sum of the $\delta^i$, as usual.  Similarly, if $\C_\bullet$ is a simplicial object in $\Cat$ with degeneracy operators $s_i$, we denote the $\sigma^i\defeq s_i^*\colon\Gr(k\C_{n+1})\to \Gr(k\C_{n})$ the corresponding codegeneracy operators.

\begin{Defi}
\label{Def:Rep-Coho} 
Let $\C_\bullet$ be a semi-simplicial object in $\Cat$ and let $k$ be a field. Define the \hadgesh{representation cohomology $\mathscr{R}^*(\C_\bullet, k)$ of $\C_\bullet$ over $k$} by \[\RH^*(\C_\bullet, k) \defeq H^*(\Gr(k\C_\bullet), d^*).\] 
\end{Defi}

\begin{Ex}\label{Ex:Simp-grp}
Let $\Gamma\bul$ be a simplicial group. Then $k\Gamma\bul\mod$ is a cosimplicial object in $\Cat$, and hence one has the associated representation cohomology $\RH^*(\Gamma\bul, k)$.
\end{Ex}

 If $F\colon\C_\bullet\to\D_\bullet$ is a functor of semi-simplicial objects in $\Cat$, then one has the obvious induced homomorphism 
\[F^*\colon \RH^*(\D_\bullet, k) \to \RH^*(\C_\bullet, k).\]

Next we restrict to the special cases of the (semi-)simplicial objects that arise from the constructions of Section \ref{Sec:Prelims}. Fix a small category $\C$. In Section \ref{SSec:Prelims-Setup} we defined two simplicial objects in $\Cat$, $\widetilde{\E}^\C_\bullet$ and $\widetilde{\G}^\C_\bullet$, and two semi-simplicial objects $\E^\C_\bullet$ and $\G^\C_\bullet$. This gives rise to four respective cohomology theories associated to $\C$. For any small category $\C$, define
\[\widetilde{\EE}^*(\C, k) \defeq \RH^*(\widetilde{\E}^\C_\bullet, k)\quad\text{and}\quad  \widetilde{\GG}^*(\C, k)\defeq \RH^*(\widetilde{\G}^\C_\bullet, k).\] 
If $\C$ does not contain any non-identity invertible morphisms, then we define 
\[\EE^*(\C, k) \defeq \RH^*(\E^\C_\bullet, k)\quad\text{and}\quad  \GG^*(\C, k)\defeq \RH^*(\G^\C_\bullet, k).\] 
Thus, when all variants are defined, Diagram (\ref{Eq:Square}) of semi-simplicial objects in $\Cat$ and naturality of our constructions yield a  commutative square:
\[\xymatrix{
\GG^*(\C, k) & \widetilde{\GG}^*(\C, k) \ar[l]_{\gamma^*}\\
\EE^*(\C, k)\ar[u]^{\check{\iota}^*} & \widetilde{\EE}^*(\C, k)\ar[u]_{\iota^*} \ar[l]_{\epsilon^*}
}\]
The aim of this article is to study the properties of these cohomology theories and the information they provide about the category $\C$ and the module category $k\C\mod$.

\subsection{Basic Properties}
\label{SSec:Basic}
Clearly, if $F$ is an equivalence of semi-simplicial objects in $\Cat$, then $F^*$ is an isomorphism. To be an equivalence of semi-simplicial objects  is a stronger condition  than requiring $F$ to be an  equivalence of categories in each dimension, since the inverses must also form a functor of semi-simplicial objects. 
However, the next lemma shows that having dimension-wise equivalences suffices in the case of the cohomology theories we consider here.

\begin{Lem}\label{Lem:Equiv}
Let $F\colon\C\bul\to\D\bul$ be a functor of semi-simplicial objects in $\Cat$ such that $F_n\colon\C_n\to\D_n$ is an equivalence of categories for each $n$. Then the induced map $F^*$ is an isomorphism. In particular, if $F\colon\C\to\D$ is an equivalence of categories, then $\widetilde{\E}^F_\bullet$, $\widetilde{\G}^F_\bullet$, $\E^F_\bullet$ and $\G^F_\bullet$ all induce isomorphisms on the respective cohomology.
\end{Lem}
\begin{proof}
The second claim follows from the first claim and Corollary \ref{Cor:Equivalence}. Since each $F_n$ is an equivalence of categories, it follows that $F_n^*\colon k\D_n\mod\to k\C_n\mod$ is an equivalence of categories and hence induces an isomorphism $F_n^*\colon\Gr(k\D_n)\to\Gr(k\C_n)$. Hence the  map $F^*$ of graded abelian groups induced by $F$ is a cochain map that is a dimension-wise isomorphism, and the proof is complete.
\end{proof}

An obvious example arises from the category of a group, or more generally, a groupoid.

\begin{Prop}\label{Prop:Groupoids}
Let $\HH$ be a groupoid. Then 
\[\widetilde{\EE}^0(\HH, k) \cong\widetilde{\GG}^0(\HH, k) \cong \bigoplus_{x}\Gr(k\Aut_\HH(x)),\]
where the sum runs over representatives of the connected components of $\HH$, and $\widetilde{\EE}^i(\HH, k) = \widetilde{\GG}^i(\HH, k)=0$ for all $i>0$.
\end{Prop}
\begin{proof}
Since every morphism in $\HH$ is invertible, it is clear that $\widetilde{\E}^\HH_\bullet \cong \widetilde{\G}^\HH_\bullet$, hence the corresponding cohomology groups coincide. Write $\HH$ as a disjoint union of connected components $\HH = \coprod \HH_x$, where $x\in\HH_x$. Then 
\[\widetilde{\EE}^*(\HH, k) \cong \bigoplus \widetilde{\EE}^*(\HH_x, k).\]
Furthermore, the inclusion $\B\Aut_{\HH}(x) \to \HH_x$ is an equivalence of categories, and the equivalence extends to equivalences $\widetilde{\E}^{\B\Aut_{\HH}(x)}_n\xto{\simeq}\widetilde{\E}^{\HH_x}_n$ for all $n$. Hence, by Lemma \ref{Lem:Equiv}, $\widetilde{\EE}^*(\B\Aut_{\HH}(x), k)\cong\widetilde{\EE}^*(\HH_x, k)$. Thus it suffices to consider the case where $\HH = \B\Gamma$, for some discrete group $\Gamma$.

For each $n\ge 0$,  the objects of the category $\widetilde{\E}^{\B\Gamma}_n$ are the simplices of the nerve of $\Gamma$, namely  ordered $n$-tuples $\underline{x} = [x_1|\ldots| x_n]$ of elements of $\Gamma$ (considered as morphisms of the single object of $\Gamma$). A morphism  $\underline{x}\to\underline{y}$ in $\widetilde{\E}^{\B\Gamma}_n$ is an $(n+1)$-tuple $(z_0,z_1,\ldots, z_n)$ of elements of $\Gamma$ such that $z_ix_i = y_iz_{i-1}$ for all $1\le i\le n$. Since all elements of $\Gamma$ are invertible, such a morphism is determined uniquely by each of the $z_j$. Hence in particular $\widetilde{\E}^{\B\Gamma}_n\cong \widetilde{\G}^{\B\Gamma}_n$ for all $n$. Furthermore, since all morphisms in $\widetilde{\E}^{\B\Gamma}_n$ are isomorphisms, it is a connected groupoid, and hence equivalent as such to the automorphism group of the constant object $\underline{e}$, where $e\in\Gamma$ is the identity element. 

For each $n\ge 0$, let $\mathbf{p}_n\subseteq\widetilde{\E}^{\B\Gamma}_n$ denote the full subcategory on the basepoint (the image of the degeneracy of the unique object in $\widetilde{\E}^{\B\Gamma}_0 = \B\Gamma$). Then $\mathbf{p}_n\cong \B\Gamma $, and by the argument above the inclusion $\mathbf{p}_n\subseteq\widetilde{\E}^{\B\Gamma}_n$ is an equivalence of categories. Let $\mathbf{p}_\bullet$ be the simplicial object in $\mathbf{Cat}$ where all faces and degeneracies are the identity. Then the inclusion $\mathbf{p}_\bullet\to\widetilde{\E}^{\B\Gamma}_\bullet$ is a simplicial functor, that is dimension-wise an equivalence of categories.  By Lemma \ref{Lem:Equiv} the inclusion induces an isomorphism on cohomology. It follows that 
\[\widetilde{\EE}^0(\B\Gamma, k)\cong \widetilde{\GG}^0(\B\Gamma, k)\cong \Gr(k\Gamma),\] 
and $\widetilde{\EE}^i(\B\Gamma, k)=\widetilde{\GG}^i(\B\Gamma, k)=0$ for $i>0$. 
\end{proof}

Notice that Proposition \ref{Prop:Groupoids}  is an easy example of Theorem \ref{Th:E0G0}. The next example shows that representation cohomology generalises ordinary cohomology of the nerve of a category. 

\newcommand{\Fl}{\mathrm{Fl}}

\begin{Ex}\label{Ex:Discrete-cat-coho}
Let $\C$ be a small category and let $\widetilde{\E}^{\C,\emptyset}_\bullet = \widetilde{\G}^{\C,\emptyset}_\bullet$ and (assuming $\C$ has no non-identity invertible morphisms)  $\E^{\C,\emptyset}_\bullet = \G^{\C,\emptyset}_\bullet$ be the corresponding discrete (semi-)simplicial objects. 

A $k\widetilde{\E}^{\C,\emptyset}_n$-module is an assignment of a finite dimensional $k$-vector space with every simplex in $\widetilde{\E}^{\C,\emptyset}_n$. Thus $\Gr(k\widetilde{\E}^{\C,\emptyset}_n)$ is the group of all functions $|\C|_n\to \Z$, and hence it is isomorphic to $\Hom(C_n(|\C|), \Z)$, where $C_n(|\C|)$ is the $n$-th chain group of the nerve. It follows that 
\[\RH^*(\widetilde{\E}^{\C,\emptyset}_\bullet, k) =  \RH^*(\widetilde{\G}^{\C,\emptyset}_\bullet, k)\cong H^*(|\C|, \Z).\] 
Similarly, if $\C$ does not contain non-identity invertible morphisms, then $\E^{\C,\emptyset}_\bullet$ is the  semi-simplicial object, with $\E^{\C,\emptyset}_n$  the discrete category with objects the non-degenerate simplices of $|\C|_n$.  Thus we have $\RH^*(\E^{\C,\emptyset}_\bullet, k) =  \RH^*(\G^{\C,\emptyset}_\bullet, k)\cong H^*(|\C|, \Z)$. The respective restrictions induce homomorphisms and a commutative diagram
\begin{equation}\label{Diag:RH*-diag}
\xymatrix{
\GG^*(\C, k)\ar@{-->}[drr]&&&&\widetilde{\GG}^*(\C, k)\ar@{-->}[llll]\ar[dll]\\
&& H^*(|\C|, \Z)\\
\EE^*(\C, k)\ar@{-->}[uu]\ar@{-->}[urr] &&&&\widetilde{\EE}^*(\C, k)\ar[uu]\ar@{-->}[llll]\ar[ull]
}\end{equation}
where the dashed arrows serve as a reminder that the theories $\GG^*$ and $\EE^*$ only exist under the assumption that $\C$ has no non-identity invertible morphisms.
\end{Ex}

\begin{Rem}\label{Rem:Reachability}
Let $X$ be an acyclic  quiver and let $\T_X$ denote its \hadgesh{reachability category} \cite{CR}, i.e.,  the thin category (preorder) associated to $X$. Thus $\T_X$ is  the category whose objects are the vertices of $X$ and with a unique non-identity morphism $x\to y$ if there is a directed path in $X$ from $x$ to $y$ and only the identity endomorphism for each object. The \hadgesh{reachability homology} of the quiver $X$ can be defined as the ordinary homology of the nerve of $\T_X$ \cite[Remark 4.2]{HR}. Similarly, one can define \hadgesh{reachability cohomology} as the cohomology of the nerve of $\T_X$. By Example \ref{Ex:Discrete-cat-coho}, $\RH^*(\G^{\T_X,\emptyset}_\bullet, k) \cong H^*(|\T_X|, \Z)$, and the obvious inclusion $\iota_\emptyset\colon\G^{\T_X,\emptyset}_\bullet \to \G^{\T_X}_\bullet$ induces a homomorphism  
\[\iota_\emptyset\colon \GG^*(\T_X, k) \to  H^*(|\T_X|, \Z).\]

Furthermore, if $X$ is finite and $R$  is an arbitrary commutative ring, then one has 
\[\Gr(k\G^{\T_X}_\bullet)\otimes R \to \Gr(k\G^{\T_X, \emptyset}_\bullet)\otimes R \cong \Hom(C_*(|\T_X|),\Z)\otimes R\cong \Hom(C_*(|\T_X|), R).\]
Taking cohomology, one gets a homomorphism 
\[H^*(\Gr(k\G^{\T_X}_\bullet)\otimes R)\to H^*(|\T_X|, R)\]
where the target is reachability cohomology of $X$ with coefficients in $R$.
\end{Rem}

Computing representation cohomology in general is very difficult. This is primarily due to the fact that the representation type of  a small category is typically infinite and almost always wild. This motivates the introduction of filtrations on our constructions.

\begin{Defi}\label{Def:MaxDim-d}
Let $\C$ be a small category. A module $M\in k\C\mod$ is said to be of \hadgesh{bounded dimension $d$} if $\dim_k(M(x))\le d$ for each object $x\in \C$.
\end{Defi}

Obviously, if $M\in k\widetilde{\E}^\C_n\mod$ is an indecomposable module of bounded dimension $d$, then $\delta^i[M]$ is a direct sum of (not necessarily indecomposable) modules of bounded dimension $d$.  Hence the graded subgroup of $\Gr(k\widetilde{\E}^\C\bul)$ that is generated by indecomposable modules of bounded dimension $d$   is closed under coface (and codegeneracy) operators, and hence forms a cosimplicial abelian subgroup of $\Gr(k\widetilde{\E}^\C\bul)_d\le\Gr(k\widetilde{\E}^\C\bul)$. Thus we obtain an \hadgesh{increasing filtration of $\Gr(k\widetilde{\E}^\C\bul)$} by subgroups
\[\Gr(k\widetilde{\E}^\C\bul)_1\le\Gr(k\widetilde{\E}^\C\bul)_2\le \cdots\le \Gr(k\widetilde{\E}^\C\bul)_d \le \cdots \le\Gr(k\widetilde{\E}^\C\bul).\]
Of course, the same considerations apply to $\widetilde{\G}$, $\E$ and $\G$. The corresponding cohomology groups will be denoted by $\widetilde{\EE}^*(\C,k)_d$, and similarly for the other variants.

The subgroup $\Gr(k\widetilde{\E}^\C\bul)_d$ is generally still hard to control even for $d=1$. An exception occurs if $\C$ is a finite poset whose Hasse diagram is a tree.   In this case the number of isomorphism classes of indecomposable modules  $M\in k\C\mod$  of bounded dimension $1$ is  finite. In fact, isomorphism classes of such modules are in bijective correspondence with (connected) subtrees of the Hasse diagram. Notice also that if $\C$ is a a poset whose Hasse diagram is a disjoint union of trees, then any module $M\in k\C\mod$ of bounded dimension $1$ is a sum of interval modules, but in general one may have infinitely many isomorphism classes of interval modules (See Section \ref{SSec:Cyclic}). 

The following corollary is obvious. 

\begin{Cor}
If $\C$ is a finite poset whose Hasse diagram is disjoint union of trees, then the groups $\Gr(k\widetilde{\E}^\C_n)_1, \Gr(k\widetilde{\G}^\C_n)_1, \Gr(k\E^\C_n)_1$ and $\Gr(k\G^\C_n)_1$ are finitely generated for all $n\ge 0$. Hence the corresponding cohomology groups are finitely generated abelian groups.
\end{Cor}

%\iffalse
%%%%%%%%%%%%%%%%%%%%%%%%%%%%%%%%%%%%%%%%%%%
%%%%%%%%%%%%%%%%%%%%%%%%%%%%%%%%%%%%%%%%%%%
\subsection{Products}

Let $\C$ be a small category. Then its nerve $|\C|$ is a simplicial set, whose chain complex admits the Alexander-Whitney diagonal  approximation, which induces the usual cup product on $H^*(|\C|, \Z)$. Unsurprisingly, the representation cohomology theories for $\C$ defined above admit a (not generally commutative)  ring structure that is compatible with that of $H^*(|\C|, \Z)$ via the homomorphisms in Diagram (\ref{Diag:RH*-diag}).

For each $n>0$ and $p, q\le n$, let
\[\varphi_p\colon|\C|_n\to|\C|_p\quad\text{and}\quad \beta_q\colon|\C|_n\to|\C|_q\]
denote the front and back face maps respectively, where
\[\varphi_p[u_1|\cdots|u_n] \defeq [u_1|\cdots|u_p],\quad  \text{and}\quad \beta_q[u_1|\cdots|u_n] \defeq [u_{n-q+1}|\cdots|u_n].\]
Since these maps are defined as iterated face maps, they induce homomorphisms on Grothendieck groups
\[\varphi_p^*\colon\Gr(k\widetilde{\E}^\C_p)\to\Gr(k\widetilde{\E}^\C_n),\quad\text{and}\quad \beta_q^*\colon\Gr(k\widetilde{\E}^\C_q)\to\Gr(k\widetilde{\E}^\C_n),\]
as well as on the respective components of any semi-simplicial subgroup of $\Gr(k\widetilde{\E}^\C_\bullet)$. In particular, these maps are natural with respect to inclusions of semi-simplicial subgroups, as well as with respect to maps induced by functors $\C\to\D$ that induce semi-simplicial maps between the corresponding objects. In particular, under the identification of $\Gr(k\widetilde{\E}^{\C,\emptyset}\bul)$ with the cochain complex $C^*(|\C|, \Z)$, and using the product structure map $\mu$ in the Grothendieck group, one has a map
\[\Gr(k\widetilde{\E}^{\C,\emptyset}_p)\otimes\Gr(k\widetilde{\E}^{\C,\emptyset}_q)\xto{\varphi_p^*\otimes\beta_{q}^*} \Gr(k\widetilde{\E}^{\C,\emptyset}_{p+q})\otimes\Gr(k\widetilde{\E}^{\C,\emptyset}_{p+q})\xto{\mu}\Gr(k\widetilde{\E}^{\C,\emptyset}_{p+q}),\]
which coincides with the map induced by the Alexander-Whitney diagonal.

Similarly, for each $n\geq 0$ and a pair of non-negative integers $p, q\le n$, there are  obvious functors
\[\widetilde{\E}^\C_{n}\to  \widetilde{\E}^\C_{p} \quad\text{and}\quad \widetilde{\E}^\C_{n}\to  \widetilde{\E}^\C_{q},\]
given by sending an object $[u_1|\cdots  |u_{n}]$ to $[u_1|\cdots|u_p]$ and $[u_{n-q+1}|\cdots |u_{n}]$ respectively, with the obvious action on morphisms. 
Taking  Grothendieck groups and using the multiplicative structure, we get a map
\[\cup\colon\Gr(k\widetilde{\E}^\C_{p})\otimes\Gr(k\widetilde{\E}^\C_{q})\to \Gr(k\widetilde{\E}^\C_{p+q})\otimes\Gr(k\widetilde{\E}^\C_{p+q})\to \Gr(k\widetilde{\E}^\C_{p+q}),\]
and similarly for the corresponding groups in any semi-simplicial subobject of $\Gr(\widetilde{\E}^\C_\bullet)$. 

\begin{Thm}\label{Thm:Products}
Let $\C$ be a small category and let $\D_\bullet\subseteq \widetilde{\E}^\C\bul$  be any  semi-simplicial subobject. Then the restriction of $\cup$ to $\Gr(k\D\bul)$
\[\cup\colon\Gr(k\D_p)\otimes\Gr(k\D_q)\to\Gr(k\D_{p+q})\]
for each pair of non-negative integers $p$ and  $q$, induces the structure of a differential graded $\Z$-algebra on $\Gr(k\D_\bullet)$. Thus the cohomology $H^*(\Gr(k\D_\bullet), d)$ has  the structure of a unital graded associative ring, with the unit given by the class of the constant module $\underline{k}$. Furthermore, this structure is natural with respect to inclusion of semi-simplicial subobjects and functors that induce semi-simplicial maps between the corresponding objects. 
\end{Thm}
\begin{proof}
Associativity of the pairing and naturality of the product are clear from the construction. Similarly, the class of the constant object $[\underline{k}]
\in\Gr(k\D_0)$ is clearly a two sided unit for this product. Thus it remains to show that  $\cup$  is a cochain map, namely that the diagram
\[\xymatrix{
\Gr(k\D_p)\otimes\Gr(k\D_q)\ar[rrrr]^(.7)\mu\ar[d]^{(d^p\otimes 1)\top ((-1)^p1\otimes d^q)}&&&& \Gr(k\D_{p+q})\ar[d]^{d}\\
{\left(\Gr(k\D_{p+1})\otimes\Gr(k\D_q)\right)\oplus\left(\Gr(k\D_{p})\otimes\Gr(k\D_{q+1})\right)}\ar[rrrr]^(.7){\mu_l + \mu_r} &&&&\Gr(k\D_{p+q+1}),
}\]
where $\mu_l$ and $\mu_r$ are the multiplication maps on the left and right summands, respectively, commutes for all $p, q\ge 0$.

Let $M\in k\D_p\mod$ and $N\in k\D_q\mod$ be any modules. Then $\mu([M]\otimes[N])$ is in the class of a module $P\in k\D_{p+q}\mod$, which takes the value $M([u_1|\cdots|u_p])\otimes N([u_{p+1}|\cdots|u_{p+q}])$ on the object $[u_1|\cdots|u_{p+q}]$. Then the $i$-th summand of $d^{p+q}[P]$ is $(-1)^i\delta^i[P]$, where $\delta^i[P]$ is represented by the module $Q_i$ defined by  
\begin{multline}\label{Eq:Prod-Chn}
Q_i([u_1|\cdots|u_{p+q+1}])   = 
\begin{cases}
	M([u_2|\cdots|u_{p+1}])\otimes N([u_{p+2}|\cdots|u_{p+q+1}]) & i=0\\
	M([u_1|\cdots|u_{i+1}u_i|\cdots|u_{p+1}])\otimes N([u_{p+2}|\cdots|u_{p+q+1}]) & 0 < i \le  p\\
	M([u_1|\cdots|u_{p}])\otimes N([u_{p+1}|\cdots|u_{i+1}u_i|\cdots|u_{p+q+1}]) & p < i \le  p+q\\
						M([u_1|\cdots|u_{p}])\otimes N([u_{p+1}|\cdots|u_{p+q}]) & i = p+q+1
						\end{cases}
\end{multline}
On the other hand, with the usual sign convention of a differential on a tensor product of cochain complexes,
\[d^{p+q}([M]\otimes[N]) = d^p([M])\otimes[N] + (-1)^p[M]\otimes d^q([N]).\]
Let $U_i$ represent the module $\delta^i[M]$, $0\le i\le p+1$ and let $V_{p+j}$ represent $\delta^j[N]$, $0\le j\le q+1$. Then
\[U_i([u_1|\cdots|u_{p+1}]) = 
\begin{cases}
	M([u_2|\cdots|u_{p+1}]) & i=0\\
	M([u_1|\cdots|u_{i+1}u_i|\cdots|u_{p+1}]) & 0<i\le p\\
	M([u_1|\cdots|u_{p}]) & i=p+1
\end{cases}
\] 
Similarly,
\[V_{p+j}([u_{p+1}|\cdots|u_{p+q+1}]) =
\begin{cases}
	N([u_{p+2}|\cdots|u_{p+q+1}]) & j=0\\
	N([u_{p+1}|\cdots|u_{j+p+1}u_{j+p}|\cdots|u_{p+q+1}]) & 0 < j\le q\\
	N([u_{p+1}|\cdots|u_{p+q}]) & j=q+1
\end{cases} 
\]
Apply $\mu$ to the summands of $d^{p+q}([M]\otimes[N]) $ that reside in $\Gr(k\D_{p+1})\otimes\Gr(k\D_q)$.  
\begin{multline}\label{Eq:Prod-Chn-2}
(U_i\otimes N)([u_1|\cdots|u_{p+q+1}])  = U_i([u_1|\cdots|u_{p+1}])\otimes N([u_{p+2}|\cdots|u_{p+q+1}]) =\\
\begin{cases}
	M([u_2|\cdots|u_{p+1}])\otimes N([u_{p+2}|\cdots|u_{p+q+1}]) & i=0\\
	M([u_1|\cdots|u_{i+1}u_i|\cdots|u_{p+1}])\otimes N([u_{p+2}|\cdots|u_{p+q+1}]) & 0<i\le p\\
	M([u_1|\cdots|u_{p}])\otimes N([u_{p+2}|\cdots|u_{p+q+1}]) & i=p+1
\end{cases}
\end{multline}
For $0\le i\le p$ these coincide with Equation (\ref{Eq:Prod-Chn}) and will appear with the same sign. For $i=p+1$ we have a summand $(-1)^{p+1}M([u_1|\cdots|u_{p}])\otimes N([u_{p+2}|\cdots|u_{p+q+1}])$.

Similarly, for the summands that reside in $\Gr(k\D_{p})\otimes\Gr(k\D_{q+1})$, we have 
\begin{align*}
(M\otimes V_{p+j})([u_1|\cdots|u_{p+q+1}]) & = M([u_1|\cdots|u_{p}])\otimes V_{p+j}([u_{p+1}|\cdots|u_{p+q+1}]) =\\
&\begin{cases}
	M([u_1|\cdots|u_{p}])\otimes N([u_{p+2}|\cdots|u_{p+q+1}]) & j=0\\
	M([u_1|\cdots|u_{p}])\otimes N([u_{p+1}|\cdots|u_{j+p+1}u_{j+p}|\cdots|u_{p+q+1}]) & 0 < j\le q\\
	M([u_1|\cdots|u_{p}])\otimes N([u_{p+1}|\cdots|u_{p+q}]) & j=q+1
\end{cases} 
\end{align*}
The summand corresponding to $j=0$ is $M([u_1|\cdots|u_{p}])\otimes N([u_{p+2}|\cdots|u_{p+q+1}])$ and will appear with a sign $(-1)^{p+2}$, and so will cancel the  summand from  (\ref{Eq:Prod-Chn-2}) that corresponds to $i=p+1$. The other summands with the correct signs are also seen easily to coincide with corresponding terms in (\ref{Eq:Prod-Chn}). This completes the proof that the $\cup$ is a cochain map. 
\end{proof}

The following corollary is immediate.

\begin{Cor}\label{Cor:Reduction}
Let $\C$ be a small category. Then the map 
\[\GG^*(\C, k)\to H^*(|C|, \Z)\] is a ring homomorphism.  Thus all homomorphisms in Diagram (\ref{Diag:RH*-diag}) are ring homomorphisms.
\end{Cor}
%\fi

%%%%%%%%%%%%%%%%%%%%%%%%%%%%%%%%%%%%%%%%%%%
%%%%%%%%%%%%%%%%%%%%%%%%%%%%%%%%%%%%%%%%%%%
%%%%%%%%%%%%%%%%%%%%%%%%%%%%%%%%%%%%%%%%%%%
%%%%%%%%%%%%%%%%%%%%%%%%%%%%%%%%%%%%%%%%%%%
%%%%%%%%%%%%%%%%%%%%%%%%%%%%%%%%%%%%%%%%%%%
%%%%%%%%%%%%%%%%%%%%%%%%%%%%%%%%%%%%%%%%%%%
%%%%%%%%%%%%%%%%%%%%%%%%%%%%%%%%%%%%%%%%%%%
%%%%%%%%%%%%%%%%%%%%%%%%%%%%%%%%%%%%%%%%%%%
%%%%%%%%%%%%%%%%%%%%%%%%%%%%%%%%%%%%%%%%%%%
%%%%%%%%%%%%%%%%%%%%%%%%%%%%%%%%%%%%%%%%%%%
%%%%%%%%%%%%%%%%%%%%%%%%%%%%%%%%%%%%%%%%%%%

\section{\texorpdfstring{$\widetilde{\EE}^*$ and $\widetilde{\GG}^*$ cohomology}{widetilde E and G cohomology}}
\label{Sec:Cohomology-EG}

In this section we show that  $\widetilde{\EE}^*(\C, k)$ is acyclic for any small category $\C$ and any field $k$. We compute $\widetilde{\EE}^0(\C, k)$ and show that it is isomorphic to $\widetilde{\GG}^0(\C, k)$. Assuming $\C$ is a direct category, we show that   $\widetilde{\GG}^*(\C, k)$ in positive dimensions is essentially the integral cohomology  of the nerve of $\C$. 

\subsection{\texorpdfstring{The groups $\widetilde{\EE}^0(\C, k)$, $\widetilde{\GG}^0(\C, k)$}{The groups widetilde E0 and G0}}
\label{SSec:E-G-low}
The main aim of this subsection is to show that $\widetilde{\EE}^0(\C, k)\cong \widetilde{\GG}^0(\C, k)$  is the free abelian group generated by the indecomposable locally constant  $k\C$-modules  and conclude with the proof of Theorem \ref{Th:E0G0}. The following three statements are preparatory.

\begin{Lem}\label{Lem:delta-Loc-Const}
Let $\C$ be a small category and let $M$ be either a $k\widetilde{\E}^\C_n$-module or a $k\E^\C_n$-module that is locally constant. Then, for every $0\le i\le n$, $\delta^i[M] = \delta^{i+1}[M]$. Therefore, $d^n[M] = 0$ if $n$ is even and $d^n[M] = \delta^n[M]$ if $n$ is odd. Furthermore, if $M\in k\widetilde{\G}^\C_0\mod$  is locally constant, then $\delta^0[M] = \delta^1[M]$.
\end{Lem}
\begin{proof}
Let $[v_1|\cdots|v_{n+1}]\in\widetilde{\E}^\C_{n+1}$ be any object. Then for $0<i\le n$ there is a morphism 
\[\partial_{i+1}[v_1|\cdots|v_{n+1}]\to \partial_i[v_1|\cdots|v_{n+1}]\]
 in $\widetilde{\E}^\C_{n}$ that takes the form
\[\xymatrix{
a_0\ar[r]^{v_1}\ar@{=}[d] & a_1\ar[r]^{v_2}\ar@{=}[d]& \cdots &\ar[r]^{v_{i-1}}& a_{i-1}\ar[rr]^{v_{i}}\ar@{=}[d]&& a_{i}\ar[rr]^{v_{i+2}\circ v_{i+1}}\ar[d]^{v_{i+1}} && a_{i+2}\ar[r]\ar@{=}[d]&\cdots & \ar[r] & a_{n+1}\ar@{=}[d]\\
a_0\ar[r]^{v_1} & a_1\ar[r]^{v_2}& \cdots &\ar[r]^{v_{i-1}} & a_{i-1}\ar[rr]^{v_{i+1}\circ v_{i}}&& a_{i+1}\ar[rr]^{v_{i+2}} && a_{i+2}\ar[r]&\cdots & \ar[r] & a_{n+1}\\
}.\]
Since $M$ is locally constant, applying it to this morphism we obtain an isomorphism 
\[\partial_{i+1}^*M([v_1|\cdots|v_{n+1}])\xto{\cong} \partial_i^*M([v_1|\cdots|v_{n+1}]).\]
For $i=0$ one has 
\[\xymatrix{
a_0\ar[rr]^{v_2\circ v_1}\ar[d]^{v_1} && a_2\ar[r]^{v_2}\ar@{=}[d]& \cdots  \ar[r] & a_{n+1}\ar@{=}[d]\\
a_1\ar[rr]^{v_2} && a_2\ar[r]^{v_2}& \cdots  \ar[r]& a_{n+1}\\
}\]
that gives an isomorphism 
\[\partial_{1}^*M([v_1|\cdots|v_{n+1}])\xto{\cong} \partial_0^*M([v_1|\cdots|v_{n+1}]).\]
Similarly, for $i=n$, the morphism
\[\xymatrix{
a_0\ar[r]^{v_1}\ar@{=}[d] & a_1\ar[r]^{v_2}\ar@{=}[d]& \cdots  \ar[r] & a_{n-1}\ar[rr]^{v_n}\ar@{=}[d] && a_{n}\ar[d]^{v_{n+1}}\\
a_0\ar[r]^{v_1} &                 a_1\ar[r]^{v_2}                 & \cdots  \ar[r] & a_{n-1}\ar[rr]^{v_{n+1}\circ v_n}  && a_{n+1}\\
}\]
gives an isomorphism
\[\partial_{n+1}^*M([v_1|\cdots|v_{n+1}])\xto{\cong} \partial_n^*M([v_1|\cdots|v_{n+1}]).\]
These isomorphisms are clearly natural with respect to morphisms in $\widetilde{\E}^\C_n$, and so they define natural isomorphisms $\partial_{i+1}^*M\xto{\cong}\partial^*_iM$ for all $0\le i\le n$, so $\delta^{i+1}[M] = \delta^i[M]$. 
Furthermore, since  $\E^\C_n$ is a full subcategory of $\widetilde{\E}^\C_n$,  the same argument applies, and so the statement holds for any $M\in k\E^\C_n\mod$ as well.

Finally, if $M\in k\widetilde{\G}^\C_0\mod$ is locally constant, then for each object $[v]\in\widetilde{\G}^\C_1$, where $a\xto{v} b$, one has an isomorphism 
\[M(a) = M(\partial_1[v]) \xto{M(v)} M(\partial_0[v]) = M(b).\]
Define a natural isomorphism of $k\widetilde{\G}_1^\C$-modules $\alpha\colon M\circ\partial_1 \to M\circ\partial_0$ by $\alpha_{[v]} = M(v)$. If $[u]$ is another object in $\widetilde{\G}^\C_1$, where $u\colon c\to d$, then a morphism $[v]\to[u]$ is determined by a morphism $w\colon b\to c$ in $\C$. Hence, there is  a commutative square
\[\xymatrix{
a \ar[r]^{v}\ar[d]_{w\circ v} & b\ar[d]^{u\circ w}\\
c\ar[r]^u & d
}\]
and   
\[M(u\circ w)\circ \alpha_{[v]}  = \alpha_{[u]}\circ M(w\circ v).\]
Thus $\alpha$ is a natural isomorphism and the proof is complete.
\end{proof}

\begin{Cor}\label{Cor:delta-Loc-Const}
Let $\C$ be a small category and let $M, N$ be either $k\widetilde{\E}^\C_n$-modules or  $k\E^\C_n$-modules. Let $X = [M]-[N]$ be an element in the corresponding Grothendieck group.  Write $M\cong M_0\oplus M_1$, and $N\cong N_0\oplus N_1$, where $M_0$ and $N_0$ are maximal locally constant summands in $M$ and $N$ respectively. Set $X_i = [M_i]- [N_i]$, for $i = 0, 1$. Then for any $n\geq 0$
\[d^n(X) = \begin{cases}
				d^n(X_1) & n\;\text{even}\\
				d^n(X_1) + \delta^n(X_0) & n\;\text{odd}
		\end{cases}
\]
In particular, if $X$ is an $n$-cocycle for $n$ odd, then $M_0\cong N_0$. Furthermore, if  $X\in\Gr(k\widetilde{\G}^\C_0)$ is any element, then $d^0(X) = d^0(X_1)$.
\end{Cor}
\begin{proof}
The first part of the corollary and its last statement follow at once from Lemma \ref{Lem:delta-Loc-Const}. Thus, let $n$ be odd and  assume that $X$ is a cocycle. 
Since $d^n([M]-[N]) = 0$, one has by Lemma \ref{Lem:delta-Loc-Const}
\[d^n[M] = \delta^n([M_0])+d^n([M_1]) = \delta^n([N_0]) +d^n([N_1]) = d^n[N].\]
Notice that $d^n[M_1]$, written in terms of its unique decomposition as a linear combination of indecomposable modules, of cannot contain a locally constant summand, since if it does, then at least one summand of the form $\delta^i[M_1]$ must contain a locally constant summand. But then $[M_1]= \sigma^i\delta^i[M_1]$ contains a locally constant summand, where $\sigma^i$ is the $i$-th codegeneracy operator, thus contradicting maximality of $M_0$.  Similarly, $d^n[N_1]$ does not contain locally constant summands. 
It follows that $d^n[M_1] = d^n[N_1]$ and $\delta^n[M_0] = \delta^n[N_0]$ and, by applying $\sigma^n$ to the second equality, that $M_0\cong N_0$.
\end{proof}

\begin{Lem}\label{Lem:delta-indecomp}
Let $\C$ be a small category and    let $M\in k\widetilde{\E}^\C_n\mod$ be an indecomposable module. Then for each $0\le i\le n$, any representative of the element $\delta^i[M] = [\partial_i^*M]\in \Gr(k\widetilde{\E}^\C_{n+1})$ is indecomposable in $k\widetilde{\E}^\C_{n+1}\mod$. Furthermore, the same claim  holds for any indecomposable module $M\in k\widetilde{\G}^\C_0\mod$.
\end{Lem}
\begin{proof}
Fix an integer $0\le i\le n+1$ and suppose that $\delta^i[M]=[P] + [Q]$, where $P, Q\neq 0$ in $k\widetilde{\E}^\C_{n+1}\mod$. 
Then,
\[
[M]=(\sigma^i\circ \delta^i)[M] = \sigma^i[P] +  \sigma^i[Q] = [s_i^*P\oplus s_i^*Q],
\]
where $\sigma^i=s_i^*$ denotes the $i$-th codegeneracy operator. 
Since $M$ is indecomposable, either $s_i^*P=0$ or $s_i^*Q=0$. Without loss of generality, we may assume that $s_i^*Q=0$.   Since $Q\neq 0$, there is a non-zero element $x\in Q([v_1|v_2|\cdots|v_{n+1}])$ for some $\underline{v} = [v_1|v_2|\cdots|v_{n+1}]\in\widetilde{\E}^\C_{n+1}$, where $v_i\colon a_{i-1}\to a_i$ is a morphism in $\C$. Fix a natural isomorphism
$\eta\colon P\oplus Q\to \partial_i^*M$,  and consider the element 
\[(0,x)\in P(\underline{v})\oplus Q(\underline{v}).\] 
Let
\[y \defeq \eta_{\underline{v}}(0,x)\in \partial_i^*M(\underline{v}) = 
\begin{cases}
M([v_2|\cdots|v_{n+1}]) & i=0\\
M([v_1|v_2|\cdots|v_{n}]) & i=n+1\\
M([v_1|v_2|\cdots|v_{i+1}v_i|\cdots|v_{n+1}]) & \text{otherwise}
\end{cases}
\]

For $0<i<n+1$, 
\begin{align*}
&\partial_i^*M\left([v_1|v_2|\cdots|v_{n+1}]\xto{(1_{a_0},1_{a_1},\ldots, 1_{a_{i-1}}, v_{i+1}, 1_{a_{i+1}},\ldots, 1_{a_{n+1}})}[v_1|v_2|\cdots|v_{i+1}v_i|1_{a_{i+1}}|\cdots|v_{n+1}]\right)(y) = \\
&M\left([v_1|\cdots|v_{i-1}|v_{i+1}v_i|\cdots|v_{n+1}]\xto{(1_{a_0},1_{a_1},\ldots, 1_{a_{i-1}},  1_{a_{i+1}},\ldots, 1_{a_{n+1}})}
[v_1|\cdots|v_{i-1}|v_{i+1}v_i|\cdots|v_{n+1}]\right)(y) = y\neq 0.\end{align*}
But, 
\begin{align*}
0\neq \eta^{-1}_{[v_1|v_2|\cdots|v_{i+1}v_i|1_{a_{i+1}}|\cdots|v_{n+1}]}(y) = &\left(P\oplus Q\right)(1_{a_0},1_{a_1},\ldots, 1_{a_{i-1}}, v_{i+1}, 1_{a_{i+1}},\ldots, 1_{a_{n+1}})(0,x)   = \\ 
&\left(0, Q(1_{a_0},1_{a_1},\ldots 1_{a_{i-1}}, v_{i+1}, 1_{a_{i+1}},\ldots, 1_{a_{n+1}})(x)\right).
\end{align*}
Thus $Q(1_{a_0},1_{a_1},\ldots 1_{a_{i-1}}, v_{i+1}, 1_{a_{i+1}},\ldots, 1_{a_{n+1}})(x)\neq 0$ in  
\[Q([v_1|v_2|\cdots|v_{i+1}v_i|1_{a_{i+1}}|\cdots|v_{n+1}]) = s_i^*Q([v_1|v_2|\cdots|v_{i+1}v_i|\cdots|v_{n+1}]) = 0\] and we obtain a contradiction.  The argument is very similar for $i=0$ and will be omitted.

Finally, for $i=n+1$,  a morphism $ [v_1|v_2|\cdots|v_n|v_{n+1}] \to [v_1|v_2|\cdots|v_{n}|1_{a_n}]$ needs not exist, but there is a morphism the other way. Thus
consider the map
\begin{align*}\partial_{n+1}^*M\left([v_1|v_2|\cdots|v_{n}|1_{a_n}]\xto{(1_{a_0},1_{a_1},\ldots,1_{a_{n}}, v_{n+1})}[v_1|v_2|\cdots|v_n|v_{n+1}]\right) = \\
M\left([v_1|\cdots|v_{n}]\xto{(1_{a_0},\ldots,1_{a_{n}})}[v_1|\cdots|v_{n}]\right).\end{align*}
Since $y\in \partial_{n+1}^*M([v_1|v_2|\cdots|v_{n+1}])=\partial_{n+1}^*M([v_1|v_2|\cdots|v_{n}|1_{a_n}])$, and the map above is the identity, we have
\[\eta_{[v_1|v_2|\cdots|v_{n}|1_{a_n}]}^{-1}(y) = (x',x'')\in (P\oplus Q)([v_1|v_2|\cdots|v_{n}|1_{a_n}])\]
By naturality, $(P\oplus Q)(1_{a_0},1_{a_1},\ldots,1_{a_{n}}, v_{n+1})$ is an isomorphism, so 
\[(P\oplus Q)(1_{a_0},1_{a_1},\ldots,1_{a_{n}}, v_{n+1})(x',x'') = (0,x).\] 
In particular, $Q(1_{a_0},1_{a_1},\ldots,1_{a_{n}}, v_{n+1})(x'') = x\neq 0$.
But, once more, $x\in Q([v_1|v_2|\cdots|v_{n}|1_{a_n}]) = s_n^*Q([v_1|v_2|\cdots|v_n]) = 0$ and we obtain a contradiction.
 Hence $\delta^iM$ is indecomposable for all $0\le i \le n+1$ and the claim follows for $k\widetilde{\E}^\C_\bullet$. 

It remains to prove the last statement of the lemma. Let $M\in k\C\mod = k\widetilde{\G}^C_0\mod$ be an indecomposable module. We must show that $\partial_i^*M$ is indecomposable in $k\widetilde{\G}^C_1\mod$ for $i=0,1$. But for any morphism $a\xto{v}b$ in $\C$, one has the morphisms $[v]\xto{(v, 1_b)}[1_b]$ and $[1_a]\xto{(1_a,v)}[v]$ in $\widetilde{\G}^\C_1$ induced by $1_b$ and $1_a$ respectively. Thus one has
\[\partial_0^*M\left([v]\xto{(v, 1_b)}[1_b]\right) = M(1_b), \quad\text{and}\quad \partial_1^*M\left([1_a]\xto{(1_a, v)}[v]\right) = M(1_a).\]
The proof then proceeds in exactly the same way as the general case for $k\widetilde{\E}^\C_n$.
\end{proof}

We note that Lemma \ref{Lem:delta-indecomp} does not hold in $\widetilde{\G}^\C_n$  for $n>0$, as the following example shows. 

\begin{Ex}\label{Ex:delta-not-indecomp}
Let $\P$ be the poset $0<1$. The category $\widetilde{\G}^\P_1$ has objects $\{00, 01, 11\}$ and the obvious morphisms between them. Let $S_{01}$ denote the simple module on $\widetilde{\G}^\P_1$ that takes the value $k$ on $01$ and $0$ elsewhere. Then 
$\delta^1[S_{01}] = [S_{01}\circ\partial_1]$, and the latter takes the value $k$ on $001$ and $011$, and $0$ on all other objects. But $001$ and $011$ are not comparable in $\widetilde{\G}^\P_2$ (they are in $\widetilde{\E}^\P_2$), so 
\[S_{01}\circ\partial_1 \cong S_{001} \oplus S_{011},\]
and hence $\delta^1[S_{01}]$ is not represented by an indecomposable module.
\end{Ex}

\begin{Lem}\label{Lem:Loc-Const}
Let $\C$ be a small category and let $\D_\bullet$ denote either $\widetilde{\E}^\C_\bullet$ or $\widetilde{\G}^\C_\bullet$.
 Let  $M, N\in k\D_0\mod$ be modules that satisfy $\delta^0[M] = \delta^1[N]$. Then $M$ and $N$ are locally constant. 
\end{Lem}
\begin{proof}
Notice that $\widetilde{\E}^\C_0 = \widetilde{\G}^\C_0 = \C$. We prove the statement first with respect to $\delta^i$ on $\widetilde{\E}^\C_0$, and then observe that the same argument works to prove it with respect to $\delta^i$ on $\widetilde{\G}^\C_0$. 

Let $\alpha\colon \delta^0M\rightarrow \delta^1N$ be a natural isomorphism, and let $[v]\in \widetilde{\E}^\C_1$ be an object, where $v\colon a\to b$ is a morphism in $\C$. Then one has a commutative diagram 
\[
\xymatrix{
\delta^0M([1_a]) \ar[d]^{\alpha_{1_a}}_{\cong} \ar[rr]^{\delta^0M(1_a,v)}&& \delta^0M([v]) \ar[d]^{\alpha_{v}}_{\cong}\ar[rr]^{\delta^0M(v,1_b)} && \delta^0M([1_b]) \ar[d]^{\alpha_{1_b}}_{\cong}\\
\delta^1N([1_a]) \ar[rr]^{\delta^1N(1_a,v)}&& \delta^1N([v])\ar[rr]^{\delta^1N(v,1_b)} && \delta^1N([1_b])
}
\] 
By definition, the same diagram can be written as
\[
\xymatrix{
M(a) \ar[d]^{\alpha_{1_a}}_{\cong} \ar[rr]^{M(v)}&& M(b) \ar[d]^{\alpha_{v}}_{\cong}\ar[rr]^{\Id} && M(b) \ar[d]^{\alpha_{1_b}}_{\cong}\\
N(a) \ar[rr]^{\Id}&& N(a) \ar[rr]^{N(v)} && N(b). }
\] 
It follows that both $M(v)$ and $N(v)$ are isomorphisms, and since $v$ is an arbitrary morphism in $\C$, $M$ and $N$ are both locally constant. 

To prove the corresponding statement for $\widetilde{\G}^\C_\bullet$, notice that the morphisms of the form $[1_a]\xto{(1_a, v)}[v]$ and $[v]\xto{(v, 1_b)}[1_b]$ exist in $\widetilde{\G}^\C_1$ for each $v\colon a\to b$ in $\C$, as observed in the proof of Lemma \ref{Lem:delta-indecomp}. 
\end{proof}

We are now ready  to show that 
$\widetilde{\EE}^0(\C, k)$ and $\widetilde{\GG}^0(\C, k)$ are generated by isomorphism classes of locally constant $k\C$-modules.

\begin{Prop}\label{Prop:G0}
Let $\C$ be a small category. Then $\widetilde{\EE}^0(\C, k)$ is the free abelian group generated by all isomorphism classes of indecomposable locally constant modules in $k\C\mod$. Furthermore, the inclusion $\widetilde{\G}^\C_\bullet\to\widetilde{\E}^\C_\bullet$ induces an isomorphism $\widetilde{\EE}^0(\C, k)\to \widetilde{\GG}^0(\C, k)$.
\end{Prop}
\begin{proof}
Let $X = [M]-[N]\in\Gr(k\widetilde{\E}^\C_0)$ be a $0$-cocycle and assume that neither $M$ nor $N$ is locally constant. Since $X$ is a cocycle we have the relation $\delta^0[M]+\delta^1[N] = \delta^0[N]+\delta^1[M]$. By Corollary \ref{Cor:delta-Loc-Const}, we may assume that neither $M$ nor $N$ contain any locally constant summand. 
Write $M$ and $N$ as sums of indecomposable modules
\[
[M]=[M_1]+\cdots+[M_m], \quad \text{and} \quad [N]=[N_1]+\cdots+[N_n],
\]
where $M_1,\ldots,M_m,N_1,\ldots,N_n$ are indecomposable.   Then by hypothesis
\[
\delta^0[M_1]+\cdots+\delta^0[M_m]+\delta^1[N_1]+\cdots+\delta^1[ N_n] = \delta^0[N_1]+\cdots+\delta^0[N_n]+\delta^1[M_1]+\cdots+\delta^1[ M_m].
\]
By Lemma \ref{Lem:delta-indecomp}, all  summands on both sides of this equation are classes of indecomposable modules, and since the decomposition is unique, both sides of the equation must have the same indecomposable terms. By Lemma \ref{Lem:Loc-Const} and our assumption that  $M$ and $N$ do not contain any locally constant summand, the module $\delta^0[M_1]\neq\delta^1[M_i]$ for all $i\in\{1,\ldots,m\}$. Thus, $\delta^0[M_1]=\delta^0[N_{j}]$ for some $j\in \{1,\ldots,n\}$. Proceeding inductively, we get $m=n$ and
\[
\delta^0[M]=\delta^0[M_1]+\cdots+\delta^0[M_m] = \delta^0[N_1]+\cdots+\delta^0[N_n] = \delta^0[N].
\]
Applying the codegeneracy  $\sigma^0$ on both sides, we get $[M]=[N]$.
Hence $X = 0$ and it follows that the only $0$-cocycles are differences of locally constant modules. This proves the claim for $\widetilde{\EE}^0(\C, k)$.

Since Corollary \ref{Cor:delta-Loc-Const} and Lemmas \ref{Lem:delta-indecomp} and \ref{Lem:Loc-Const} apply also to $\Gr(k\widetilde{\G}^\C_0)$, it follows that $\widetilde{\GG}^0(\C, k)$ is also generated by equivalence classes of locally constant modules, proving the second claim.
\end{proof}

As a corollary we obtain an explicit identification of $\widetilde{\EE}^0(\C, k)$. 

\begin{Prop}\label{Prop:E0}
Let $\C$ be a small connected category. Then
\[\widetilde{\GG}^0(\C, k)\cong \widetilde{\EE}^0(\C, k) \cong \Gr(k\Gamma),\]
where $\Gamma = \pi_1(|\C|)$.
\end{Prop}
\begin{proof}
Let $\C_0$ denote the localisation of $\C$ at the full set of its morphisms. Thus $\C_0$ is the groupoid associated with $\C$. By the universal property of localisation, for any locally constant $M\in k\C\mod$, there exists a unique $M_0\in k\C_0\mod$ that factors $M$. In other words, there is a $1-1$ correspondence between isomorphism classes of locally constant $k\C$-modules and isomorphism classes of $k\C_0$-modules, and this correspondence is induced by the localisation projection. Thus the homomorphism
\[\widetilde{\EE}^0(\C_0, k) \to \widetilde{\EE}^0(\C, k)\]
is an isomorphism. 

By Proposition \ref{Prop:Groupoids}, $\widetilde{\EE}^0(\C_0, k)\cong \Gr(k\Aut_{\C_0}(x))$ for any object $x\in \C$, and by \cite[Proposition 1]{Qu}, $\pi_1(|\C|) \cong \Aut_{\C_0}(x)$. 
\end{proof}

%%%%%%%%%%%%%%%%%%%%%%%%%%%%%%%%%%%%%%%%%%%%%%%%%%%%%%
%%%%%%%%%%%%%%%%%%%%%%%%%%%%%%%%%%%%%%%%%%%%%%%%%%%%%%

\subsection{\texorpdfstring{Higher-dimensional $\widetilde{\EE}^*$ and $\widetilde{\GG}^*$ cohomology}{Higher dimensional widetilde E and G cohomology}}
\label{SSec:Higher-Cohomology-EG-tilde}

Having dealt with $\widetilde{\EE}^0$ and $\widetilde{\GG}^0$, in this section we  consider the higher-dimensional cohomology  groups. The main results, Theorems \ref{Thm:E0G0} and \ref{Thm:Higher-G} are that $\widetilde{\EE}^*(\C, k)$ is acyclic for any small connected category $\C$ and that $\widetilde{\GG}^n(\C, k)$ coincides with $H^n(|\C|, \Z)$ for all small connected and direct categories $\C$ and all $n\geq 2$.

\subsubsection{\texorpdfstring{Acyclicity of  $\,\widetilde{\EE}^*$}{Acyclicity of widetilde E}}
\label{SSSec:Higher-Cohomology-E-tilde}

We now show that the higher cohomology groups $\widetilde{\EE}^i(\C,k)$ vanish for all small categories $\C$. To prove this some preparation is required. 

Let $U^\bullet$ be a  cosimplicial abelian group. Recall the \hadgesh{Moore cochain complex}  $(N^\bullet(U), d^\bullet)$ defined by $N^0(U) \defeq U^0$ and for $n\geq 1$
\[N^n(U)\defeq \coKer\left(\bigoplus_{i=0}^{n-1} U^{n-1} \xto{\sum_{i=0}^{n-1}\delta^i} U^n\right),\] 
with $d^n = (-1)^{n}\delta^{n+1}\colon N^n(U)\to N^{n+1}(U)$. It is well known that the standard cohomology of $U^\bullet$ is isomorphic to that of $(N^\bullet(U), d^\bullet)$ (See for instance \cite[Sec. III.2]{GJ}, or \cite[Sec. 8.4]{Wei}). 

\begin{Lem}
\label{Lem:coboundary-rels}
Let $U^\bullet$ be a cosimplicial abelian group, and let $X, Y\in U^n$ be elements, where $n\geq 1$. Suppose that $\delta^i X = \delta^j Y$ for some $0\le i< j\le n+1$. Then the following statements hold.
\begin{enumerate}[(1)]
\item If $i=j-1$, then $X=Y$ and $X \in \Ima(\delta^i)$.
\label{Lem:coboundary-rels-1}

\item If $i<j-1$, then  $X \in \Ima(\delta^{j-1})$ and $Y \in \Ima(\delta^i)$.
\label{Lem:coboundary-rels-2}
\end{enumerate}
\end{Lem}
\begin{proof}
If $i=j-1$, then applying $\sigma^i$ to $\delta^{i}X = \delta^{i+1}Y$, we get $X=Y$. Also 
\[X=\sigma^{i+1}\delta^{i+1}X = \sigma^{i+1}\delta^iX = \delta^i\sigma^iX.\] This proves Part \ref{Lem:coboundary-rels-1}.
If $i<j-1$, then 
\[X = \sigma^i\delta^iX = \sigma^i\delta^jY = \delta^{j-1}\sigma^iY,\quad\text{and}\quad
Y = \sigma^j\delta^jY = \sigma^j\delta^iX = \delta^i\sigma^{j-1}X.\]
This proves Part \ref{Lem:coboundary-rels-2}.
\end{proof}

%%%

\begin{Prop}\label{Prop:Acyclicity}
Let $\C$ be a small category. Then the cosimplicial object $\Gr(k\widetilde{\E}^\C_\bullet)$ is acyclic.
\end{Prop}
\begin{proof}
Let $U^\bullet$ denote $\Gr(k\widetilde{\E}^\C_\bullet)$ for short, and let $(N^\bullet(U), d^\bullet)$ denote the Moore cochain complex associated to $U^\bullet$. Let $\widebar{Z}\in N^n(U)$ be a cocycle, represented by $Z \in U^n$.
Write 
\[Z = \sum_{k=1}^r[M_k] - \sum_{l=1}^s[N_l],\]
where $M_k, N_l\in k\widetilde{\E}^\C_n\mod$ are non-zero and indecomposable. We may assume that the sets $\{M_k\}_{k=1}^r$ and $\{N_l\}_{l=1}^s$ have no modules that are isomorphic. We refer to $r+s$ as the  \hadgesh{decomposition length of $Z$}, and this is well defined by our assumption.  

Since  $d^n\widebar{Z}= \widebar{\delta^{n+1}Z} = 0$ by assumption, there exist elements $\{Y_i\}_{i=0}^{n}$ in $N^n(U)$, such that $\delta^{n+1}Z = \sum_{i=0}^{n} \delta^iY_i$. By Lemma \ref{Lem:delta-indecomp}, $\delta^{n+1}[M_1]$ is an indecomposable module. Hence it is represented by an indecomposable summand $T_1$ of some $Y_j$, such that $\delta^{n+1}[M_1] = \delta^i[T_1]$. There are two possibilities: if $i=n$, then, by  Lemma \ref{Lem:coboundary-rels}\ref{Lem:coboundary-rels-1}, $[M_1] = [T_1]\in\Ima(\delta^{n})$, while if $i<n$, then, by  Lemma \ref{Lem:coboundary-rels}\ref{Lem:coboundary-rels-2},  $[M_1]\in\Ima(\delta^i)$ and $[T_1]\in\Ima(\delta^{n})$. In the first case,  $\widebar{[M_1]}$ is a coboundary in $N^n(U)$. In the second case, since $i<n$, $\widebar{[M_1]}=0$ in $N^n(U)$. Thus $Z-[M_1]$ is a cocycle that represents the same cohomology class as $Z$, and is of decomposition length $r+s-1$. By downward induction on the decomposition length of $Z$, each of its indecomposable summands is either a coboundary or $0$ in $N^n(U)$, and hence $\widebar{Z}$ itself is a coboundary and the proof is complete.
\end{proof}

We are now ready to complete the proof of Theorem \ref{Th:E0G0}, which we restate below.

\begin{Thm}\label{Thm:E0G0}
Fix a field $k$ and $\C\in\Cat$. Then 
\begin{enumerate}[(1)]
\item The  group $\widetilde{\EE}^0(\C, k)$ is a free abelian group generated by the locally constant indecomposable modules in $k\C\mod$. \label{Thm:E0G0-1}

\item The natural map $\widetilde{\EE}^0(\C,k)\to\widetilde{\GG}^0(\C,k)$ is an isomorphism. \label{Thm:E0G0-2}

\item The cohomology theory $\widetilde{\EE}^*(-, k)$ is acyclic and 
\[\widetilde{\GG}^0(\C, k)\cong \widetilde{\EE}^0(\C, k) \cong \Gr(k\Gamma),\]
where $\Gamma = \pi_1(|\C|)$. \label{Thm:E0G0-3}
\end{enumerate}
\end{Thm}
\begin{proof}
Parts \ref{Thm:E0G0-1}  and \ref{Thm:E0G0-2} are  Proposition \ref{Prop:G0}, and Part \ref{Thm:E0G0-3} is Propositions \ref{Prop:E0} and \ref{Prop:Acyclicity}.
\end{proof}

We obtain a few immediate corollaries.

\begin{Cor}\label{Cor:Any-Cx}
Let $X$ be a simplicial complex and let $\F X$ denote its face poset, considered as a category. Let $\Gamma_X\defeq \pi_1(|X|)$. Then $\widetilde{\EE}^0(\F X, k)\cong\Gr(k\Gamma_X)$. In particular, if $X$ is simply connected then $\widetilde{\EE}^0(\F X,k) \cong \Z$.
\end{Cor}
\begin{proof}
Since the geometric realisation of the nerve $|\F X|$  is homotopy equivalent to the geometric realisation of $X$, the claim follows at once from Theorem \ref{Thm:E0G0}.
\end{proof}

\begin{Cor}\label{Cor:E0-trees}
Let $\Q = (V, E)$ be a finite quiver and let $\PP(\Q)$ denote its path category. Then $\widetilde{\EE}^0(\PP(\Q), k) \cong \Gr(k\Gamma)$, where $\Gamma$ is the free group on $N = m+ |E| - |V| = m - \chi(\Q)$ generators, where $m$ is the number of connected components of $\Q$ and $\chi(\Q)$ is its Euler characteristic.
\end{Cor}
\begin{proof}
The fundamental group of the path category of a quiver is the free group of the geometric realisation of the underlying quiver. That, in turn, is well known to be the free group of rank $N = m+ |E| - |V|$. The corollary now follows at once from Theorem \ref{Thm:E0G0}.
\end{proof}

\begin{Ex}\label{Ex:E0-ex}
The following are easy examples of categories for which one obtains an explicit expression for $\widetilde{\EE}^0$. Fix a field $k$.
\begin{enumerate}[(1)]
\item Let $\Q$ be a connected tree and let $\P$ be the poset it generates. Then $\widetilde{\EE}^0(\P, k)\cong \Z$. \label{Ex:E0-ex-1}

\item Let $\Q$ be any quiver whose underlying undirected graph is a cycle. Then $\widetilde{\EE}^0(\PP(\Q),k) \cong \Gr(k[\Z])$.

\item Let $\Q$ be an $(n,m)$ bipartite quiver. Then $\widetilde{\EE}^0(\PP(\Q), k)\cong\Gr(kF_\Q)$, where $F_\Q$ is a free group of rank $nm -n -m +1$.
\end{enumerate}
\end{Ex}

For a finite poset $\P$ and an object $x\in\P$, denote by $\P_{\ge x}$ the sub-poset, whose objects are all $y\in\P$ such that $y\geq x$. Recall that the modules that take the value $k$ on $\P_{\ge x}$ for some $x$ and the value $0$ elsewhere are exactly the indecomposable projective $k\P$-modules.

\begin{Cor}
\label{Cor:Loc-const=constant}
Let $\P$ be a finite  poset that is generated by a connected tree or is the face poset of a simply connected simplicial complex, and let $M\in k\P\mod$ be a locally constant module. Then  $M$  is isomorphic to a finite direct sum of modules of the form $\underline{k}$. In particular, if $\P$ is a poset whose Hasse diagram is a rooted tree, then $M$ is projective.
\end{Cor}
\begin{proof}
By Part \ref{Ex:E0-ex-1} of Example \ref{Ex:E0-ex}, $\widetilde{\EE}^0(\P, k)\cong \Z$ and it is generated by locally constant $k\P$-modules. But $\underline{k}$ is a locally constant module, and it clearly generates a copy of $\Z$ in $\widetilde{\EE}^0(\P, k)$, which implies the general claim. The second statement follows from the fact that the indecomposable projective modules in $k\P\mod$ are those modules that are constant with value $k$ on an overset $\P_{\geq x}$ for some $x\in \P$ and that vanish outside this set.
\end{proof}

%%%%%%%%%%%%%%%%%%%%%%%%%%%%%%%%%%%%%%%%%%%%%%%%%%%%%%
%%%%%%%%%%%%%%%%%%%%%%%%%%%%%%%%%%%%%%%%%%%%%%%%%%%%%%
%%%%%%%%%%%%%%%%%%%%%%%%%%%%%%%%%%%%%%%%%%%%%%%%%%%%%%
%%%%%%%%%%%%%%%%%%%%%%%%%%%%%%%%%%%%%%%%%%%%%%%%%%%%%%
%%%%%%%%%%%%%%%%%%%%%%%%%%%%%%%%%%%%%%%%%%%%%%%%%%%%%%

\subsubsection{\texorpdfstring{$\widetilde{\GG}^*$ cohomology for direct categories}{G-cohomology}}
\label{SSSec:Higher-G-tilde-Cohomology}

We now describe the $\widetilde{\GG}^*(\C, k)$ where   $\C$ is a \hadgesh{direct category}.  For any $n\in\mathbb N$, let $\I(k\widetilde{\G}^\C_n)\le \Gr(k\widetilde{\G}^\C_n)$ denote the subgroup  generated by  interval modules (recall that by Lemma \ref{Lem:Interval-modules-indec}, interval modules are indecomposable). By Lemma \ref{Lem:Inverse-Interval},  $\I(k\widetilde{\G}^\C_\bullet)$  is a cosimplicial subgroup of $\Gr(k\widetilde{\G}^\C_\bullet)$. 

Objects of $\widetilde{\G}_n^\C$ will be written as $\mathbf{v} = [v_1|v_2|\cdots|v_n]$ for short. If $a\in\C$ is any object, then we shall denote by $\mathbf{1_a}$ the object $[1_a|1_a|\cdots|1_a]\in \widetilde{\G}^\C_n$. Objects of the form $\mathbf{1_a}$ will be referred to as \hadgesh{homogeneous objects} of $\widetilde{\G}_n^\C$. In particular, every object of $\widetilde{\G}_0^\C$ is homogeneous by default. If $\alpha\colon \mathbf{v} \to\mathbf{u}$ is any morphism in $\widetilde{\G}_n^\C$, with $a$ the codomain of $v_n$ and $b$ the domain of $u_1$, then $\alpha$ factors through both $\mathbf{1_a}$ and $\mathbf{1_b}$. On the other hand, the morphisms $\mathbf{v}\to \mathbf{1_a}$ and $\mathbf{1_b}\to \mathbf{u}$ are irreducible. It follows that  morphisms of these two types are the only irreducible morphisms in $\widetilde{\G}^\C_n$. For example, in the poset $\widetilde{\G}^{[3]}_1$, the obvious morphism $(0,1)\to (2,3)$ factors as 
\[(0,1)\to(1,1)\to(1,2)\to(2,2)\to(2,3).\]

For an arbitrary small category $\D$ and a module $M\in k\D\mod$, define the \hadgesh{support of $M$} to be the full subcategory  $\supp(M)\subseteq \D$ whose objects are   $d\in\D$ such that $M(d)\neq 0$. For $0\le i \le n$, let $\partial_i\colon \widetilde{\G}^\C_n\to \widetilde{\G}^\C_{n-1}$ denote the $i$-th face operator, as usual. Then for any $M\in k\widetilde{\G}^\C_{n-1}\mod$,
\[
\supp(\partial_i^* M) = \partial_i^{-1}(\supp(M)).
\]
In particular, if $M$ is an interval module and $\supp(\partial_i^* M)$ is connected, then $\partial_i^* M$ is an interval module by Lemma \ref{Lem:Inverse-Interval}. In general, $\supp(\partial_i^* M)$ need not be connected, as we saw in Example \ref{Ex:delta-not-indecomp}.

By definition, the support of an interval module is convex and connected. Below we refer to an interval consisting of a single object $\mathbf{v}$ as a \hadgesh{singleton interval}, or simply a \hadgesh{singleton}. Such singletons will be denoted by $\{\mathbf{v}\}$, where $\mathbf{v}$ is the defining object. 

\begin{Lem}\label{Lem:preim-simples}
Suppose that $S=\{\mathbf{v}\}$ is a singleton  in $\widetilde{\G}^\C_n$, and let $i\in \{0,\ldots,n+1\}$.
\begin{enumerate}[(1)]
\item If $\mathbf{v} = \mathbf{1_a}$, for some $a\in\C$,  then 
\[\partial_i^{-1}(\{\mathbf{1_a}\}) = \begin{cases}
								\{\mathbf{1_a}\} & 1\le i\le n\\
								[\varphi|1_a|\cdots|1_a],\;\varphi\in\C(-,a) & i = 0\\
								[1_a|\cdots|1_a|\psi], \; \psi\in\C(a,-) & i = n+1.
						\end{cases}
\]\label{Lem:preim-simples-a}

\item If $\mathbf{v}$ is a non-homogeneous object, then the preimage $\partial_i^{-1}(\{\mathbf{v}\})$ is a disjoint union of singleton intervals on non-homogeneous objects. \label{Lem:preim-simples-b}
\end{enumerate}
\end{Lem}

\begin{proof}
Let $\textbf{v}\in \widetilde{\G}^\C_n$ be an object.  If $\textbf{v} = \mathbf{1_a}$ for some $a\in\C$, then for  $1\le i\le n$, one has $\partial_i^{-1}(\{\mathbf{1_a}\}) = \{\mathbf{1_a}\}$. If $i=0$, the objects in $\partial_0^{-1}(\{\mathbf{1_a}\})$ have the form $[\varphi|1_a|\cdots|1_a]$, where $\varphi$ is a morphism in $\C$ with codomain $a$. All such objects, where the codomain is not $a$, map by an irreducible morphism to $\mathbf{1_a}$ and are not comparable among themselves,  because $\C$ is direct  by hypothesis. Similarly, the objects in $\partial_{n+1}^{-1}(\{\mathbf{1_a}\})$ have the form $[1_a|\cdots|1_a|\psi]$, where $\psi$ is a morphism in $\C$ with domain $a$. In this case, there is an irreducible morphism from $\mathbf{1_a}$ to each of these objects that are not comparable among themselves, again because $\C$ is direct.  This proves Part \ref{Lem:preim-simples-a}.

For Part \ref{Lem:preim-simples-b}, notice that  if  $\mathbf{v}=[v_1|v_2|\cdots|v_n]$ is non-homogeneous, then every object in $\partial_i^{-1}(\{\mathbf{v}\})$ is non-homogeneous and any two such  objects  are not comparable since $\C$ is direct. This completes the proof.
\end{proof}

For a poset $\P$ and an interval $S$ in $\P$ we denote by $|S|$ the size of a maximal chain of irreducible relations in $S$.

\begin{Prop}\label{Prop:preim}
Let $n\geq 1$. Suppose that $S$ is an interval in $\widetilde{\G}^\C_n$ with $|S|>1$. Then, for any $0\le i\le n+1$, the preimage $\partial_i^{-1}(S)$ is a disjoint union of intervals that can be described as follows.
\begin{enumerate}[(1)]
\item  For any $1\le i \le n$, the preimage $\partial_i^{-1}(S)$ is connected. \label{Prop:preim-a}
\end{enumerate}

\noindent For $i=0$ or $i=n+1$, the inverse image $\partial_i^{-1}(S)$ always contains a connected component that  includes all homogeneous objects  in $\partial_i^{-1}(S)$. In addition, let $\mathbf{v} = [v_1|\ldots|v_n]$ be an object in $S$, with $a_0$ and $a_n$ the domain of $v_1$ and the codomain of $v_n$, respectively.  

\begin{enumerate}[(1)]\setcounter{enumi}{1}
\item If $\mathbf{1_{a_0}}\to\mathbf{v}$ is a morphism in $S$, then for each non-identity morphism $\alpha$ with codomain $a_0$ the preimage $\partial_0^{-1}(S)$ includes the non-homogeneous  object $[\alpha|v_1\cdots|v_n]$.
\label{Prop:preim-b}

\item If $\mathbf{v}\to \mathbf{1_{a_n}}$ is a morphism in $S$, then for each non-identity morphism $\beta$  with domain $a_n$ the preimage $\partial_{n+1}^{-1}(S)$  includes  the  non-homogeneous   object $[v_1\cdots|v_n|\beta]$.
\label{Prop:preim-c}
\end{enumerate}
In particular, all singletons in $\partial_0^{-1}(S)$ and $\partial_{n+1}^{-1}(S)$ are non-homogeneous.
\end{Prop}

\begin{proof}
By Lemma \ref{Lem:Inverse-Interval}, the preimages of intervals are disjoint unions of intervals. Thus it remains to investigate the nature of those intervals. Consider first preimages of the morphisms that are irreducible  in $S$.

\noindent\underline{Case 1:}   Irreducible morphisms of the  form $\mathbf{1_a}\to\mathbf{v}$.
\begin{itemize}
\item If $1\le i\le n+1$, then
$\partial_i^{-1}(\mathbf{1_a}\to\mathbf{v})$ consists of the object $\mathbf{1_a}$ and non-homogeneous objects $\mathbf{u} = [u_1|\cdots|u_{n+1}]$, with $a$ the domain of $u_1$ and $\mathbf{u}\in\partial_i^{-1}(\mathbf{v})$. This is connected, because there is an irreducible morphism $\mathbf{1_a}\to \mathbf{u}$  for any $\mathbf{u}$ that satisfies the conditions above. 

\item If $i=0$, then
$\partial_0^{-1}(\mathbf{1_a}\to\mathbf{v})$ consists of the preimage $\partial_{0}^{-1}(\mathbf{1_a})$ that is obviously connected and the objects $\mathbf{u}\in\partial_0^{-1}(\mathbf{v})$, all of which are non-homogeneous.  Such objects $\mathbf{u}$, where the domain of $u_1$ receives no morphisms from $a$, form isolated singleton intervals by Lemma \ref{Lem:preim-simples}, whereas the objects where $a$ maps to the domain of $u_1$ form a connected component. 
\end{itemize}

\noindent\underline{Case 2:} Irreducible morphisms of the form $\mathbf{u}\to \mathbf{1_a}$.
\begin{itemize}
\item If $0\le i \le n$, then  $\partial_i^{-1}(\mathbf{u}\to \mathbf{1_a})$ consist of the preimage $\partial_{n+1}^{-1}(\mathbf{1_a})$ that is obviously connected and non-homogeneous objects $\mathbf{w} = [w_1|\cdots|w_{n+1}]$, with $a$ the codomain of $w_{n+1}$ and $\mathbf{w}\in\partial^{-1}(\mathbf{u})$. This is connected because there is an irreducible morphism $\mathbf{w}\to\mathbf{1_a}$  for each $\mathbf{w}$ that satisfies the conditions above. 

\item If $i=n+1$, then  $\partial_{n+1}^{-1}(\mathbf{u}\to \mathbf{1_a})$ consists of the object $\mathbf{1_a}$ and all objects in $\partial_{n+1}^{-1}(\mathbf{u})$, all of which are non-homogeneous. Those objects for which there are no morphisms from the  codomain of $u_{n+1}$ to $a$ form isolated  singleton intervals by Lemma \ref{Lem:preim-simples}, while the other objects together with $\mathbf{1_a}$ form a connected component. 
\end{itemize}

Next, recall that the morphisms of the form $\mathbf{1_a}\to\mathbf{u}$ and $\mathbf{v}\to \mathbf{1_a}$ are the only irreducible morphisms in $\widetilde{\G}^\C_n$.  The convexity condition that intervals must satisfy implies that any interval $S$ is determined uniquely by the irreducible morphisms it contains. For each $a\in \C$ such that $\mathbf{1_a}\in S$, let $N_a$ denote the neighbourhood of $\mathbf{1_a}$ in $S$, that is the subcategory that contains $\mathbf{1_a}$ and all irreducible morphisms into and out of $\mathbf{1_a}$ in $
S$. Then $\partial^{-1}_i(N_a)$ is connected for $1\le i\le n$. Since  $S$ is connected, any two homogeneous objects  $\mathbf{1_a}$ and $\mathbf{1_b}$ in $S$ are connected by some zig-zag of morphisms in $S$. Since $\partial_i\circ s_i = \Id$, applying the degeneracy operator $s_i$ to that zig-zag  gives a zig-zag of morphisms in $\partial^{-1}_i(S)$ between $\mathbf{1_a}\to\mathbf{u}$ and $\mathbf{v}\to \mathbf{1_b}$ in $\partial_i^{-1}(S)$. Therefore, all homogeneous objects in  $\partial^{-1}_i(S)$ reside in the same connected component.  
Since $\partial^{-1}_i(S)$ is a union of all inverse images of  the form $\partial^{-1}_i(N_a)$, it follows that $\partial^{-1}_i(S)$ is connected. Part \ref{Prop:preim-a} follows.

By the same argument, except using the identity  $\partial_0\circ s_0 = \Id$, we see that all homogeneous objects in  $\partial^{-1}_0(S)$ reside in the same connected component.  Hence the inverse image under $\partial_0$ of each object in $S$ consists of objects in $\widetilde{\G}^\C_{n+1}$ that are also in that same connected component as the homogeneous objects, and possibly objects that are  singleton objects there. This proves Part \ref{Prop:preim-b}. Part \ref{Prop:preim-c} is proved similarly, using the simplicial identity $\partial_{n+1}\circ s_n=\Id$.
\end{proof}

\begin{Cor}\label{Cor:interval-deltas}
For $n\geq 1$, let $M\in k\widetilde{\G}^\C_n\mod$ be an interval module defined on some non-singleton interval. Then
\begin{enumerate}[(1)]
\item   $\partial^*_iM$ is a non-singleton interval module for all $1\le i \le n$.
\item $\partial^*_0 M = P\oplus Q$, where $P$ is a non-singleton interval module and $Q$ is a direct sum of singleton interval modules. Furthermore, the support of each summand of $Q$  is an object of the form $[\alpha|v_1|\cdots|v_n]$, as in Proposition \ref{Prop:preim}\ref{Prop:preim-b}.

\item $\partial^*_{n+1} M = P\oplus Q$, where $P$ is a non-singleton interval module and $Q$ is a direct sum of singleton interval modules.  Furthermore, the support of each summand of $Q$  is an object of the form $[v_1|\cdots|v_n|\beta]$, as in Proposition \ref{Prop:preim}\ref{Prop:preim-c}.
\end{enumerate}
\end{Cor}

\begin{Defi}
For each integer $n\ge 0$, let $\I^s(k\widetilde{\G}^\C_n)$ be the subgroup of $\I(k\widetilde{\G}^\C_n)$ generated by non-homogeneous singleton modules.
\end{Defi}

\begin{Lem}\label{Lem:I_0-cosimplicial}
$\I^s(k\widetilde{\G}^\C_\bullet)$ is a cosimplicial subgroup of $\I(k\widetilde{\G}^\C_\bullet)$.
\end{Lem}

\begin{proof}
By Lemma \ref{Lem:preim-simples}\ref{Lem:preim-simples-b}, the cofaces map non-homogeneous singleton modules to sums of non-homogeneous singleton modules. Note, in particular, that $\I^s(k\widetilde{\G}^\C_0)=0$, since there are no non-homogeneous singletons  in $\widetilde{\G}^\C_0$. This shows that $\I^s(k\widetilde{\G}^\C_\bullet)$ is closed under cofaces.

To show closure under degeneracies, consider  the non-homogeneous singleton modules $S_{\textbf{v}}$, where $\textbf{v}=[v_1|\cdots|v_n]\in \widetilde{\G}_n^\C$ and $n\geq 1$. Then for $0\le i\le n-1$,  $\sigma^iS_{\textbf{v}}$ evaluated on $\textbf{w}\in \widetilde{\G}_{n-1}^\C$ is given by
\[
(\sigma^iS_{\textbf{v}})(\textbf{w})=S_{\mathbf{v}}(s_i(\mathbf{w}))=
\begin{cases}
k &   s_i(\textbf{w})=\textbf{v}\\
0 & \text{otherwise.} 
\end{cases}
\] 
But $s_i(\textbf{w})=\textbf{v}$ if and only if $v_{i+1}$ is an identity morphism in $\C$, and $\partial_i(\textbf{v})=\textbf{w}$. Therefore,
\[
\sigma^iS_{\textbf{v}}=
\begin{cases}
S_{\partial_i(\textbf{v})} &  v_{i+1}=\Id\\
0 & \text{otherwise.} 
\end{cases}
\] 
Note that if $v_{i+1}$ is an identity morphism, then $\partial_i(\textbf{v})$ cannot be homogeneous, because otherwise $\textbf{v}$ would be homogeneous.
\end{proof}

\begin{Defi}
Let $\C$ be a direct category. The \hadgesh{reduced  interval module group for $\C$ with coefficients in $k$} is the cosimplicial abelian group
\[
\I_0(k\widetilde{\G}^\C_\bullet)\defeq \I(k\widetilde{\G}^\C_\bullet)/\I^s(k\widetilde{\G}^\C_\bullet).
\] 
Similarly, the \hadgesh{reduced Grothendieck group for $\widetilde{\G}^\C_\bullet$ with coefficients in $k$} is the cosimplicial abelian group
\[
\Gr_0(k\widetilde{\G}^\C_\bullet)\defeq \Gr(k\widetilde{\G}^\C_\bullet)/\I^s(k\widetilde{\G}^\C_\bullet).
\]
We denote the image of $X\in \Gr(k\widetilde{\G}^\C_n)$ in the reduced Grothendieck group by $X_0$, and the cofaces and codegeneracies by $\delta_0^i$ and $\sigma_0^i$, respectively. Obviously, $\I_0(k\widetilde{\G}^\C_\bullet)$ is a subgroup of $\Gr_0(k\widetilde{\G}^\C_\bullet)$.
\end{Defi}

In particular, if $M$ and $N$ are $k\widetilde{\G}^\C_n$-modules such that $[M]_0 = [N]_0$ in $\Gr_0(k\widetilde{\G}^\C_n)$, then all the indecomposable summands of $M$ and $N$ are isomorphic, with the exception of the non-homogeneous singleton module summands.

By Proposition \ref{Cor:interval-deltas}, any coface $\delta_0^i$ in $\I_0(k\widetilde{\G}_\bullet^\C)$ maps a class of an interval module to a class of an interval module. In fact, the same is true for general indecomposable modules in $\Gr_0(k\widetilde{\G}_n^\C)$.

\begin{Lem}\label{Lem:reduced-indecomp}
For any $0\le i\le n+1$ and  an indecomposable module $M\in k\widetilde{\G}^\C_n\mod$, the coface $\delta^i_0([M]_0)$ is the class of an indecomposable module   $N\in k\widetilde{\G}_{n+1}^\C\mod$.
\end{Lem}

\begin{proof}
Suppose that
\[
\delta_0^i([M]_0)=[\partial^*_i(M)]_0=[P\oplus Q]_0= [P]_0+ [Q]_0
\]
 for some $k\widetilde{\G}^\C_{n+1}$-modules $P$ and $Q$. 
 
Suppose first that $i\neq n+1$. Then
\[
[M]_0=(\sigma_0^i\circ \delta_0^i)([M]_0)=\sigma_0^i([P]_0)+\sigma_0^i([Q]_0).
\]
Since $M$ is indecomposable, either $\sigma_0^i([P]_0)=0$ or $\sigma_0^i([Q]_0)=0$. Without loss of generality, we may assume that $\sigma_0^i([Q]_0)=0$. If $Q\notin \I^s(k\widetilde{\G}^\C_{n+1})$ then $\supp(Q)$ contains at least one homogeneous object $\mathbf{1_a}$ for some object $a\in\C$ (since every morphism that is not the identity on a non-homogeneous object factors through a homogeneous object). Now,
\[
\sigma^iQ(\mathbf{1_a}) = Q(s_i(\mathbf{1_a}))=Q(\mathbf{1_a})\neq 0.
\]
Hence  $\sigma_0^i([Q]_0)=[s^*_i Q]_0\neq 0$, which is a contradiction to our assumption. Therefore $Q\in\I^s(k\widetilde{\G}^\C_n)$, so $[Q]_0=0$. Finally, for $i=n+1$,  use the cosimplicial identity $\sigma^{n}\circ \delta^{n+1}=\Id$, and the claim follows similarly.
\end{proof}

\begin{Prop}\label{Prop:reduced-acyclicity}
Let $\C$ be a direct category. Then the reduced cosimplicial Grothendieck group $\Gr_0(k\widetilde{\G}_\bullet^\C)$ and the reduced interval module group $\I_0(k\widetilde{\G}_\bullet^\C)$ are both acyclic.
\end{Prop}

\begin{proof} 
Acyclicity of both objects is proven very similarly to the proof of  Proposition \ref{Prop:Acyclicity}, using Lemmas \ref{Lem:coboundary-rels} and \ref{Lem:reduced-indecomp}, where the latter  is analogous to Lemma \ref{Lem:delta-indecomp}.

Let $U^\bullet$ denote $\Gr_0(k\widetilde{\G}^\C_\bullet)$ for short, and let $(N^\bullet(U), d^\bullet)$ denote the Moore cochain complex associated to $U^\bullet$. We denote the cofaces in the Moore cochain complex $N^\bullet(U)$ by $\delta_N^i$, and the cofaces in  $U^\bullet$ by $\delta_U^i$. For $X\in \Gr(k\widetilde{\G}_\bullet^\C)$, we denote its image in $U^\bullet$ by  $X^U$, and the image of $X^U$ in the Moore cochains  $N^\bullet(U)$ by $X^N$. 

Let $X \in \Gr(k\widetilde{\G}_n^\C)$ represent a cocycle $X^N\in N^n(U)$.
Write 
\[X = \sum_{k=1}^r[M_k] - \sum_{l=1}^s[N_l],\]
where $M_k, N_l\in k\widetilde{\G}^\C_n\mod$ are non-zero and indecomposable. We may assume that the sets $\{M_k\}_{k=1}^r$ and $\{N_l\}_{l=1}^s$ have neither modules that are isomorphic, nor non-homogeneous singleton modules. Hence it makes sense to consider $r+s$ as the \hadgesh{decomposition length} of $X$. 

Since, by assumption,  $d^n X^N= (\delta_U^{n+1}X^U)^N = (\delta^{n+1}X)^N=0$ in the Moore cochains, there exist elements $\{Y_i = [P_i]-[Q_i]\}_{i=0}^{n}$ in $\Gr(k\widetilde{\G}_n^\C)$ such that 
\[(\delta^{n+1}X)^U =  \sum_{k=1}^r(\delta^{n+1}[M_k])^U - \sum_{l=1}^s(\delta^{n+1}[N_l])^U = \sum_{i=0}^{n} (\delta^i Y_i)^U.\] 
By Lemma \ref{Lem:reduced-indecomp}, $(\delta^{n+1}[M_1])^U$ is represented by some indecomposable $k\widetilde{\G}_{n+1}^\C$-module. Hence, for some $0\le i\le n$, there is an indecomposable summand $T_1$ of $P_i$,  such that $(\delta^{n+1}[M_1])^U = (\delta^i [T_1])^U$. There are two possibilities: if $i=n$, then, by Lemma \ref{Lem:coboundary-rels}\ref{Lem:coboundary-rels-1}, $[M_1]^U = [T_1]^U\in\Ima(\delta^{n}_U)$, while if $i<n$, then, by Lemma \ref{Lem:coboundary-rels}\ref{Lem:coboundary-rels-2},  $[M_1]^U\in\Ima(\delta^i_U)$ and $[T_1]^U\in\Ima(\delta^{n}_U)$. In the first case,  $[M_1]^U$ is a coboundary in $N^n(U)$. In the second case, since $i<n$, $[M_1]^N=0$ in $N^n(U)$. Thus $(X-[M_1])^U$ is a cocycle that represents the same cohomology class as $X^U$, and $X-[M_1]$ is of decomposition length $r+s-1$. By downward induction on the decomposition length of $X$, each of its indecomposable summands is either a coboundary or $0$ in $N^n(U)$, and hence $X^U$ itself is a coboundary and the proof for $\Gr_0(k\widetilde{\G}_\bullet^\C)$ is complete.

The proof for $\I_0(k\widetilde{\G}_\bullet^\C)$  uses the same argument as above, but with interval modules (which are inherently indecomposable) instead of general indecomposable modules.
\end{proof}

For a cochain complex $C^*$ of a cosimplicial abelian group, let $C\red^*$ denote the subcomplex of  $C^*$ given by the dimension-wise intersection of  kernels of all codegeneracy operators. It is a basic fact that the inclusion $\C\red^*\to \C^*$ induces an isomorphism in cohomology. We are now ready to prove Theorem \ref{Th:Higher-G}, which we restate in more detail as follows.

\begin{Thm}\label{Thm:Higher-G}
For any direct category $\C$ the following statements hold.
\begin{enumerate}[(1)]
\item The inclusions $\I^s(k\widetilde{\G}^\C_\bullet)\to \I(k\widetilde{\G}^\C_\bullet)\to \Gr(k\widetilde{\G}^\C_\bullet)$ induce isomorphisms 
\[H^n(\I^s(k\widetilde{\G}^\C_\bullet))\xto{\cong} H^n(\I(k\widetilde{\G}^\C_\bullet))\xto{\cong} \widetilde{\GG}^n(\C, k)\]
for all $n\geq 2$. \label{Thm:Higher-G-1}
\item For $n \le 1$ one has
\[\widetilde{\GG}^0(\C, k)\cong H^0(\Gr_0(k\widetilde{\G}^\C_\bullet))\quad\text{and}\quad  \widetilde{\GG}^1(\C, k)\cong \Ker\left(\I^s(k\widetilde{\G}^\C_1)\xto{d^1} \I^s(k\widetilde{\G}^\C_2)\right).\] \label{Thm:Higher-G-2}
\item The map
$\widetilde{\GG}^n(\C, k)\to H^n(|\C|, \Z)$
 induced by natural inclusion of simplicial objects  $|\C|_\bullet\to \widetilde{\G}^\C_\bullet$ is an epimorphism for $n=1$ and an isomorphism for $n>1$. \label{Thm:Higher-G-3}
\end{enumerate}
In fact the inclusion $\I^s(k\widetilde{\G}^\C_\bullet)\red\to C^*(|\C|, \Z)$ induces an isomorphism in cohomology for all $n\geq 1$.
\end{Thm}

\begin{proof}
The short exact sequence of cosimplicial groups
\begin{equation}
    0\rightarrow \I^s(k\widetilde{\G}_\bullet^\C)\rightarrow \Gr(k\widetilde{\G}_\bullet^\C)\rightarrow \Gr_0(k\widetilde{\G}_\bullet^\C)\rightarrow 0 \label{Eq:Higher-G}
\end{equation}
induces a long exact sequence in cohomology. 
Since, by Proposition \ref{Prop:reduced-acyclicity}, $H^n(\Gr_0(k\widetilde{\G}_\bullet^\C))=0$ for $n\geq 1$, the maps 
\[
 H^n(\I^s(k\widetilde{\G}^\C_\bullet))\xto{\cong} H^n(\Gr(k\widetilde{\G}_\bullet^\C)) = \widetilde{\GG}^n(\C, k)
\]
are isomorphisms for all $n\geq 2$.
  Similarly, Proposition \ref{Prop:reduced-acyclicity} and the short exact sequence 
\[
    0\rightarrow \I^s(k\widetilde{\G}_\bullet^\C)\rightarrow \I(k\widetilde{\G}_\bullet^\C)\rightarrow \I_0(k\widetilde{\G}_\bullet^\C)\rightarrow 0
\]
imply that the maps $H^n(\I^s(k\widetilde{\G}_\bullet^\C))\to H^n(\I(k\widetilde{\G}_\bullet^\C))$ are isomorphisms for all $n\geq 2$.  
Part \ref{Thm:Higher-G-1} follows.

Next, notice that $\Gr(k\widetilde{\G}_0^\C ) = \Gr_0(k\widetilde{\G}_0^\C)$, since $\I^s(k\widetilde{\G}_0^\C)=0$. Clearly every linear combination of isomorphism classes of locally constant modules is an  element of $H^0(\Gr_0(k\widetilde{\G}^\C_\bullet))$, as well as of $\widetilde{\GG}^0(\C,k)$. Hence the natural map $\widetilde{\GG}^0(\C, k)\to H^0(\Gr_0(k\widetilde{\G}^\C_\bullet))$, which by construction is a monomorphism, is also an epimorphism and hence an isomorphism. It follows that the first connecting homomorphism in the long exact cohomology sequence associated to the short exact sequence (\ref{Eq:Higher-G}) is trivial. Thus, 
\[\GG^1(\C,k)\cong H^1(\I^s(k\widetilde{\G}^\C_\bullet) )= \Ker\left(\I^s(k\widetilde{\G}^\C_1)\xto{d^1} \I^s(k\widetilde{\G}^\C_2)\right),\]
where the second equality holds because $\I^s(k\widetilde{\G}_0^\C)=0$.
Part \ref{Thm:Higher-G-2} follows.

To prove the last statement, notice first that if $\mathbf{v}\in\widetilde{\G}^\C_n$ for $n\geq 1$, then 
\[
\sigma^iS_{\textbf{v}}=
\begin{cases}
S_{\partial_i(\textbf{v})} &  v_{i+1}=\Id\\
0 & \text{otherwise.} 
\end{cases}
\] 
(See the proof of Lemma \ref{Lem:I_0-cosimplicial}). Hence for $n\geq 1$, if $S_{\mathbf{v}}\in\I^s(k\widetilde{\G}^\C_n)$ is a singleton module where $\mathbf{v}$ contains an identity morphism, then  $S_{\mathbf{v}}$ is not in the kernel of at least one codegeneracy, and hence not in the intersection of all kernels. Clearly if $\sigma^iS_\mathbf{v} = \sigma^iS_\mathbf{u}$ then $\mathbf{v} = \mathbf{u}$, because in that case there is some $\mathbf{w}\in \widetilde{\G}^\C_{n-1}$  such that $\mathbf{v} = s_i(\mathbf{w}) = \mathbf{u}$.  Thus $\I^s(k\widetilde{\G}^\C_n)\red$, for $n\geq 1$,  is freely generated by singleton modules $S_{\mathbf{v}}$, where $\mathbf{v}\in \widetilde{\G}^\C_n$ is non-degenerate.    Since $\I^s(k\widetilde{\G}^\C_0) = 0$, the cochain map 
\[\I^s(k\widetilde{\G}^\C_\bullet)\red\to C^*(|\C|, \Z)\]
that takes   the singleton module $S_{\mathbf{v}}$ on a non-degenerate simplex $\mathbf{v}$ to the characteristic  function $\chi_\mathbf{v}$ induces an isomorphism in cohomology in all dimensions $n\ge 2$ and an epimorphism for $n=1$. Notice that this map factors through the inclusion $\I^s(k\widetilde{\G}^\C_\bullet)\red\to\Gr(k\widetilde{\G}^\C_\bullet)$ followed by the projection to $C^*(|\C|, \Z)$. This proves Part \ref{Thm:Higher-G-3}.
\end{proof}

%%%%%%%%%%%%%%%%%%%%%%%%%%%%%%%%%%%%%%%%%%%
%%%%%%%%%%%%%%%%%%%%%%%%%%%%%%%%%%%%%%%%%%%
%%%%%%%%%%%%%%%%%%%%%%%%%%%%%%%%%%%%%%%%%%%
%%%%%%%%%%%%%%%%%%%%%%%%%%%%%%%%%%%%%%%%%%%
%%%%%%%%%%%%%%%%%%%%%%%%%%%%%%%%%%%%%%%%%%%
%%%%%%%%%%%%%%%%%%%%%%%%%%%%%%%%%%%%%%%%%%%
%%%%%%%%%%%%%%%%%%%%%%%%%%%%%%%%%%%%%%%%%%%
%%%%%%%%%%%%%%%%%%%%%%%%%%%%%%%%%%%%%%%%%%%
%%%%%%%%%%%%%%%%%%%%%%%%%%%%%%%%%%%%%%%%%%%
%%%%%%%%%%%%%%%%%%%%%%%%%%%%%%%%%%%%%%%%%%%
%%%%%%%%%%%%%%%%%%%%%%%%%%%%%%%%%%%%%%%%%%%
\section{\texorpdfstring{$\EE^*$ and $\GG^*$ cohomology}{E and G cohomology}}
We now turn to general computations of  $\EE^*$ and $\GG^*$.  We assume throughout that all categories discussed in this section have no non-identity invertible morphisms, so that these theories are well defined. Obtaining general computational results about $\EE^*$ and $\GG^*$ is more difficult even in dimension $0$, since the restriction to non-degenerate objects and as such concentrating on the meaningful morphisms in $\C$ gives a much more involved structure, and in particular higher cohomology. The main results in this section are the implication for a module $M\in k\C\mod$ to represent a $0$-cocycle, and bounds on the degrees in which positive-dimensional cohomology can appear in the case of a category with finite composition length.

%%%%%%%%%%%%%%%%%%%%%%%%%%%%%%%%%%%%%%%%%%%
%%%%%%%%%%%%%%%%%%%%%%%%%%%%%%%%%%%%%%%%%%%

\subsection{\texorpdfstring{The groups $\EE^0(\C,k)$ and $\GG^0(\C, k)$}{The groups E0 and G0}}
\label{SSec:Cohomology-chEG}

Recall that  $\E^\C_0 = \G^\C_0 =\C$. For $n>0$, the object sets $\obj(\G^\C_n) = \obj(\E^\C_n)$  consist of  the \hadgesh{non-degenerate} simplices of $|\C|$. The categories $\E^\C_n$ and  $\G^\C_n$ are  full subcategories of $\widetilde{\E}^\C_n$ and $\widetilde{\G}^\C_n$, respectively. 

The following proposition shows that $\EE^0$ behaves similarly to $\widetilde{\EE}^0$ and $\widetilde{\GG}^0$ on classes of genuine (as opposed to virtual) modules, with a slight variation.

\begin{Prop}\label{Prop:E-check-0-cocycles}
Let $\C$ be a small category and let $M\in k\C\mod$ be such that $[M]$ represents a class in $\EE^0(\C, k)$. Then $M$ is locally constant on each square-component of $\C$ that contains at least two different non-identity morphisms. 
\end{Prop}
\begin{proof}
Fix a natural isomorphism $\eta\colon \partial_1^*M \to \partial_0^*M$.
Let $u\colon x\to y$ and $v\colon z\to w$ be any two distinct (non-identity) morphisms in the same square-component of $\C$. Assume first that there is a morphism in $\E_1^\C$ from $u$ to $v$. Thus there are morphisms $s\colon x\to z$ and $t\colon y\to w$ in $\C$, such that $t\circ u = v\circ s$. 

\noindent{\bf Case 1: $s$ and $t$ are non-identity morphisms.} In this case $[s]$, $[t]$ and $[t\circ u]$ are objects in  $\E_1^\C$, and we have morphisms $(1_x,v)\colon [s]\to [t\circ u]$  and $(u, 1_w)\colon[t\circ u]\to [t]$. Since $\eta$ is a natural isomorphism,  we have a commutative  diagram
\[\xymatrix{
M(x)=\partial_1^*M([s])\ar@{=}[d]_{\partial_1^*M(1_x,v) = M(1_x)}\ar[rr]^{\eta_{[s]}}_\cong &&  \partial_0^*M([s])=M(z)\ar[d]^{M(v) = \partial_0^*M(1_x,v)}\\
M(x)=\partial_1^*M([t\circ u])\ar[d]_{\partial_1^*M(u,1_w)=M(u)}\ar[rr]^{\eta_{[t\circ u]}}_\cong && \partial_0^*M([t\circ u])=M(w)\ar@{=}[d]^{M(1_w) = \partial_0^*M(u, 1_w)}\\
M(y)=\partial_1^*M([t])\ar[rr]^{\eta_{[t]}}_\cong &&  \partial_0^*M([t])=M(w)
}\]
It follows that $M(u)$ and $M(v)$ are both isomorphisms. 

\noindent{\bf Case 2: $s = 1_x$ and $t$ is a  non-identity morphism.} In this case $[u]$, $[t]$, and $[t\circ u]$ are objects in  $\E_1^\C$, and we have morphisms $(1_x,t)\colon [u]\to[t\circ u]$ and $(u,1_w)\colon [t\circ u]\to [t]$.  Again, since $\eta$ is a natural isomorphism,  we have a commutative  diagram
\[\xymatrix{
M(x)=\partial_1^*M([u])\ar@{=}[d]_{\partial_1^*M(1_x,t)=M(1_x)}\ar[rr]^{\eta_{[u]}}_\cong &&  \partial_0^*M([u])=M(y)\ar[d]^{\partial_0^*M(1_x,t) = M(t)}\\
M(x)=\partial_1^*M([t\circ u])\ar[d]_{\partial_1^*M(u,1_w) = M(u)}\ar[rr]^{\eta_{[t\circ u]}}_\cong && \partial_0^*M([t\circ u])=M(w)\ar@{=}[d]^{\partial_0^*M(u,1_w) = M(1_w)}\\
M(y)=\partial_1^*M([t])\ar[rr]^{\eta_{[t]}}_\cong &&  \partial_0^*M([t])=M(w)
}\]
As before, it follows that $M(t)$ and $M(u)$ are isomorphisms, and hence so is $M(t\circ u)=M(t)\circ M(u)$. But $M(t\circ u) = M(v\circ s) = M(v)$. Thus $M(v)$ is an isomorphism as well.

\noindent{\bf Case 3: $t = 1_y$ and $s$ is a  non-identity morphism.} This is similar to the previous cases and is left for the reader. 

Next consider the general case. Since $u$ and $v$ are in the same square-component of $\C$, the objects $[u]$ and $[v]$ are in the same connected component of $\E_1^\C$. This means that there exist objects $[q_1], [q_2], \ldots, [q_r]$, for some $r\geq 0$, and a zig-zag of morphisms 
\[[u] \rightarrow [q_1]\leftarrow [q_2]\rightarrow\cdots \leftarrow [q_r]\rightarrow[v].\]
Each such morphism can be used by the argument above to show that the morphisms that define its source and target induce isomorphisms upon applying $M$. Hence, in particular, both $M(u)$ and $M(v)$ are isomorphisms, and it follows that $M$ is locally constant on  each square-component of $\C$, as claimed.
\end{proof}

Notice that $\EE^0(\C, k)$ generally contains classes of virtual modules that are not locally constant, as Proposition \ref{Prop:hG0E0-[n]} demonstrates.

Next we consider $\GG^0(\C,k)$. As one could expect, the implications for a module $M$ to be a $0$-cocycle are considerably  weaker, but one still gains information about such modules. 

\begin{Lem}
\label{Lem:G-check-0-cocycles}
Let $\C$ be a small category and let $M\in k\C\mod$ be such that $[M]$ represents a class in $\GG^0(\C, k)$. Then, for any pair of morphisms $u\colon a\to b$ and $v\colon c\to d$  that are in the same line-component of $\C$, there are isomorphisms $\alpha\colon M(a)\to M(c)$ and $\beta\colon M(b)\to M(d)$ such that $\beta\circ M(u) = M(v)\circ\alpha$. 
\end{Lem}
\begin{proof}
Since the morphisms $u$ and $v$  are in the same line-component of $\C$ by hypothesis, the corresponding objects  $[u]$ and $[v]$ are in the same connected component of $\G^\C_1$.    Assume first that there is a morphism $(\varphi,\psi)\colon[u]\to[v]$. By definition of the object $\G^\C_\bullet$, there is a morphism $w\colon b\to c$, such that $\varphi = w\circ u$ and $\psi = v\circ w$.   Let $\eta\colon\partial^*_1M\to\partial_0^*M$ be a natural isomorphism. If $w=1_b$, then set $\alpha =\eta_{[u]}$ and $\beta = \eta_{[v]}$. Any other morphism $(\varphi,\psi)\colon[u]\to[v]$, including the case where  $w$ is a non-invertible endomorphism of $b$, can be factored as $(\varphi,\psi) = (w,v)\circ (u,w)$, and one has the relations:
\[M(w)\circ\eta_{[u]} = \eta_{[w]}\circ M(u),\quad\text{and}\quad M(v)\circ\eta_{[w]} = \eta_{[v]}\circ M(w).\]
Thus, setting $\alpha \defeq \eta_{[w]}\circ\eta_{[u]}$ and $\beta \defeq \eta_{[v]}\circ\eta_{[w]}$, we obtain the desired relation $\beta\circ M(u) = M(v)\circ\alpha$.

In general, since the morphisms $u$ and $v$  are in the same line-component of $\C$,  there is a zig-zag of morphisms in $\G^\C_1$ connecting $[u]$ to $[v]$. The relation of the form above holds for each morphism in the zig-zag. Hence the claim follows by induction on the length of the chain connecting $[u]$ and $[v]$. 
\end{proof}

\begin{Rem}\label{Rem:G-check-0-cocycles}
Notice that Lemma \ref{Lem:G-check-0-cocycles} makes no claim of uniqueness of the isomorphisms $\alpha$ and $\beta$. Indeed, in general one cannot expect the choice of these maps to be unique. (Compare \cite[Theorem 5.3]{BLR})
\end{Rem}

\begin{Prop}
\label{Prop:G-check-0-cocycles-rk}
Let $\C$ be a small category and let $M, N\in k\C\mod$ be such that $X = [M] - [N]\in \GG^0(\C, k)$. Then $\rk X$ is a constant function on the set of all non-identity morphisms of each line-component of $\C$ and on the set of all identity morphisms of objects within each line-component. 
\end{Prop}
\begin{proof}
By Lemma \ref{Lem:G-check-0-cocycles} the statement is clear if either $M$ or $N$ is trivial. Hence assume that neither of them is. Then by hypothesis $\delta^0[M]+\delta^1[N] = \delta^0[N]+\delta^1[M]$. Let 
\[\eta^X \colon M\circ\partial_0\oplus N\circ\partial_1 \to N\circ\partial_0\oplus M\circ\partial_1\]
be a natural isomorphism. Let $[u], [v]\in \G^\C_1$ be objects, where $u\colon a\to b$ and $v\colon b\to c$. 
\[\xymatrix{
(M\circ\partial_0\oplus N\circ\partial_1)[u]\ar@{=}[r]&M(b)\oplus N(a)\ar[rr]^{\eta^X_{[u]}}_{\cong}\ar[d]_{M(v)\oplus N(u)} && N(b)\oplus M(a) \ar[d]^{N(v)\oplus M(u)}\ar@{=}[r]& (N\circ\partial_0\oplus M\circ\partial_1)[u] \\
(M\circ\partial_0\oplus N\circ\partial_1)[v]\ar@{=}[r]&M(c)\oplus N(b) \ar[rr]_{\eta^X_{[v]}}^\cong&& N(c)\oplus M(b) \ar@{=}[r]& (N\circ\partial_0\oplus M\circ\partial_1)[v]
}\] 
The top isomorphism implies at once that
\[\rk X (1_a) = \rk[M](1_a) - \rk[N](1_a)  = \rk[M](1_b) - \rk[N](1_b)  = \rk X (1_b).\]
Similarly, the bottom isomorphism implies that $\rk X (1_b) = \rk X (1_c)$. The same inductive argument as in the proof of Lemma \ref{Lem:G-check-0-cocycles} now shows that $\rk X $ is constant on $1_x$ for all objects $x$ that are in the same line-component of $\C$.

Commutativity of the diagram also shows that 
\[\rk X (u) =  \rk[M](u) -  \rk[N](u) = \rk[M](v) - \rk[N](v) = \rk X (v).\]
Again, by induction and using line-connectivity, one shows that $\rk X (w)$ is constant for all morphisms $w$ that are in the same line-component of $\C$, and the proof is thus complete.
\end{proof}

Restricting to posets whose Hasse diagrams are rooted trees, the idea of the proof of Lemma \ref{Lem:G-check-0-cocycles} has another useful consequence.

\begin{Lem}\label{Lem:G-check-0-Kernels}
Let $\P$ be a poset with composition length at least $2$, whose Hasse diagram $\HH_\P$ is a rooted (not necessarily finite) tree with root $r$. Let $M\in k\P\mod$ be a module such that $[M]\in \GG^0(\P, k)$. Then the following statements hold. 
\begin{enumerate}[(1)]
\item For any two relations $x, y > a$ in $\P$, where  $a\neq r$, 
\[\Ker M(a<x) = \Ker M(a<y).\] \label{Lem:G-check-0-Kernels-1}
\item If $a < b$ and $c < d$ are in the same line-component of $\P$ and $a, c \neq r$, then $\Ker M(a<b)\cong\Ker M(c<d)$. \label{Lem:G-check-0-Kernels-2}
\item If, in addition, the root $r$ of $\P$ has a unique successor $s$, then  
\[\Ker M(r<s) = \Ker M(r<s')\]
 for any $s'>s>r$, such that $s'$ is not maximal. \label{Lem:G-check-0-Kernels-3}
\end{enumerate} 
\end{Lem}
\begin{proof}
Fix a natural isomorphism $\eta\colon\partial^*_1M\to\partial^*_0M$.   Since $a\neq r$, there is some object $z\in\P$ such that $z< a$.   Then there are obvious morphisms $[z<a]\to[a<x]$ and $[z<a]\to [a<y]$, so 
\[M(a<x)\circ\eta_{[z<a]}  = \eta_{[a<x]}\circ M(z<a),\quad\text{and}\quad M(a<y)\circ\eta_{[z<a]}  = \eta_{[a<y]}\circ M(z<a).\] 
It follows that \[\eta_{[a<x]}^{-1} \circ M(a<x)\circ \eta_{[z<a]} = \eta_{[a<y]}^{-1} \circ M(a<y)\circ \eta_{[z<a]},\] and since $\eta_{[z<a]} $ is an isomorphism, \[\eta_{[a<x]}^{-1} \circ M(a<x) = \eta_{[a<y]}^{-1} \circ M(a<y).\] Part \ref{Lem:G-check-0-Kernels-1} follows.

If $a < b$ and $c < d$ are in the same line-component of $\P$,  there is some $x<y$ in $\P$ such that $y\le a$ and $y\le c$. Since $y$ is not minimal, Part \ref{Lem:G-check-0-Kernels-1} implies that $\Ker M(y<a)=\Ker M(y<c)$. The relation $y<a<b$ implies that $M(a<b)\circ\eta_{[y<a]} = \eta_{[a<b]}\circ M(y<a)$. Similarly, the relation $y<c<d$ gives $M(c<d)\circ\eta_{[y<c]} = \eta_{[c<d]}\circ M(y<c) = \eta_{[c<d]}\circ M(y<a)$. Putting these observations together, we see that 
\[\Ker M(a<b) = \eta_{[y<a]}\left(\Ker M (y<a)\right),\quad\text{and}\quad \Ker M(c<d) = \eta_{[y<c]}\left(\Ker M(y<a)\right).\]
This proves Part \ref{Lem:G-check-0-Kernels-2}.

Since $s'$ is not maximal, there is some $s''>s'$ in $\P$. Clearly $\Ker M(r<s) \subseteq \Ker M(r<s')$, since $r<s<s'$. The same relation also implies  that $\Ker M(r<s)\cong \Ker M(s<s')$, while the relation $r<s'<s''$ implies similarly that $\Ker M(r<s')\cong \Ker M(s'<s'')$. Notice that the relations $s<s'$ and $s'<s''$ are in the same line-component of $\P$. Hence, by Part \ref{Lem:G-check-0-Kernels-2}, $\Ker M(s<s')\cong \Ker M(s'<s'')$. This shows that $\Ker M(r<s) \cong \Ker M(r<s')$, and the inclusion shows that they coincide. This proves Part \ref{Lem:G-check-0-Kernels-3}.
\end{proof}

Recall that a module $M\in k\C\mod$ is said to be virtually trivial if $M(\varphi)$ is the zero homomorphism for any non-identity morphism in $\C$. 

\begin{Defi}\label{Def:Kernel-submodule}
Let $\P$ be a poset whose Hasse diagram $\HH_\P$ is a rooted (not necessarily finite) tree. Assume further that the root of $\P$ has a unique successor. Let $M\in k\P\mod$ be a module such that $[M]\in \GG^0(\P, k)$. Define a virtually trivial submodule $\K_M\subseteq M$ by 
\[\K_M(a) = \begin{cases}
				\Ker(M(a<y)) & a\quad\text{not maximal}, \quad\text{any}\; y>a\\
				0 & a\quad\text{maximal}
		\end{cases}
\]
\end{Defi}

By Lemma \ref{Lem:G-check-0-Kernels}, the submodule $\K_M$ is well defined for all modules $M$ whose isomorphism class is a class in $\GG^0(\P, k)$. For any poset $\P$, we say that a relation  $a<b$ is a \hadgesh{maximal relation}, if $a$ is minimal and $b$ is maximal in $\P$.

\begin{Defi}\label{Def:M'}
Let $\P$ be a poset, whose Hasse diagram is a disjoint union of rooted trees, and let $M\in k\P\mod$ be any module. Define a submodule $M'\subseteq M$  by 
\[M'(x) = \begin{cases}
					M(x) & \text{if $x$ is not maximal}\\
					M(y<x)(M(y)) & \text{if $x$ is maximal and $y$ is its predecessor}
			\end{cases}
\]
with the obvious action on morphisms.
\end{Defi}

\begin{Prop}\label{Prop:G-check-0-quotients}
Let $\P$ be a poset of composition length at least $3$, whose Hasse diagram $\HH_\P$ is a rooted (not necessarily finite) tree whose root  has a unique successor. Let $M\in k\P\mod$ be a module such that $[M]\in \GG^0(\P, k)$. Let $M'\subseteq M$ be as defined in Definition \ref{Def:M'}, and let $\widebar{M}\in k\P\mod$ denote the quotient module $M'/\K_{M'}$. Then $\widebar{M}$ is a locally constant module. In particular, $\widebar{M}$ is projective.
\end{Prop}
\begin{proof}
Fix a natural isomorphism $\eta\colon\partial_1^*M\to\partial_0^*M$ as before. Since $\HH_\P$ is a rooted tree whose root has a unique successor, the full subcategory of  $\G^\P_1$ whose objects are  all non-maximal relations   in $\P$ is connected, and $\G^\P_1$ has an isolated singleton component for each maximal relation in $\P$ (i.e. for each leaf of $\HH_\P$). Let $a<b$ be any non-maximal relation in $\P$. Assume first that $b$ is not maximal, 
so that $M$ and $M'$ coincide on $a$ and $b$ and the relation between them, and there is some  $v>b$ in $\P$.  The relation $a<b<v$ implies that
$M'(b<v)\circ\eta_{[a<b]} = \eta_{[b<v]}\circ M'(a<b)$.
It follows that 
\[\eta_{[a<b]}(\K_{M'}(a)) = \eta_{[a<b]}(\Ker M'(a<b)) = \Ker M'(b<v) = \K_{M'}(b),\]
where the first and third equality hold by Definition \ref{Def:Kernel-submodule}, and the second  since $\eta$ is a natural isomorphism.
Thus $\eta_{[a<b]}$ induces an isomorphism $\widebar{M}(a)\to\widebar{M}(b)$. Also, the projection $\pi_{-}\colon M'(-)\to\widebar{M}(-)$ is clearly a natural transformation. Next, consider the diagram,
\[\xymatrix{
\widebar{M}(a)\ar[drr] &&& M'(a)\ar@{->>}[rrr]^{\pi_a}\ar[dr]^{M'(a<v)}\ar[dl]_{M'(a<b)}\ar@{->>}[lll]_{\pi_a} &&& \widebar{M}(a)\ar@/^2pc/[ddlll]^{\widebar{M}(a<b)}\ar@{>->}[dll]_{\varphi^v_{a}}\\
&&M'(b)\ar@{->>}[dr]^{\pi_b}\ar[rr]^{M'(b<v)} && M'(v)\\
&&&\widebar{M}(b)\ar@{>->}[ur]^{\varphi^v_{b}}
}\]
All triangles commute by construction, except possibly the relation $\varphi^v_{b}\circ\widebar{M}(a<b) = \varphi^v_{a}$. But this relation holds if and only if it holds when precomposed with $\pi_a$, since $\pi_a$ is an epimorphism, and this follows easily by a diagram chase. Since $\varphi^v_{a}$ is a monomorphism, so is $\widebar{M}(a<b)$, and since its domain and codomain are isomorphic $\widebar{M}(a<b)$ is an isomorphism.

Next, assume that $b$ is maximal, and let $x$ be its predecessor. Then $M'(x<b)\colon M'(x)\to M'(b)$ is onto by definition, with kernel given by $\Ker M(x)$. Hence $\widebar{M}(x<b)\colon\widebar{M}(x) \to \widebar{M}(b)$ is an isomorphism. If $a\neq x$, then $a<x<b$, and by the previous argument $\widebar{M}(a<x)$ is an isomorphism, and hence so is the composition $\widebar{M}(a<b)$. This shows that $\widebar{M}$ is locally constant.

Finally, by Corollary \ref{Cor:Loc-const=constant}, $\widebar{M}$ is isomorphic to a finite sum of constant modules of the form $\underline{k}$, and is therefore projective.
\end{proof}

Theorem \ref{Th:0-cocycle-trees}, which we restate in more detail below, now follows as a consequence of Proposition \ref{Prop:G-check-0-quotients}. Recall that for a  category $\C$ and a subset $X\subseteq\obj(\C)$, we denote by $S_X$ the virtually trivial $k\C$-module that takes the value $k$ on each object $x\in X$ and zero everywhere else. If $X = \{x\}$ or $\obj(\C)$, we denote $S_X$ by $S_x$ and $S_\C$, respectively.

\begin{Thm}\label{Thm:0-cocycle-trees}
Let $\P$ be a finite poset of composition length at least $3$, whose Hasse diagram $\HH_\P$ is a rooted  tree. Assume that the  root  of $\HH_\P$ has a unique successor. Let $M\in k\P\mod$ be a module such that $[M]\in \GG^0(\P, k)$. Then with the notation of Proposition \ref{Prop:G-check-0-quotients}, 
\[M \cong  \bigoplus_n \underline{k}\oplus \bigoplus_d S_\P,\]
where  $d = \dim(\K_M(x))$ for any non-maximal $x\in \P$ and $n = \rk M(x<y)$ for any non-identity relation in $\P$. 
\end{Thm}
\begin{proof}
Write
\[M/M' = \bigoplus_{x\in\obj(\P)\atop\text{maximal}}\bigoplus_{n_x}S_x,\]
where $n_x=\dim(\coKer(M(y<x)))$ for $y$ the predecessor of $x$. When $x$ is maximal, each module $S_x$ is also projective, and hence $M/M'$, being a finite sum of projective modules, is itself projective. 
Consider the short exact sequence of $k\P$-modules
\[0\to M'\to M\to M/M'\to 0.\]
Since $M/M'$ is projective, the sequence splits and so $M\cong M'\oplus M/M'$.

Next, by Proposition  \ref{Prop:G-check-0-quotients}, $\widebar{M}$ is locally constant, and hence isomorphic to a finite direct sum of constant modules of the form $\underline{k}$. There is a short exact sequence of $k\P$-modules
\[0\to \K_{M'}\to M'\to \widebar{M}\to 0\]
and since $\widebar{M}$ is projective, the sequence is split. Now, $\K_{M'}$ is virtually trivial and  $\dim(\K_{M'}(x)) = \dim(\K_{M'}(y))$ for all non-maximal $x,y\in\P$. Furthermore, since the rank invariant of $M$ is constant by Proposition \ref{Prop:G-check-0-cocycles-rk}, $\dim((M/M')(z)) = \dim(\K_{M'}(x))$ for any maximal object $z$ and an arbitrary, non-maximal, object $x$ in $\P$. Thus 
\[M/M'\oplus \K_{M'} \cong \bigoplus_d S_\P,\]
where $d = \dim(\K_{M'}(x))$ for any non-maximal $x\in \P$. Thus we have 
\[M \cong M'\oplus M/M' \cong \widebar{M}\oplus \K_{M'} \oplus M/M' \cong \bigoplus_n \underline{k}\oplus \bigoplus_d S_\P,\]
where $n = \rk M(x<y)$ for any non-identity relation in $\P$. 
\end{proof}

%%%%%%%%%%%%%%%%%%%%%%%%%%%%%%%%%%%%
%%%%%%%%%%%%%%%%%%%%%%%%%%%%%%%%%%%%%

\subsection{\texorpdfstring{Higher-dimensional $\EE^*$ and $\GG^*$ groups}{Higher dimensional E and G groups}}
\label{SSec:Cohomology-chEG}

We observe here that for categories with finite composition length,  the higher-dimensional $\EE^*$ and $\GG^*$ cohomology is only interesting in a finite range of dimensions. Clearly,  if $N$ is the maximal composition length in $\C$, then $\E^\C_n=\G^\C_n=\emptyset$ for $n>N$. Hence in this case $\GG^n(\C, k) = \EE^n(\C,k) = 0$ for all $n>N$. However, a bit more can be said in both cases. Recall that a morphism $\varphi$ in a category $\C$ is said to be irreducible if for any expression of the form $\varphi = \alpha\circ\beta$, either $\alpha$ or $\beta$ is an identity morphism. 

\begin{Prop}\label{Prop:ch-G-mid-dim}
Let $\C$ be a finite direct category. Assume that  $N$ is the composition length of $\C$, for some positive integer $N$. Then $\EE^n(\C, k) = \GG^n(\C,k) = 0$ for all $n>N$. Furthermore,  the  map  
\[\iota^*\colon\GG^n(\C, k)\xto{\cong}\RH^n(\G^{\C,\emptyset},k) \cong H^n(|\C|, \Z)\]
induced by the inclusion $\iota\colon \G^{\C, \emptyset}_\bullet\subseteq \G^{\C}_\bullet$ is an isomorphism for all $n\geq n_0 \defeq \lfloor\frac{N}{2}\rfloor+1$ and an epimorphism for $1\le n\le \lfloor\frac{N}{2}\rfloor$. \label{Prop:ch-G-mid-dim-1}
\end{Prop}
\begin{proof}
First, it is clear that if $N$ is the maximal composition length in $\C$, then $\E^\C_n=\G^\C_n=\emptyset$ for $n>N$. The first statement follows.

By Theorem \ref{Thm:Higher-G}, the corresponding map 
\[\widetilde{\iota^*}\colon \widetilde{\GG}^n(\C, k)\to H^n(|\C|,\Z)\] is an epimorphism for all $n\geq 1$. Hence, by commutativity of Diagram (\ref{Diag:RH*-diag}),  $\iota^*$ is an epimorphism for all $n\geq 1$. Next, fix some $n>\lfloor N/2\rfloor$. Then any object $[u_1|u_2|\ldots|u_n] \in\G^\C_n$ is isolated. This is because the existence of any  morphism into or out of such an object  implies that $\C$ contains a chain of irreducible morphisms of length at least $2n > N$, which contradicts maximality of $N$. Hence, for such $n$,  one has $\G^\C_n\cong \G^{\C,\emptyset}_n$. Notice that  $n_0$ is the minimal integer for which $\G^\C_{n}$ is a discrete category for every $n\geq n_0$ and that the isomorphism $\Gr(k\G^\C_{n_0})\to \Gr(k\G^{\C,\emptyset}_{n_0})$ maps the subgroup of coboundaries in the domain to the corresponding subgroup of coboundaries in the codomain. The second statement follows. 
\end{proof}

%%%%%%%%%%%%%%%%%%%%%%%%%%%%%%%%%%%%%%%%%%%
%%%%%%%%%%%%%%%%%%%%%%%%%%%%%%%%%%%%%%%%%%%
%%%%%%%%%%%%%%%%%%%%%%%%%%%%%%%%%%%%%%%%%%%
%%%%%%%%%%%%%%%%%%%%%%%%%%%%%%%%%%%%%%%%%%%
%%%%%%%%%%%%%%%%%%%%%%%%%%%%%%%%%%%%%%%%%%%
%%%%%%%%%%%%%%%%%%%%%%%%%%%%%%%%%%%%%%%%%%%
%%%%%%%%%%%%%%%%%%%%%%%%%%%%%%%%%%%%%%%%%%%
%%%%%%%%%%%%%%%%%%%%%%%%%%%%%%%%%%%%%%%%%%%
%%%%%%%%%%%%%%%%%%%%%%%%%%%%%%%%%%%%%%%%%%%
%%%%%%%%%%%%%%%%%%%%%%%%%%%%%%%%%%%%%%%%%%%
%%%%%%%%%%%%%%%%%%%%%%%%%%%%%%%%%%%%%%%%%%%

\section{Some Computational Examples}
\label{Sec:Computational}

In this section we present cohomology computations for some simple families of examples.  Computing positive-dimensional cohomology is very challenging in general, but we are able to do it in some cases, particularly when  we can show that higher-dimensional cohomology vanishes. In some cases we consider  the representation cohomology generated by modules of bounded dimension $1$ (See Definition \ref{Def:MaxDim-d} and the following discussion).

%%%%%%%%%%%%%%%%%%%%%%%%%%%%%%%%%%%%%%%%%%%%%%%%
%%%%%%%%%%%%%%%%%%%%%%%%%%%%%%%%%%%%%%%%%%%%%%%%

\subsection{\texorpdfstring{The Poset $[n]$}{The Poset n}} 
\label{SSec:ch-no_ch-GE-[n]}
The category $[n]$ is, of course, the simplest example, but as we shall observe, it presents some interesting behaviour in dimension $0$, and is actually highly nontrivial for $\EE^*$ and $\GG^*$ theory, as our machine calculations demonstrate.   We start with $\widetilde{\EE}^*$ and $\widetilde{\GG}^*$. 

\begin{Prop}\label{Prop:GE-[n]}
For any $n\geq 2$ and a field $k$, 
\[\widetilde{\EE}^i([n], k) \cong \begin{cases} 	\Z & i=0\\
							0 & i>0\end{cases}
\quad\text{and}\quad 
\widetilde{\GG}^i([n],k) \cong \begin{cases}
					\Z & i=0\\
					\Z^n & i=1\\
					0 & i>1 \end{cases}.
\]
\end{Prop}
\begin{proof}
By Theorem \ref{Th:E0G0},  $\widetilde{\EE}^0([n],k)\cong\widetilde{\GG}^0([n],k)\cong\Gr(k)\cong \Z$  is generated by the class of the constant module $\underline{k}$, and  $\widetilde{\EE}^i([n],k)= 0$ for any $i>0$.  Also by Theorem \ref{Th:Higher-G}, $\widetilde{\GG}^i([n],k) \cong H^i(|[n]|,\Z) = 0$, for all $i>1$, since the geometric realisation of the nerve $|[n]|$ is the standard $n$-simplex and hence contractible. 

It remains to compute $\widetilde{\GG}^1([n], k)$.  We denote objects of $\widetilde{\G}^{[n]}_2$ by triples $abc$ where $a\le b\le c$. By the more precise formulation of Theorem \ref{Th:Higher-G} as Theorem \ref{Thm:Higher-G}, $\widetilde{\GG}^i([n],k)$ is generated by linear combinations of non-homogeneous singleton modules, i.e., by simple modules $S_{ij}$, where $i<j$. Thus we only have to examine cocycles  that are combinations of such modules. For $0\le i<j \le n$ and $0\le a\le b\le c\le n$, we have 
\begin{align*}
\delta^0(S_{ij})(abc) = \begin{cases}
						k & (b,c) = (i,j)\\
						0 & \text{otherwise}
				\end{cases}
				\quad\text{so}\quad
				\delta^0(S_{ij}) =\sum_{a\le i}S_{aij} \\
				\\
\delta^1(S_{ij})(abc) = \begin{cases}
						k & (a,c) = (i,j)\\
						0 & \text{otherwise}
				\end{cases}
				\quad\text{so}\quad
				\delta^1(S_{ij}) = \sum_{i\le b\le j }S_{ibj} \\	
				\\
\delta^2(S_{ij})(abc) = \begin{cases}
						k & (a,b) = (i,j)\\
						0 & \text{otherwise}
				\end{cases}
				\quad\text{so}\quad
				\delta^2(S_{ij}) = \sum_{j\le c}S_{ijc} \\
\end{align*}

Thus we have 
\[d^1(S_{ij}) = \sum_{a< i}S_{aij} - \sum_{i< b< j }S_{ibj} + \sum_{j< c}S_{ijc}.\]
Let $z = \sum_{i<j}x_{ij}S_{ij}$ be a $1$-cocycle. Then 
\[0 = d^1(z) = \sum_{a< i<j}x_{ij}S_{aij} - \sum_{i< b< j }x_{ij}S_{ibj} + \sum_{i<j< c}x_{ij}S_{ijc} =
\sum_{0\le a<b<c\le n}(x_{bc} - x_{ac} + x_{ab})S_{abc}.\]
It follows by linear independence of the modules $S_{ijk}$ that for each triple $a<b<c$, the relation  $x_{bc} - x_{ac} + x_{ab} =0$ must hold. Thus the expressions $x_{ab}S_{ab} + x_{bc}S_{bc} - x_{ac}S_{ac}$, where the relation above is satisfied, generate a subgroup of $1$-cocycles.  One thus obtains a system of homogeneous linear equations in the variables $x_{ij}$ for all $0\le i<j\le n$. By basic linear algebra and an inductive argument,   the variables $x_{i,i+k}$ where $k>1$ can be expressed as 
\[x_{i, i+k} = \sum_{j=i}^{k-1} x_{j, j+1},\]
while no further relations hold among the $x_{j,j+1}$.
Hence the null space of the $d^1$ matrix is $n$-dimensional. 

On the other hand, for $n>1$, one may verify by direct calculation that for any $0\le i<j\le n$, the virtual module $d^0(J_{i,j})$ contains no simple summands, where $J_{i,j}$ is the interval module determined by the interval $[i,j]$ . Hence the image of $d^0$, when restricted to the subgroup generated by these interval modules, intersects trivially with $\Ker(d^1)$. Similarly, $d^0$ restricted to simple modules $S_m$, for any $m\ge 0$, is a combination of non-singleton intervals. Therefore, $\Ima(d^0)$ intersects trivially with $\Ker(d^1)$, and it follows that $\widetilde{\GG}^1([n], k) \cong \Z^n$ as claimed. 
\end{proof}

We now turn to $\EE^*$ and $\GG^*$ cohomology of $[n]$. We give a full computation of $0$-dimensional cohomology for both theories, and then present low dimension machine computations of $\EE^*([n],k)_1$ and $\GG^*([n],k)_1$.

\begin{Prop}\label{Prop:hG0E0-[n]}
Let $\C$ be the poset  $[n]$, with objects the natural numbers $0,\ldots, n$, where $n\geq 2$, and the natural order relation. Then 
$\GG^0([n], k)\cong \Z^3$ and $\EE^0([n], k)\cong \Z^2$. Furthermore, the natural map $\EE^0([n],k)\to\GG^0([n],k)$ is split injective with cokernel generated by  $[S_{[n]}]$.
\end{Prop}
\begin{proof}
For non-negative integers $0\le a<b\le n$, let $J_{a,b}$ denote the interval module on $[n]$, with $J_{a,b}(i) = k$ for $a\le i\le b$ and $J_{a,b}(i) = 0$ otherwise. Let $S_a$ denote  the simple $k[n]$-module with value $k$ at $a$ and $0$ elsewhere. Similarly, denote objects in $\widetilde{\G}^{[n]}_1$ by $ab$, where $a<b$. For $a<b-1$, let $J_{a(a+1), (b-1)b}$ denote the module that is constant on the path
\[a(a+1)\to (a+1)(a+2)\to\cdots \to (b-1)b\] with value $k$ on objects, and $0$ everywhere else. Then the isomorphism classes of the modules $J_{a,b}$ form a basis for $\Gr(k[n])$. Clearly the modules $J_{0,n}$ and $\sum_{a=0}^n S_a$ are $0$-cocycles and hence represent two elements in $\GG^0([n],k)$. Consider the sum $[J_{0,n-1}] + [J_{1,n}] - [J_{1,n-1}]$. It is easy to calculate
\[J_{0,n-1}\partial_i = \begin{cases} 	J_{01,(n-2)(n-1)} & i=0\\
								J_{01,(n-1)n} & i=1, \end{cases}\quad
J_{1,n}\partial_i = \begin{cases} 		J_{01,(n-1)n} & i=0\\
								J_{12, (n-1)n} & i=1, \end{cases}\quad							
\]
and
\[J_{1,n-1}\partial_i = \begin{cases} 		J_{01,(n-2)(n-1)} & i=0\\
								J_{12, (n-1)n} & i=1. \end{cases}	
\]					
Consequently, $d^0([J_{0,n-1}] + [J_{1,n}] - [J_{1,n-1}]) = [J_{01,(n-1)n}] - [J_{01,(n-1)n}] =[0]$. The set  
\[\{J_{0,n}, \quad S_{[n]}\defeq \sum_{a=0}^n S_a, \quad T_{[n]}\defeq [J_{0,n-1}] + [J_{1,n}] - [J_{1,n-1}]\}\]
is clearly linearly independent and thus forms a free abelian subgroup of rank $3$ of $\GG^0([n], k)$. 

To show that the three elements above generate $\GG^0([n], k)$, it suffices to show that any element that does not include $J_{[0,n]}$ and $S_{[n]}$ in its decomposition is a multiple of $T_{[n]}$.  Let 
\[X = \sum_{0\le a < b< n}\alpha_{a,b}[J_{a,b}] + \sum_{0< a <  n}\beta_{a,n}[J_{a,n}]\]
be such an  element of $\Gr(k[n])$. Then  $\rk X$ is constant on all morphisms in a line-component of $[n]$. In particular, one has $\rk X(0<n-1) = \alpha_{0,n-1}$ and $\rk X(1<n) = \beta_{1,n}$. Thus $\alpha_{0,n-1} =\beta_{1,n}$. Also $\rk X(1<n-1) = \alpha_{0,n-1} + \beta_{1,n} + \alpha_{1,n-1}$. Hence $\alpha_{1,n-1} = -\alpha_{0,n-1}$. Write $A = \alpha_{0,n-1}$. Then $\rk X(a<b) = A$ for any interval $0\le a<b\le n$. Next, since  $[0,n-2]\subset [0,n-1]$, $A = \rk X(0<n-2) = A + \alpha_{0,n-2}$, so $\alpha_{0,n-2}=0$. By induction, $\alpha_{a,b} = 0$ for all $0\le a < b \le n-2$. Similarly, since $[2,n]\subset[1,n]$, we have $A = \rk X(2<n) = A + \beta_{2,n}$, so $\beta_{2,n} =0$. Again by induction, $\beta_{a,n} = 0$ for all $1<a<n$. Thus $X  = A\cdot T_{[n]}$, as claimed. 

Notice that $[n]$ has two line-components, one which consists only of the morphism $0<n$ and the other consisting of all the other morphisms. Hence if $X = A\cdot J_{0,n} + B\cdot S_{[n]} + C\cdot T_{[n]}$, then $\rk X$ takes the value $A$ on the first component and the constant value $A+B+C$ on the second. 

Finally, by naturality, $\EE^0([n], k)\le \GG^0([n],k)$. The module $J_{0,n}$ is clearly a cocycle in $\Gr(k\E^{[n]}_0)$, and it is easy to verify that so is $T_{[n]}$. On the other hand, $S_{[n]}$ is not a cocycle in $\Gr(k\E^{[n]}_0)$. For instance, for $n=2$, the category $\E^{[n]}_1$ is the interval $01< 02 < 12$, and $S_{[2]}\circ\partial_0$ is the module $k\xto{0}k\xto{=}k$, while $S_{[2]}\circ\partial_1$ is the module $k\xto{=}k\xto{0}k$, and these are clearly not isomorphic. A similar pattern holds for all $n\ge 2$.  This proves the last statement of the proposition.
\end{proof}

Computing higher dimensional cohomology, even for $[n]$, is highly nontrivial. To get an idea of what these groups look like for small examples, we used a simple Python code \cite{Code}. The code is designed to work on posets rather than general finite categories, and as such takes into account the subgroups of the respective Grothendieck groups generated by interval modules (all of which are indecomposable). These, of course, are always finitely generated for a finite poset, and are hence far from detecting the entire cohomology. We mark the resulting cohomology groups with a subscript $\I$. Example \ref{Ex:ChG*E*} lists the computations of  the groups $\EE^i([n], k)_\I$ and $\GG^i([n], k)_\I$ for small values of $n$. 
\begin{Ex}
\label{Ex:ChG*E*}
The following computations were done with the code available in \cite{Code}. We note that in general $\GG^i([n], k)_\I$ are $\EE^i([n], k)_\I$ are different from $\GG^i([n], k)_1$ and $\EE^i([n], k)_1$ because the categories $\G^{[n]}_i$ and $\E^{[n]}_i$ generally contain cycles for $i\ge 1$. 
\renewcommand{\arraystretch}{1.5}
\begin{center}
\begin{tabular}{||c|c|c|c|c||}
\hline
\multicolumn{5}{|c|}{$\GG^i([n], k)_\I$}\\
\hline
$i = $ & 0&1 & 2 & 3\\
\hline
$n=1$ & $\Z^2$ &  0 & 0 & 0\\
\hline
$n=2$ & $\Z^3$ & 0& 0 &0 \\
\hline
$n=3$ & $\Z^3$ & $\Z^6$ & 0 & 0\\
\hline
$n=4$ & $\Z^3$  & $\Z^{78}$ & 0 & 0\\
\hline
$n=5$ & $\Z^3$ & $\Z^{841}$ & $\Z^{44}$ & 0 \\
\hline
\end{tabular}
\begin{tabular}{||c|c|c|c|c|c||}
\hline
\multicolumn{6}{|c|}{$\EE^i([n], k)_\I$}\\
\hline
$i = $ & 0&1 & 2 & 3 &4 \\
\hline
$n=1$ & $\Z^2$ &  0 & 0 & 0& 0\\
\hline
$n=2$ & $\Z^2$ & $\Z$ & 0 & 0& 0\\
\hline
$n=3$ & $\Z^2$ & $\Z^7$ & 0 & 0& 0\\
\hline
$n=4$ & $\Z^2$  & $\Z^{28}$ & $\Z^{27}$ & 0& 0\\
\hline
$n=5$ & $\Z^2$ & $\Z^{103}$ & $\Z^{686}$ & $\Z^{75}$ & 0\\
\hline
\end{tabular}
\end{center}
\end{Ex}

The last observation in the proof of Proposition \ref{Prop:hG0E0-[n]} can be utilised further. Let $\C$ be an arbitrary finite category and let 
\[S_\C\defeq \bigoplus_{x\in \C} S_x,\]
namely $S_\C$ is the module that takes the value $k$ on each object and induces the $0$ homomorphism for each morphism. 

\begin{Cor}\label{Cor:S_C-in-G0}
For any finite category $\C$,  the element $[S_\C]$ is a nontrivial class in $\GG^0(\C, k)$.
\end{Cor}
\begin{proof}
Let $\varphi\colon a\to b$ be a morphism in $\C$, where $a\neq b$. Thus $[\varphi]$ is an object in $\G^\C_1$. Now, 
\[\delta^0(S_\C)[\varphi] =  S_\C(b) = k = S_\C(a) = \delta^1(S_C)[\varphi].\]
In $\G^\C_1$  a morphism from $[\varphi]$ to $[\psi\colon c\to d]$ is determined by a morphism $\alpha\colon b\to c$, and $\delta^0(S_\C)[\alpha] = S_\C[\psi\circ\alpha] = 0$, since $\psi\circ\alpha$ is not the identity morphism. Similarly,  $\delta^1(S_\C)[\alpha] = S_\C[\alpha\circ\varphi] = 0$ for the same reason. Hence 
\[\delta^0(S_\C)\cong\delta^1(S_\C) = S_{\G^\C_1}.\]
It follows that $[S_\C]$ is a cocycle and hence an element of $\GG^0(\C,k)$. 
\end{proof}

It is easy to verify by a simple diagram chase that the natural restriction map $\EE^0(\C, k)\to \GG^0(\C, k)$ is a monomorphism. The following corollary singles out a family of modules that are not in its image. 

\begin{Cor}\label{Cor:S_C-in-E0}
Let $\C$ be an arbitrary  small category with composition length at least $2$, and let $M\in k\G_0^\C$ be a module such that $[M]$ represents a nontrivial cohomology class in $\GG^0(\C, k)$. Assume further that there is a faithful functor $\iota\colon [2]\to \C$, whose image is contained in a single line-component of $\C$ and such that the restriction $\iota^*(M)$ is isomorphic to a direct sum of simple $k[2]$-modules. Then $[M]$ is not  in the image of the restriction from  $\EE^0(\C, k)$.
\end{Cor}
\begin{proof}
If $[M]\in \Gr(k\G^\C_0)$ is a cocycle, then,  by Lemma \ref{Lem:G-check-0-cocycles}, $M(x)\cong M(y)$ for all $x, y\in\C$. Thus, $\iota^*{M}$ is isomorphic to a direct sum of modules of the form  $S_{[2]}$. Considering $[M]$ as an element of $\Gr(k\E^\C_0)$, we see that   $\iota^*[M]$ is not a cocycle in $\Gr(k\E^{[2]}_0)$, as  demonstrated in the proof of Proposition \ref{Prop:hG0E0-[n]}. But then $[M]$ cannot be a cocycle in $\EE^0(\C, k)$ since $i^*$ is a cochain map. Thus, $[M]$ is not in the image of the restriction map  $\EE^0(\C, k)\to \GG^0(\C, k)$.
\end{proof}

%%%%%%%%%%%%%%%%%%%%%%%%%%%%%%%%%%%%%%%%%%%%%%%%
%%%%%%%%%%%%%%%%%%%%%%%%%%%%%%%%%%%%%%%%%%%%%%%%

\subsection{\texorpdfstring{The poset $\D_n$}{The poset Dn}}
Next we consider another very simple family of posets. Let $\D_n$ denote the poset with $n+2$ objects $0\le i\le n+1$, with non-identity morphisms $0\to 1$ and $1\to i$ for all $i>1$. Since $\D_n$ is generated by a tree, $\widetilde{\GG}^0(\D_n, k)\cong \Z$ for all $n$ by Theorem \ref{Th:E0G0}. By contrast we will  show below that $\GG^0(\D_n, k)$ is infinitely generated in general. However, genuine (as opposed to virtual) modules that are cocycles are of a particularly easy form to describe. 

\begin{Prop}
\label{Prop:dandelion-1}
Let $M\in k\D_n\mod$ be a module that 
represents a cohomology class $[M]\in\GG^0(\D_n, k)$. Then 
\[[M] = a[\underline{k}] + b[S_{\D_n}],\]
where $a, b$ are non-negative integers and $S_{\D_n}$ is as in Corollary \ref{Cor:S_C-in-G0}. 
\end{Prop}
\begin{proof}
As before, we will use the order notation to denote morphisms. Let $M\in k\D_n\mod$, $n\ge 2$,  be a module that represents a cohomology class $[M]\in\GG^0(\D_n, k)$. By Lemma \ref{Lem:G-check-0-cocycles}, $M(0)\cong M(i)$ for all $i>0$ and there are isomorphisms $\alpha_{0,1}\colon M(0)\to M(1)$ and $\alpha_{1,i}\colon M(1)\to M(i)$, such that 
\[M(1<i)\circ\alpha_{0,1} =  \alpha_{1,i}\circ M(0<1).\]
Then  for any $j\neq i$,
\[M(1<j) = \alpha_{1,j}\circ M(0<1)\circ\alpha_{0,1}^{-1}=\alpha_{1,j}\circ\alpha_{1,i}^{-1}\circ M(1<i)\circ\alpha_{0,1}\circ\alpha_{0,1}^{-1}=\alpha_{1,j}\circ\alpha_{1,i}^{-1}\circ M(1<i).\] 
Since $\alpha_{1,j}\circ\alpha_{1,i}^{-1}$ is an isomorphism, it follows that $\Ker M(1<i) = \Ker M(1<j)$ for all $2\le i, j \le n$. Thus $M(1) \cong K_1\oplus I_1$, where $K_1 = \Ker M(1<2)$, and  $I_1\subseteq M(1)$ is a complement subspace. For $2\le j\le n$, let $I_j = \alpha_{1,j}(I_1)$ and let $K_j = \alpha_{1,j}(K_1)$. Then $M(j)\cong K_j\oplus I_j$ and $M(1<j)|_{I_1}$ is an isomorphism from $I_1$ to $I_j$. 

Let $I_0 = \alpha_{0,1}^{-1}(I_1)$ and $K_0 = \alpha_{0,1}^{-1}(K_1)$. Let $\alpha_{0,1}^{-1}(x) = y\in K_0$, with $x\in K_1$. 
\[\alpha_{1,2}M(0<1)(y) = \alpha_{1,2}M(0<1)\alpha_{0,1}^{-1}(x) = M(1<2)(x) = 0.\]
Hence $K_0\le \Ker M(0<1)$, and since the dimensions are equal, $K_0= \Ker M(0<1)$. On the other hand, 
\[\alpha_{1,2}M(0<1)(I_0) = \alpha_{1,2}M(0<1) \alpha_{1,0}(I_1) = M(1<2)(I_1) = I_2.\] Hence $M(0<1)$ maps ${I_0}$  onto $I_1$, and by dimension comparison the restriction to $I_0$ is injective and hence an isomorphism $I_0\to I_1$.

Let $K$ denote the virtually trivial module with $K(i) = K_i$, and let $I$ denote the locally constant module with $I(j) = I_j$ and $I(j<k) = M(j<k)|_{I_j}$.   It follows that $M \cong K\oplus I$, where $K$ is a virtually trivial module of constant point dimension, and $I$ is locally constant. Thus 
\[I \cong a\cdot\underline{k},\quad\text{and}\quad K \cong b\cdot S_{\D_n},\]
where $a = \dim_k I_j$, $b = \dim_k K_j$  (both of which are the same for all $j$). The first isomorphism follows from Corollary \ref{Cor:Loc-const=constant}. The statement follows.
\end{proof}

We next show that $\GG^i(\D_n, k)\neq 0$ if and only if $i=0$ and describe $\GG^0(\D_n, k)$ for all $n\ge 2$. For any $N\in k\D_n\mod$, let $I(N)$ denote the module given by the direct sum
\small{\[\xymatrix{
N(2)  &\cdots  & N(n+1)\\
& N(1)\ar[ul]\ar[ur]\\
&N(0)\ar[u]} 
\xymatrix{
\\
\oplus
}
\xymatrix{
N(1) & \cdots & N(1)\\
& N(0)\ar[ul]\ar[ur]\\
&0\ar[u]} 
\xymatrix{
\\
\oplus
}
\xymatrix{
N(0) & \cdots & N(0)\\
& 0\ar[ul]\ar[ur]\\
&0\ar[u]}  \]
}
where the first summand is $N$ itself, and the morphism in the second  summand coincide with the corresponding ones in $N$.
\begin{Prop}
\label{Prop:dandelion-2}
For any $n\ge2$, the cohomology $\GG^i(\D_n, k)$ vanishes in positive dimensions and $\GG^0(\D_n, k)$ contains a  free abelian group generated by elements of the form $[I(N)]-[I(N_0)]$, where $N\in k\D_n\mod$ is an indecomposable module such that $N(0)\neq 0$ and $N_0$ is a module that coincides with $N$ everywhere except at $0$, where it takes the value $0$. 
\end{Prop}
\begin{proof}
The category $\G_1^{\D_n}$ has objects $01, 12,\ldots, 1(n+1)$ and $02, \ldots, 0(n+1)$ with non-identity morphisms $01\to 1j$ for all $2\le j\le n$. The category $\G_2^{\D_n}$ is discrete with objects $012, \ldots, 01(n+1)$. 
A module  $M\in \G_2^{\D_n}\mod$ is an assignment of a finite dimensional vector space to each object. Let $N\in k\G_1^{\D_n}$ be any module. Then 
\[N\circ\partial_i(01j) = \begin{cases}
					N(1j) & i=0\\
					N(0j) & i=1\\
					N(01) & i=2
				\end{cases}.\]
Let $V$ be a finite dimensional vector space and let $V_j$ denote the  $k\G_2^{\D_n}$-module that takes the value $V$ at $01j$ and $0$ elsewhere. Let $N_j\in k\G_1^{\D_n}\mod$  be the module that takes the value $V$ at $1j$ and $0$ everywhere else. Then $d^1[N_j] = [V_j]$, so $d^1$ is onto and $\GG^2(\D_n, k) = 0$.

Next, notice that $\Gr(k\G_1^{\D_n}) \cong \Gr(k\C_n)\times \Z^{n}$, where $\C_n\subseteq \G_1^{\D_n}$ is the full subcategory whose  objects are $01$ and $1j$ for $2\le j \le n+1$, and the factor $\Z^n$  corresponds to the Grothendieck group on the discrete subcategory consisting of the objects $02,\ldots, 0(n+1)$. Let $N\in k\C_n\mod$ be an indecomposable module, and let $(a_2,\ldots, a_{n+1})\in \Z^n$ be any element. Let $V = (k^{a_2},\ldots, k^{a_{n+1}})$. Then the $01j$-component of $d^1([N], [V])$ is $[N(1j)] + [N(01)] - a_j$. Thus, if we set $a_j = \dim_k N(1j) + \dim_kN(01)$ and let $V_N$ denote the corresponding module, then $(N, V_N)$ is a $d^1$-cocycle. Thus the elements of the form $([N], [V_N])$ form a basis for the subgroup $Z^1$ of $d^1$-cocycles. 

To prove that $\GG^1(\D_n, k) = 0$, it remains to show that every cocycle of the form $(N, V_N)$ is a coboundary. Let $M\in k\G^{\D_n}_0\mod$ denote the module defined by $M(0) = 0$, $M(1) = N(01)$, and $M(j) = N(1j)\oplus N(01)$  for all $2\le j\le n+1$. Let the morphism $M(1<j)$ be the map $N(01)\xto{N(01<1j)\top 0} N(1j)\oplus N(01)$. With the identification above it is now easy to observe that $d^0[M] = ([N], [V_N])$. This shows that $\GG^1(\D_n, k) = 0$.

Finally, let $N\in k\D_n\mod$ be any module. It is an easy observation that $\partial_0^*(I(N))$ is given by
\small{\[\xymatrix{
N(2)  &\cdots  & N(n+1)\\
& N(1)\ar[ul]\ar[ur]
} 
\xymatrix{
\\
\oplus
}
\xymatrix{
N(1) & \cdots & N(1)\\
& N(0)\ar[ul]\ar[ur]
} 
\xymatrix{
\\
\oplus
}
\xymatrix{
N(0) & \cdots & N(0)\\
& 0\ar[ul]\ar[ur]
},  \]
}
while $\partial_1^*(I(N))$ is 
\small{\[\xymatrix{
N(1)  &\cdots  & N(1)\\
& N(0)\ar[ul]\ar[ur]
} 
\xymatrix{
\\
\oplus
}
\xymatrix{
N(0) & \cdots & N(0)\\
& 0\ar[ul]\ar[ur]
} 
\]
}
Similarly, observe that $\partial_0^*(I(N_0))$ is given by
\small{\[\xymatrix{
N(2)  &\cdots  & N(n+1)\\
& N(1)\ar[ul]\ar[ur]
} 
\xymatrix{
\\
\oplus
}
\xymatrix{
N(1) & \cdots & N(1)\\
& N(0)\ar[ul]\ar[ur]
},  \]
}
while $\partial_1^*(I(N_0))$ is 
\small{\[\xymatrix{
N(1)  &\cdots  & N(1)\\
& N(0)\ar[ul]\ar[ur]
}. 
\]
}
Thus, $d^0([I(N)] - [I(N_0)])=0$, so each such virtual module is a $0$-cocycle. Furthermore, each such virtual module has $[N]$ as a summand. Hence the set 
\[\left\{[I(N)] - [I(N_0)]\in\Gr(k\D_n)\;|\; N\in k\D_n\mod\;\;\text{indecomposable}\right\}\]
is a set of linearly independent $0$-cocycles. This completes the proof.
\end{proof}

\begin{Rem}
The operation that takes a module $N$ to $I(N)$ can be performed on modules over  a much more general family of posets. In fact, if $\P$ is a poset whose Hasse diagram is a disjoint union of rooted trees, one can show that the set $\{[I(N)] -[I(N_0)]\}$, where $N$ runs over the set of representatives of isomorphism classes of all indecomposable $k\P$-modules, freely generates $\GG^0(k\P, k)$. A proof will appear in \cite{KLR2}.
\end{Rem}

Since the representation type of $\D_n$ is finite only for $n=2$, and infinite otherwise, the following corollary is immediate. 

\begin{Cor}\label{Cor:checkG-Dn}
The group $\GG^0(\D_n, k)$ is finitely generated if and only if $n\le 2$.
\end{Cor}

The  cochain map 
\[\Gr(k\G_\bullet)_\I\to \Gr(k\G_\bullet)\]
clearly induces an injection on $0$-dimensional cohomology. Proposition \ref{Prop:dandelion-2} allows an easy calculation of $\GG^0(\D_n, k)_\I$ for all $n$. Notice that since $\D_n$ is generated by a tree, indecomposable $k\D_n$-modules of bounded dimension $1$ are in bijective correspondence with intervals in $\D_n$. Hence $\GG^*(\D_n,k)_\I=\GG^*(\D_n, k)_1$.

\begin{Cor}\label{Cor:ch-G-0-dandelion}
For each $n\geq 2$, the group $\GG^0(\D_n, k)_\I$ is free abelian of rank $2^n+1$.
\end{Cor}
\begin{proof}
We apply Proposition \ref{Prop:dandelion-2} to all indecomposable modules $M\in k\D_n\mod$ that satisfy the condition $M(0)\neq 0$ and $\dim_kM(x) \le 1$ for all $x\in\D_n$. The module $S_0$ is the only simple module in this set, and in addition, for each choice of $k\geq 0$ leaves among the $n$, one has  an interval module on  $\D_k\subseteq \D_n$. Thus there are $\sum_{k=0}^n{n\choose k} = 2^n$ such submodules. In total, one obtains  $2^n +1$ linearly independent $0$-cocycles.

Next, consider interval modules  $Q\in k\D_n\mod$ such that $Q(0)=0$. There are $n+1$ simple modules $S_i$, for $i>0$, and for each $1\le m\le n$ there are $n\choose m$  interval modules, one for each choice of $m$ objects among $2, \ldots, n+1$.  Computing coboundaries, we see that 
\[\delta^0[S_1] = [S_{01}], \quad \delta^1[S_1] = \sum_{j=2}^{n+1}[S_{1j}],\quad \text{and for}\;\; i>1,\quad \delta^0[S_i] = [S_{1i}],\quad\text{and}\quad \delta^1[S_i] = 0.\]
If $N\in k\D_n\mod$ is an interval module with $N(j)= k$ for $j=1$ and for $j\in\{i_1, i_2,\ldots i_m\;|\; 2\le i_r\le n+1\}$, then
\[\delta^1[N] = \sum_{j=2}^{n+1}[S_{1j}]\quad\text{while}\quad \delta^0[N] = \widehat{N}, \]
where $\widehat{N}\in k\G^{\D_n}_1\mod$ is the module that takes the value $k$ on the object $01$ and on the objects $1i_r$, where $2\le i_r\le n+1$ are the objects on which $N$ obtains non-zero values, and the value $0$ elsewhere. 

Notice that $[S_{01}]$ appears only in the coboundary of $[S_1]$. Hence $[S_1]$ cannot appear as a summand in any $0$-cocycle. Furthermore, the modules of the form $\widehat{N}$ form a linearly independent set of interval modules, and it follows that no interval of the form $N$ can appear in any $0$-cocycle. This leaves only the classes $[S_i]$ for  $i>1$, and since the set $[S_{1,i}]$ is again linearly independent, it follows that no combination of these classes will  yield a cocycle. This shows that the set of linearly independent cocycles we found is a basis for $\GG^0(\D_n, k)_\I$, and completes the proof.
\end{proof}

We end this subsection with a sample machine calculation of $\EE^i(\D_n, k)_\I$ for low values of $n$.

\begin{center}
\begin{tabular}{||c|c|c||}
\hline
 & $\EE^0(\D_n,k)_\I$ & $\EE^1(\D_n,k)_\I$\\
\hline
$\D_2$  & $\Z^4$ & $\Z^6$\\
\hline
$\D_3$ & $\Z^8$& $\Z^{21}$ \\
\hline
$\D_4$ &  $\Z^{16}$ & $\Z^{68}$ \\
\hline
$\D_5$ & $\Z^{32}$ & $\Z^{215}$\\
\hline
$\D_6$ & $\Z^{64}$ & $\Z^{670}$\\
\hline
$\D_7$ & $\Z^{128}$ & $\Z^{2065}$\\
\hline
$\D_8$ & $\Z^{256}$ & $\Z^{6312}$\\
\hline
$\D_9$ & $\Z^{512}$ & $\Z^{19179}$\\
\hline
\end{tabular}
\end{center}
The natural map $\EE^0(\D_n,k)\to \GG^0(\D_n,k)$ is a monomorphism and by Corollary  \ref{Cor:S_C-in-E0} its image does not contain the cocycle $[S_{\D_n}]\in\GG^0(\D_n,k)$. On the other hand, it is not hard to see that the simple module $S_0$ does not give rise to a $0$-cocycle, because $I(S_0) = S_{\D_n}$, which is not in the image of the restriction, while each of the intervals of Corollary \ref{Cor:ch-G-0-dandelion} are also $0$-cocycles in $\EE^0(\D_n,k)$. Thus $\EE^0(\D_n,k)_1$ is a free abelian group of rank $2^n$ for all $n$.  Also, the rank  $\rho_n$ of $\EE^1(\D_n, k)_1$ seems to follow the formula
\[\rho_n = 3^{n} -2^{n} + n -1,\]
but we could not verify this formula analytically.

\subsection{Cyclic Quivers}\label{SSec:Cyclic}
By a cyclic quiver we mean any quiver where the underlying undirected graph is a cycle. By Example \ref{Ex:E0-ex}, if $\Q$ is any such quiver, then $\widetilde{\EE}^0(\PP(\Q), k)\cong\Gr(k[\Z])$. Hence $\widetilde{\EE}^0(\PP(\Q), k)$ is a free abelian group generated by the indecomposable $k[\Z]$-modules, which we proceed to describe. 

\begin{Prop}\label{Prop:Cyclic-E0}
Let $\Q$ be a cyclic quiver. Then for any field $k$ the group $\widetilde{\EE}^0(\PP(\Q), k)$ is a free abelian group with a basis given by pairs $(q(t), m)$, where $q(t)\in k[t]$ is an irreducible polynomial of positive degree with a non-zero constant term and $m>0$ is an integer. 
\end{Prop}
\begin{proof}
The group ring $k[\Z]$ is isomorphic to the ring of Laurent polynomials $k[t, t^{-1}]$, which is a principal ideal domain. The structure theorem for finitely generated modules over principal ideal domains implies that  any such module $M$ is isomorphic to a finite sum of cyclic modules. Namely, one has an isomorphism
\[M \cong k[t]/(p_1(t))\oplus  k[t]/(p_2(t))\oplus \cdots \oplus  k[t]/(p_n(t)),\]
where $p_i(t)|p_{i+1}(t)$ for each $1\le i\le n-1$, and $p_i(0)\neq 0$ for all $i$. The second condition is required because $t$ must act on $M$ as an  automorphism, and if $t|p_i(t)$, then $t$ will be a zero divisor in the corresponding summand, which is impossible.

Write each $p_j(t)$ for $1\le j\le n$ as a product 
\[p_j(t) = q_1(t)^{m_{j,1}}\cdot q_2(t)^{m_{j,2}}\cdots q_r(t)^{m_{j,r}},\]
where the $q_i(t)$ are irreducible with a non-zero constant term. Here we use the same list of irreducible polynomials for all $p_j(t)$, with the understanding that some of the powers may be $0$. Then 
\[k[t]/(p_j(t)) \cong k[t]/(q_1(t)^{m_{j,1}})\oplus  k[t]/(q_2(t)^{m_{j,2}})\oplus \cdots \oplus  k[t]/(q_r(t)^{m_{j,r}}),\]
and each one of these summands is an indecomposable $k[t, t^{-1}]$-module.

Thus $\widetilde{\EE}^0(\PP(\Q), k)$ is the free abelian group with generators indexed by pairs $(q(t), m)$ where $q(t)\in k[t]$ is an irreducible polynomial and $m>0$ is an integer, as claimed.
\end{proof}

\begin{Rem}
If $k$ is algebraically closed then the polynomials $q(t)$ in Proposition \ref{Prop:Cyclic-E0} are all linear. Hence in this case the generating set can be indexed by a pair $(\lambda, m)$, where $\lambda$ is the single root of the polynomial. In this case the indecomposable modules correspond to Jordan blocks of size $m$ and parameter $\lambda$. 
\end{Rem}

%%%%%%%%%%%%%%%%%%%%%%%%%%%%%%%%%%%%%%%%%%%
%%%%%%%%%%%%%%%%%%%%%%%%%%%%%%%%%%%%%%%%%%%
%%%%%%%%%%%%%%%%%%%%%%%%%%%%%%%%%%%%%%%%%%%
%%%%%%%%%%%%%%%%%%%%%%%%%%%%%%%%%%%%%%%%%%%
%%%%%%%%%%%%%%%%%%%%%%%%%%%%%%%%%%%%%%%%%%%
%%%%%%%%%%%%%%%%%%%%%%%%%%%%%%%%%%%%%%%%%%%
%%%%%%%%%%%%%%%%%%%%%%%%%%%%%%%%%%%%%%%%%%%
%%%%%%%%%%%%%%%%%%%%%%%%%%%%%%%%%%%%%%%%%%%
%%%%%%%%%%%%%%%%%%%%%%%%%%%%%%%%%%%%%%%%%%%
%%%%%%%%%%%%%%%%%%%%%%%%%%%%%%%%%%%%%%%%%%%
%%%%%%%%%%%%%%%%%%%%%%%%%%%%%%%%%%%%%%%%%%%

\end{document}